\documentclass[letterpaper,11pt]{amsart}

\usepackage{graphicx,color,epsfig}
\usepackage{amsfonts,amsmath,amssymb,amsthm}

\newtheorem{thm}{Theorem}[section]
\newtheorem{prop}[thm]{Proposition}
\newtheorem{lem}[thm]{Lemma}
\newtheorem{cor}[thm]{Corollary}

\newtheorem{asm}{Assumption}

\theoremstyle{remark}
\newtheorem{rem}[thm]{Remark}

\theoremstyle{definition}
\newtheorem{defn}{Definition}

\newcommand{\ra}{\rightarrow}

\newcommand{\Lra}{\Longrightarrow}

\newcommand{\R}{\mathbb R}     % For Real numbers
\newcommand{\Z}{\mathbb Z}     % For Integers
\renewcommand{\b}{\beta}

\renewcommand{\d}{\delta}

\newcommand{\e}{\varepsilon}
\newcommand{\del}{\partial}

\renewcommand{\l}{\lambda}
\renewcommand{\L}{\Lambda}

\newcommand{\s}{\sigma}

\renewcommand{\k}{\kappa}

\newcommand{\bigo}{\mathcal{O}}

\newcommand{\fl}[1]{\lfloor #1 \rfloor}  % Floor function
\newcommand{\cl}[1]{\lceil  #1 \rceil} 	 % Ceiling function
\newcommand{\ind}[1]{ \mathbf{1}_{ \{ #1 \} } } % Indicator functions
\newcommand{\indd}[1]{ \mathbf{1}\{ #1 \} }

\newcommand{\be}{\begin{equation}}
\newcommand{\ee}{\end{equation}}
\DeclareMathOperator{\Var}{Var}   \DeclareMathOperator{\Cov}{Cov}  
   
\def\Pv{\mathbf{P}}  \def\Ev{\mathbf{E}}   %another probability distinct from the rwre probabilities 
\newcommand{\T}{\mathbb{T}} 	%used for Path-valued processes related to hitting times. 
\newcommand{\tT}{\tilde{\T}}	%used for continuous version of hitting times path. 
\renewcommand{\S}{\mathbb{S}}	%used for path-valued processes of sums of exponentials. 
\newcommand{\U}{\mathbb{U}}	%used for path-valued processes of hitting times of ladder locations. 
\renewcommand{\H}{\mathcal{H}}	%function mapping point processes to measures of paths. 

\newcommand{\V}{\mathbb{V}}	

\newcommand{\tc}{C_0}		% Tail decay constant for beta
\newcommand{\lm}{\ell}		% Notation for the deterministic measure of a linear path. 

\def\one{{\bf 1}}
\def\cadlag{c\`adl\`ag}

\newcommand{\w}{\omega}              % Shortcut for \omega
\renewcommand{\P}{\mathbb{P}}        % Annealed Probability
\newcommand{\E}{\mathbb{E}}          % Annealed Expectation
\newcommand{\vp}{\mathrm{v}_P}       % Limiting velocity
\begin{document}

%opening
\title[Weak quenched limits for paths]{Weak weak quenched limits for
  the path-valued processes of hitting times and positions of a
  transient, one-dimensional random  walk in a random environment}
\author{Jonathon Peterson}
\address{Jonathon Peterson \\  Purdue University \\ Department of Mathematics \\ 150 N University Street \\ West Lafayette, IN  47907 \\ USA}
%\curraddr{Department of Mathematics \\ Cornell University \\ Malott Hall \\ Ithaca, NY 14850 \\ USA}
\email{peterson@math.purdue.edu}
\urladdr{http://www.math.purdue.edu/~peterson}
\thanks{J. Peterson was partially supported by National Science Foundation grant DMS-0802942.}

\author{Gennady Samorodnitsky}
\address{Gennady Samorodnitsky \\ Cornell University \\ School of Operations Research and Information Engineering \\ Ithaca, NY 14853 \\ USA}
\email{gs18@cornell.edu}
\urladdr{http://legacy.orie.cornell.edu/~gennady/}
\thanks{G. Samorodnitsky was partially supported by ARO grant
  W911NF-10-1-0289, NSF grant DMS-1005903 and NSA grant
H98230-11-1-0154  at Cornell University}

%\subjclass[2000]{Primary: ; Secondary: }
% 60K35 - interacting random processes
% 60K37 - Processes in random environments 
% 05C80 - random graphs
% 60K37 - processes in random environments
% 60J80 - branching processes (Galton-Watson, birth-death chains, ...)
% 82B26 - Phase transitions (general) 
% 60F05 - Central limit and other weak theorems 
% 60G55 - Point processes

\date{\today}

\begin{abstract}
In this article we continue the study of the quenched distributions of
transient, one-dimensional random walks in a random environment. In a
previous article we showed that while the quenched distributions of
the hitting times do not converge to any deterministic distribution,
they do have a weak weak limit in the sense that - viewed as random elements of the space of probability measures - they converge in distribution to a certain random probability measure (we refer to this as a weak weak limit because it is a weak limit in the weak topology). 
Here, we improve this result to the path-valued process of hitting
times. As a consequence, we are able to also prove a weak weak quenched
limit theorem for the path  of the random walk itself. 
\end{abstract}

\maketitle

\section{Introduction and Notation}

A random walk in a random environment (RWRE) is a very simple model
for random motion in a  non-homogeneous random medium. A
nearest-neighbor RWRE on $\Z$ may  be described as follows. 
%Let \Omega = [0,1]^\Z$ be given  the natural cylindrical $\s$-field. 
Elements of the set $\Omega = [0,1]^\Z$ are called \emph{environments}
since they can be used to  define the transition probabilities for a
Markov chain. That is, for any $\w =  \{\w_x \}_{x \in \Z} \in \Omega$
and any $z \in \Z$, let $X_n$ be  a Markov chain with law $P_\w^z$
given by $P_\w^z(X_0 = z) = 1$  and
\[
 P_\w^z( X_{n+1} = y \, | \, X_n = x) =
\begin{cases}
 \w_x & \text{if } y=x+1\\
 1-\w_x & \text{if } y=x-1\\
 0 & \text{otherwise}. 
\end{cases}
\]
%\[
% P_\w^z( X_{n+1} = y+1 \, | \, X_n = y) = \w_y = 1 -  P_\w^x( X_{n+1} = y-1 \, | \, X_n = y).
%\]

Let $\Omega$ be endowed with the natural cylindrical $\s$-field, and
let $P$ be a probability measure on $\Omega$. Then, if $\w$ is a
random environment with distribution $P$, then $P_\w^z$  is a random
probability measure and is  called the \emph{quenched} law of the
RWRE. By averaging over all  environments we obtain the
\emph{averaged} law of the RWRE 
\[
 \P^z(\cdot) = \int_\Omega P^z_\w(\cdot) P(d\w). 
\]
For ease of notation, the quenched and averaged laws of the RWRE
started at $z=0$ will be denoted  by $\P_\w$ and $\P$, respectively. 
Expectations with respect to $P$, $P_\w$ and $\P$ will be denoted by
$E_P$, $E_\w$ and $\E$,  respectively. 

Throughout this paper we will make the following assumptions on the
distribution $P$ on  environments. 
\begin{asm}\label{iidasm}
 The environments are i.i.d. That is, $\{\w_x\}_{x\in\Z}$ is an
 i.i.d.\ sequence of random variables under the measure $P$. 
%Moreover, we assume that $0<\w_0<1$ with probability 1. 
\end{asm}
\begin{asm}\label{tasm}
 The expectation $E_P[ \log \rho_0 ]$ is well defined and 
$E_P[ \log \rho_0 ]< 0$. Here $\rho_i = \rho_i(\w) =
 \frac{1-\w_i}{\w_i}$, for all $i\in\Z$.  
\end{asm}
\begin{asm}\label{kasm}
 The distribution of $\log \rho_0$ is non-lattice under $P$, and there exists a $\kappa > 0$ such that $E_P[ \rho_0^\kappa ]= 1$ and $E_P[ \rho_0^\k \log \rho_0 ] < \infty$. 
\end{asm}
From Solomon's seminal paper on RWRE \cite{sRWRE}, it is well known that Assumptions \ref{iidasm} and \ref{tasm} imply that the RWRE is transient to $+\infty$; that is, $\P(\lim_{n\ra\infty} X_n = \infty) = 1$.
Moreover, Solomon showed that there exists a law of large numbers in the sense that there exists a constant $\vp$ such that $\lim_{n\ra\infty} X_n/n = \vp$, $\P$-a.s. 
Solomon also showed that the limiting velocity $\vp$ is non-zero only if $E_P[ \rho_0 ] < 1$, which is equivalent to $\k> 1$ when Assumption \ref{kasm} is in effect as well. 
Assumption \ref{kasm} was used by Kesten, Kozlov, and Spitzer in their analysis of the averaged limiting distributions for transient one-dimensional RWRE \cite{kksStable}. The parameter $\k$ in Assumption \ref{kasm} determines the magnitude of centering and scaling as well as the type of distribution obtained in the limit.
Define the hitting times of the RWRE by 
\[
 T_x = \inf\{ n\geq 0\, : \, X_n = x \}, \quad x \in \Z,
\]
and for $\k\in(0,2)$ define the properly centered and scaled versions of the hitting times and location of the RWRE by 
\be\label{acs}
 \mathfrak{t}_n = 
\begin{cases}
 \frac{T_n}{n^{1/\k}} & \k \in (0,1) \\
 \frac{T_n - n D(n)}{n} & \k = 1 \\
 \frac{T_n - n/\vp}{n^{1/\k}} & \k \in (1,2)
\end{cases}
\quad\text{and}\quad 
 \mathfrak{z}_n = 
\begin{cases}
 \frac{X_n}{n^{\k}} & \k \in (0,1) \\
 \frac{X_n - \d(n)}{n/(A \log n)^2} & \k = 1 \\
 \frac{X_n - n \vp}{\vp^{1+1/\k} n^{1/\k}} & \k \in (1,2), 
\end{cases}
\ee
where in the case $\k=1$, $A>0$ is a certain constant, and $D(n)$ and
$\d(n)$ are certain functions  satisfying $D(n) \sim A \log n$ and
$\d(n) \sim n/(A \log n)$,  respectively. 
Also, let $L_{\k,b}$ denote the distribution function of a totally
skewed to the right stable random variable with index $\k \in (0,2)$,
scaling parameter $b>0$, and  zero shift; see \cite{stStable}. The
following averaged limiting  distribution for RWRE was first proved in
\cite{kksStable}.  

\begin{thm}\label{averagedlimlaw}
 Let Assumptions \ref{iidasm} - \ref{kasm} hold, and let $\k\in(0,2)$. 
%Then, for $\mathfrak{t}_n$ and $\mathfrak{z}_n$ defined as in \eqref{acs}, there exists
Then, there exists a constant $b>0$ such that for any $x\in\R$, 
\[
 \lim_{n\ra\infty} \P(\mathfrak{t}_n \leq x) = L_{\k,b}(x), \
 x\in\R\,,
\]
and 
\[
 \lim_{n\ra\infty} \P(\mathfrak{z}_n \leq x) = 
\begin{cases}
 1 - L_{\k,b}(x^{-1/\k}) & \text{for} \ x>0 \ \text{if} \ \k \in (0,1) \\
 1 - L_{\k,b}(-x) & \text{for} \ x\in\R \ \text{if} \ \k \in [1,2). 
\end{cases}
%\begin{cases}
% 1 - L_{\k,b}(x^{-1/\k}) & \k \in (0,1) \\
% 1 - L_{\k,b}(-A^2 x) & \k = 1 \\
% 1 - L_{\k,b}(-x/\vp^{1+1/\k}) & \k \in (1,2).
%\end{cases}
\]
\end{thm}

\begin{rem}
 The cases $\k=2$ and $\k > 2$ were also considered in \cite{kksStable}, but since our main results are for $\k \in (0,2)$ we will limit our focus to these cases. We note, however, that when $\k\geq 2$ the averaged limiting distributions for the hitting times and the location of the RWRE are Gaussian. 
%When $\k>2$ there is a Gaussian limiting distribution under the standard CLT re-scaling (linear centering and $\sqrt{n}$ scaling). When $\k=2$ the limiting distribution is again Gaussian and the centerings are linear, but the scaling has logarithmic corrections to $\sqrt{n}$. 
\end{rem}

It is important to note that the limiting distributions in Theorem \ref{averagedlimlaw} are for the averaged measure $\P$. However, for certain applications the quenched measure $P_\w$ may be more applicable (e.g., for repeated experiments in a fixed non-homogeneous medium), and one naturally wonders if there is a quenched analog of Theorem \ref{averagedlimlaw}. 
Unfortunately, it was shown in \cite{pzSL1} and \cite{p1LSL2} that
there is no such strong quenched limiting distribution. 
That is, for almost every fixed environment $\w$, there is no
centering and scaling (or even environment-dependent centering and
scaling) for  which the hitting times or location of the RWRE converge
in distribution under $P_\w$. 

The negative results of \cite{pzSL1} and \cite{p1LSL2} were recently
clarified by showing that quenched limiting distributions do exist in
a weak sense (\cite{psWQLTn, dgWQLTn, estzWQLTn}). Let
$\mathcal{M}_1(\R)$ be the  space of probability measures on $\R$
equipped with the  topology of convergence in distribution. Then,
since the environment $\w$ is a random variable, the quenched
distribution $\mu_{n,\w} = P_\w(\mathfrak{t}_n \in \cdot)$ is an
$\mathcal{M}_1(\R)$-valued function of that random variable. This
function can be easily shown to be measurable, hence $\mu_{n,\w}$ is
itself a random variable, namely  a $\mathcal{M}_1(\R)$-valued random
variable.  
%It was shown in \cite{psWQLTn} that the distribution of $\mu_{n,\w}$ converges as $n\ra\infty$ to a distribution on $\mathcal{M}_1(\R)$ that can be explicitly described in terms of a non-homogeneous Poisson point processes. 
It was shown in \cite{psWQLTn}  that there exists a family of
$\mathcal{M}_1(\R)$-valued random variables $(\pi_{\l,\k})$ such that
$\mu_{n,\w} \Lra  \pi_{\l,\k}$ for some $\l>0$, where $\Lra$
denotes weak  convergence of $\mathcal{M}_1(\R)$-valued random variables\footnote{Throughout the paper, if $(Z_n),Z$ are random variables in some space $\Psi$, then $Z_n \Lra Z$ will denote weak convergence (i.e., convergence in distribution) of $\Psi$-valued random variables.}. 
We will refer to such limits as \emph{weak weak quenched limits} since the quenched distribution converges weakly with respect to the weak topology on $\mathcal{M}_1(\R)$. 
Similar results were obtained independently in \cite{dgWQLTn} and
\cite{estzWQLTn}.  

In \cite{psWQLTn}, this weak weak quenched limiting distribution for
the hitting times  was also used to obtain a result on the quenched
distribution of the location of the RWRE. It was shown that 
\be\label{WQLXn1pt}
 P_\w( \mathfrak{z}_n \leq x ) \Lra 
\begin{cases}
 \pi_{\l,\k} [ x^{-1/\k}, \infty ) & \text{for} \ x>0 \ \text{if} \ \k \in (0,1) \\
 \pi_{\l,\k} [-x,\infty) & \text{for} \ x\in\R \ \text{if} \ \k \in [1,2), 
\end{cases}
\ee
and here $\Lra$ denotes weak convergence of $\R$-valued random variables. 
Note that this is a weaker statement than the quenched limit that was obtained for the hitting times. 
%Instead of a weak limit for the quenched probability distribution, this only gives weak limits for all the one-dimensional projections (which is not enough to prove a weak limit for the random probability measure $P_\w(\mathfrak{z}_n \in \cdot)$). 
Unfortunately, weak convergence of all one-dimensional projections of
a random probability measure is not enough to specify the weak limit
of the random probability measure.  
For example, suppose that $\k \in (0,1)$. 
%I think that the coordinatewise transformation works if $\k \in [1,2)$
If $\s_{\l,\k}$ is the transformation of the random probability
measure $\pi_{\l,\k}$ defined by letting $\s_{\l,\k}(-\infty,x]$ equal
the right hand side of \eqref{WQLXn1pt}, one is tempted to guess that
$P_\w(\mathfrak{z}_n \in \cdot ) \Lra \s_{\l,\k}$ in the sense of weak
convergence of random probability measures.  However, it can be seen from our results below that this is not true (see Corollary \ref{c:fixed.time}).  

\subsection{Main Results}

The original goal of the current paper was to obtain a full weak limit for the random probability
measure $P_\w( \mathfrak{z}_n \in \cdot)$.
However, it turned out to be
necessary to obtain a weak limit for not just the 
quenched distribution of the hitting $T_n$ but also for the quenched
distribution of the path process of the sequence of hitting times.  
%In order to obtain a full weak limit for the random probability
%measure $P_\w( \mathfrak{z}_n \in \cdot)$, it turns out to be
%necessary to obtain a weak limit for not just the 
%quenched distribution of the hitting $T_n$ but also for the quenched
%distribution of the path process of the sequence of hitting times.  
This result, in turn leads to not only a weak limit for the quenched
distribution of $X_n$ but also to the weak limit of the quenched
distribution of the entire path of the RWRE, as we will see in the
sequel. 

To begin, let $D_\infty$ be the space of \cadlag\  functions
(continuous from the right with left limits) on $[0,\infty)$.  We will
equip $D_\infty$ with the $M_1$-Skorohod metric $d_\infty^{M_1}$
(instead of the more standard and slightly stronger $J_1$-Skorohod metric
$d_\infty^{J_1}$; the definitions of the Skorohod metrics are given in
Section \ref{MTgeneral}).  
Let $\mathcal{M}_1(D_\infty)$ be the space of probability measures on
$D_\infty$ equipped with the topology of weak convergence induced by
the $M_1$-metric $d_\infty^{M_1}$ on $D_\infty$. Since $(D_\infty,
d_\infty^{M_1})$ is a Polish space, this topology is equivalent to
topology induced by the Prohorov metric $\rho^{M_1}$ (see Section
\ref{MTgeneral} for a precise definition).  

For any realization of the random walk and $\e>0$, let $\T_\e \in D_\infty$ be defined by 
\be\label{Tedef}
 \T_\e(t) = 
\begin{cases}
 \e^{1/\k} T_{t/\e} & \k \in (0,1) \\
 \e (T_{t/\e} - t/\e D(1/\e)) & \k = 1 \\
 \e^{1/\k} (T_{t/\e} - t/(\e\vp)) & \k \in (1,2)
\end{cases}
\ee
(here and in the sequel  we define hitting times of non-integer points by $T_x = T_{\fl{x}}$.) 
In the case $\k=1$ the function $D$ is the function in
\eqref{acs} extended to all $x>0$; we will define it explicitly in
Section \ref{CompareWithExp}. It is easy to see that, for each
environment $\w$ and any $\e>0$, $\T_\e$ is a 
well-defined $D_\infty$-valued random variable; we denote by $m_{\e,\w}$ the
quenched law of $\T_\e$ on $D_\infty$.  This law is a measurable
function of the environment, hence a $\mathcal{M}_1(D_\infty)$-valued
random variable. We wish to show that this random variable converges
weakly as $\e\to 0$. In order to identify the limit we need to
introduce additional notation.  

Let $\mathcal{M}_p((0,\infty] \times [0,\infty))$ be the space of
Radon point processes on $(0,\infty] \times [0,\infty)$. These are
point processes assigning finite mass to $[\e,\infty]\times[0,T]$ for
any $\e>0$ and $T<\infty$. The topology of vague convergence on this
space is metrizable, and converts $\mathcal{M}_p((0,\infty] \times
[0,\infty))$ into a complete separable metric space; see
\cite[Proposition 3.17]{rEVRVPP}. We denote by $\mathcal{M}_p^f((0,\infty]
\times [0,\infty))$ the subset of $\mathcal{M}_p((0,\infty] \times
[0,\infty))$ of point processes that do not put any mass on points
with infinite first coordinate. 
Let $\vec{\tau} = \{ \tau_i \}_{i\geq 1}$ be a sequence of
i.i.d. standard exponential random variables. For a point
process $\zeta = \sum_{i\geq 1} \d_{(x_i,t_i)}\in
\mathcal{M}_p^f((0,\infty] \times [0,\infty))$  
and $\d>0$ we define a stochastic process (random path)
$W_\d(\zeta,\vec\tau)$  with sample paths in $D_\infty$ by 
\[
 W_\d(\zeta,\vec{\tau})(t) = \sum_{i\geq 1} x_i \tau_i \ind{x_i > \d, t_i \leq t}. 
\]
We also let 
\be\label{Wdef}
 W(\zeta,\vec{\tau})(t) = 
\begin{cases}
 \sum_{i\geq 1} x_i \tau_i \ind{t_i \leq t} & \text{if the sum is finite}\\
 0 & \text{otherwise}. 
\end{cases}
\ee
\begin{rem}
The notation $W_\d(\zeta,\vec{\tau})$ and $W(\zeta,\vec{\tau})$ is
somewhat misleading since the actual definitions depend on the
(measurable) ordering chosen for the points of $\zeta$. Since
$\vec{\tau}$ is an i.i.d.\ sequence of random variables, the choice of
ordering will not affect the laws of $W_\d(\zeta,\vec{\tau})$
and $W(\zeta,\vec{\tau})$, and we are only concerned with the laws of
these processes. 

It is clear that $\lim_{\d \ra 0} W_\d(\zeta,\vec\tau) =
W(\zeta,\vec\tau)$ in $D_\infty$ for every choice of
$\vec{\tau}$ for 
which $W(\zeta,\vec{\tau})(t)<\infty$ for each $t<\infty$. We will
impose assumptions on the point processes $\zeta$ such that
this holds with probability one. 
\end{rem}

For any point process $\zeta$ such that, with probability 1,
$W(\zeta,\vec{\tau})(t)<\infty$ for each $t<\infty$, 
the definitions of $W_\d(\zeta,\vec\tau)$ and $W(\zeta,\vec\tau)$
induce in natural way probability measures on $D_\infty$.   
Define functions $\H_\d, \H: \mathcal{M}_p((0,\infty]\times
[0,\infty)) \ra \mathcal{M}_1(D_\infty)$ by  
\be\label{Hdef}
 \H_\d(\zeta)(\cdot) = \Pv_\tau( W_\d(\zeta,\vec\tau) \in \cdot ),
 \quad \text{and} \quad \H(\zeta)(\cdot) = \Pv_\tau( W(\zeta,\vec\tau)
 \in \cdot ),
\ee
when $\zeta\in \mathcal{M}_p^f((0,\infty] \times [0,\infty))$ and (in the case of $\H(\zeta)$) when $W(\zeta,\vec{\tau})(t)<\infty$ for each $t<\infty$ with probability 1.
Otherwise we define $\H_\d(\zeta)$ or $\H(\zeta)$, respectively, to be the Dirac point mass at the zero process in $D_\infty$. 
Here $\Pv_\tau$ is the
distribution of the i.i.d.\ sequence of the standard exponential
random variables $\vec\tau = \{\tau_i\}_{i\geq 1}$.  

Before stating our theorem we need one last bit of notation. The cases
$\k \in[1,2)$ require a centering term in the limit. Thus, for any $m
\in \R$ let $\ell(m) \in \mathcal{M}_1(D_\infty)$ be the Dirac point
mass measure that is concentrated on the linear path $t\mapsto m
t$. If $X$ is a $D_\infty$-valued random variable with distribution $\mu \in
\mathcal{M}_1(D_\infty)$, then $\mu * \lm(-m)$ is the distribution of
the path $\{ t \mapsto X(t) - mt \}$.  
\begin{thm}\label{WQLTn1}
Let $m_{\e}(\cdot)=m_{\e,\w}(\cdot) = P_\w(\T_\e \in \cdot)$ be the quenched distribution
of the path $\T_\e$. For $0<\k<2$ let $\l=C_0\k/\bar\nu$, where $C_0$
and $\bar\nu$ are given, respectively, by \eqref{btail} and \eqref{e:nubar}
below. Let $N_{\l,\k}$ be a Poisson
point process on $(0,\infty] \times [0,\infty)$ whose intensity
measure puts no mass on infinite points, and is given by 
$\l x^{-\k-1} \, dx \, dt$ on finite points. 
Then, $m_\e \Lra \mu_{\l,\k}$ as $\e \ra 0$ where 
\begin{equation}\label{mulkdef}
 \mu_{\l,\k} = 
\begin{cases}
 \H(N_{\l,\k}) & \text{ if } \k \in (0,1) \\
 \lim_{\d\ra 0} \H_\d(N_{\l,1}) * \lm(- \l \log(1/\d) ) & \text{ if } \k = 1 \\
 \lim_{\d\ra 0} \H_\d(N_{\l,\k}) * \lm\left(-\l \d^{-\k+1}/(\k-1)\right) & \text{ if } \k \in (1,2).
\end{cases}
\end{equation}
%The following weak limits hold for $m_\e$ as $\e\ra 0$.   
%\begin{enumerate}
%\item If $\k\in(0,1)$, then $m_{\e} \Lra \H(N_{\l,\k})$ as $\e \ra 0$. 
%\item If $\k=1$, then
%\[
% m_{\e} \Lra \lim_{\d\ra 0} \H_\d(N_{\l,1}) * \lm(- \l \log(1/\d) ). 
%\]
%\item If $\k\in(1,2)$, then 
%\[
% m_{\e} \Lra \lim_{\d\ra 0} \H_\d(N_{\l,\k}) * \lm\left(-\l \d^{-\k+1}/(\k-1)\right). 
%\]
%\end{enumerate}
\end{thm}

\begin{rem}
The limits in the definition of $\mu_{\l,\k}$ in \eqref{mulkdef} when $\k \in [1,2)$
%The limits as $\delta\to 0$ in parts (2) and (3) of Theorem \ref{WQLTn1} 
are weak limits in $\mathcal{M}_1\bigl(
(D_\infty,d_\infty^{M_1})\bigr)$. In fact, we will see in the sequel
that these limits exist even as a.s.\ limits in $\mathcal{M}_1\bigl(
(D_\infty,d_\infty^{M_1})\bigr)$. Furthermore, 
the defintion of $\mu_{\l,\k}$ as $\H(N_{\l,\k})$ in the case $\k \in (0,1)$ is valid by the 
%part (1) of the theorem includes a 
well-known fact that for each
$t<\infty$, 
$W(N_{\l,\k},\vec{\tau})(t)<\infty$ with
$\Pv_\tau$-probability 1 for almost every realization of the Poisson
point process $N_{\l,\k}$. 
\end{rem}

As mentioned above, we will prove the existence of a weak limit for the quenched
distribution of the entire path of the RWRE.  To this end, we
define a centered and scaled path of the random walk $\chi_\e \in D_\infty$ by  
\be\label{chidef}
\chi_\e(t) 
= \begin{cases}
   \e^\k X_{\fl{t/\e}} & \k \in (0,1) \\
   \frac{1}{\e \d(1/\e)^2} \left( X_{\fl{t/\e}} - t \d(1/\e) \right) & \k = 1 \\
   \vp^{-1-1/\k} \e^{1/\k} \left(X_{\fl{t/\e}} - t \vp / \e \right) & \k \in (1,2), 
  \end{cases}
\ee
where in the case $\k=1$, $\d(x)$ is a function that satisfies $\d(x)
D(\d(x)) = x + o(1)$ as $x\ra\infty$. Here $D$ is the same function as
in \eqref{Tedef}. Note that, since $D(x) \sim A
\log x$,  this implies that $\d(x) \sim x/(A \log x)$ so that the
scaling factor in the definition of $\chi_\e$ when $\kappa = 1$ is
asymptotic to $\e ( A \log(1/\e) )^2$ as $\e\to 0$.  Let $p_{\e,\w} = P_\w(\chi_\e
\in \cdot)$ be the quenched law of $\chi_\e$ on $D_\infty$. It is a $\mathcal{M}_1(D_\infty)$-valued
random variable defined on $\Omega$. 

The weak limits of $p_{\e,\w}$ will be obtained by comparing the paths
of the location of the RWRE $\chi_\e$ to appropriately transformed 
paths of the hitting times $\T_\e$. To this end, we define
two transformations of paths.  Let $D_{u,\uparrow}^+ \subset D_\infty$
consist of functions that are (weakly) monotone increasing, with
$x(0)\geq 0$ and $\lim_{t\ra\infty} x(t) = \infty$.  
%It is easy to see that $D_\infty^+$ is a closed subset, and so $D_\infty^+$ is a Polish space under the inherited metric. 
Define the time-space inversion function $\mathfrak{I}:D_{u,\uparrow}^+ \ra
D_{u,\uparrow}^+$ by  
\be \label{Idef}
 \mathfrak{I}x(t) = \sup\{ s\geq 0: x(s) \leq t \}, \, t\geq 0, \ x\in  D_{u,\uparrow}^+ . 
\ee
%\begin{thm}\label{WQLXn1}
% If $\k \in (0,1)$, then there exists a $\l>0$ such that $p_{\e,\w} \overset{P}{\Lra}_\e \H(N_{\l,\k}) \circ \mathfrak{I}^{-1}$. 
%\end{thm}
Also, define the spatial reflection function $\mathfrak{R}: D_\infty
\ra D_\infty$ by $\mathfrak{R}x(t) = -x(t)$, $t\geq 0, \ x\in
D_\infty$. 
%\begin{thm}\label{WQLXn2}
% Let $\k \in [1,2)$, and let $\mu_{\l,\k}$ be the random probability measure on $D_\infty$ given by Theorem \ref{WQLTn1} so that $m_{\e,\w} \overset{P}{\Lra} \mu_{\l,\k}$ for some $\l>0$. Then, $p_{\e, \w} \overset{P}{\Lra} \mu_{\l,\k} \circ \mathfrak{R}^{-1}$. 
%\end{thm}
%\begin{rem}
%Theorems \ref{WQLXn1} and \ref{WQLXn2} will follow from stronger statements that compare the paths of the random walk $\chi_\e$ with transformations of the paths of the hitting time $\T_\e$. When $\k \in (0,1)$, $\chi_\e$ is close to the path $\mathfrak{I} \T_{\e^\k}$. When $\k \in (1,2)$, $\chi_\e$ is close to the path $-\T_{\e/\vp} = \mathfrak{R}(\T_{\e/\vp})$, and when $\k = 1$, $\chi_\e$ is close to the path $-\T_{1/\d(1/\e)} = \mathfrak{R}( \T_{1/\d(1/\e)} )$. 
%\end{rem}
\begin{thm}\label{WQLXn} \ ({\bf a}) \ The following coupling results
  hold. 
%The paths process of the random walk $\chi_\e$ may be compared to the path process of the hitting times $\T_\e$ in the following manner. 
 \begin{enumerate}
  \item If $\k \in (0,1)$, then for any $s<\infty$
\[
\lim_{\e\ra 0} \P\left( \sup_{t\leq s} | \chi_\e(t) - \mathfrak{I}\T_{\e^\k}(t) | \geq \eta \right) = 0, \quad \forall \eta> 0. 
\]
%\[
% \lim_{\e\ra 0} \P( d_\infty^{J_1}(\chi_\e, \mathfrak{I}\T_{\e^\k}) \geq \eta ) = 0, \quad \forall \eta>0. 
%\]
%$\lim_{\e\ra 0} \P( d_\infty^{J_1}(\chi_\e, \T_{\e^\k}) \geq \eta ) = 0$ for any $\eta > 0$. 
%This implies that 
%$p_{\e,\w} \overset{P}{\Lra}_\e \H(N_{\l,\k}) \circ \mathfrak{I}^{-1}$.
  \item If $\k = 1$, then 
\[
 \lim_{\e\ra 0} \P( d_\infty^{M_1}(\chi_\e, -\T_{1/\d(1/\e)}) \geq \eta ) = 0, \quad \forall \eta>0. 
\]
 \item If $\k \in (1,2)$, then 
\[
 \lim_{\e\ra 0} \P( d_\infty^{M_1}(\chi_\e, -\T_{\e/\vp}) \geq \eta ) = 0, \quad \forall \eta>0. 
\]
 \end{enumerate}

\medskip

({\bf b}) \ Let $p_{\e}(\cdot)=p_{\e,\w}(\cdot) = P_\w( \chi_\e \in \cdot)$ be the quenched
distribution of the path $\chi_\e$, and let $\mu_{\l,\k}$ be 
the random probability distribution on paths defined in \eqref{mulkdef}.
%the limiting element of $\mathcal{M}_1(D_\infty)$ given in Theorem \ref{WQLTn1}. 
Then $p_{\e} \Lra \mu_{\l,\k} \circ \mathfrak{I}^{-1}$ as $\e \ra 0$
if $\k \in (0,1)$ and $p_{\e} \Lra \mu_{\l,\k} \circ
\mathfrak{R}^{-1}$ as $\e \ra 0$ if $\k \in [1,2)$,  weakly in
$\mathcal{M}_1(D_\infty)$. 
%\[
% p_{\e,\w} \overset{P}{\Lra}
%\begin{cases}
% \H(N_{\l,\k}) \circ \mathfrak{I}^{-1} & \k \in (0,1) \\
% \lim_{\d\ra 0} \left( \H_\d(N_{\l,\k}) * \lm( - \l \log(1/\d) ) \right) \circ \mathfrak{R}^{-1} & \k = 1 \\
% \lim_{\d\ra 0} \left( \H_\d(N_{\l,\k}) * \lm( - \l \d^{-\k+1}/(\k-1) ) \right) \circ \mathfrak{R}^{-1} & \k = 1. 
%\end{cases}
%\]
\end{thm}
\begin{rem}
Note that the nature of conversion from time to space in the limiting
random probability measure in $\mathcal{M}_1(D_\infty)$ is very
different in the absence of centering term ($\k \in (0,1)$) from the
case when there is a centering term ($\k \in [1,2)$). When $\k \in
[1,2)$ the conversion is accomplished by multiplying a random path
distributed according to the limiting (random) measure by  -1.  This
is, of course, very different from the switching the time and space
axes required when $\k \in (0,1)$. 
\end{rem}

Observe that, for any $0\leq t<\infty$, 
%the map $\Phi_t:\, \mu\to \mu\{ x\in D_\infty:\, x(t)\in \cdot\}$ from $\mathcal{M}_1(D_\infty)$ to $\mathcal{M}_1(\R)$ is continuous at every $\mu\in \mathcal{M}_1(D_\infty)$ concentrated on paths continuous at$t$. 
the map $\Phi_t:\mathcal{M}_1(D_\infty) \ra \mathcal{M}_1(\R)$ defined by 
\be\label{e:projection}
\Phi_t(\mu)(A) = \mu\left( \{ x \in D_\infty: \, x(t) \in A \} \right), \quad \text{for any Borel } A\subset \R,
\ee
is continuous at every $\mu\in \mathcal{M}_1(D_\infty)$ concentrated on paths continuous at $t$.
Since the limiting probability measures on
$\mathcal{M}_1(D_\infty)$ obtained in Theorem \ref{WQLXn} is
concentrated on $\mu$ with this property, the continuous mapping theorem
immediately implies the following weak weak convergence for the
distributions of the location of the random walk at fixed times.
\begin{cor}\label{c:fixed.time}
For $0\leq t<\infty$ let $p_{\e;t}=p_{\e,\w;t} = P_\w( \chi_\e(t) \in
\cdot)\in \mathcal{M}_1(\R)$ be the quenched
distribution of $\chi_\e(t)$, and let $\nu_{\l,\k}$ be the
limiting element of $\mathcal{M}_1(D_\infty)$ given in Theorem
\ref{WQLXn}. Then $p_{\e;t}\Lra \Phi_t( \nu_{\l,\k} )$ weakly in
$\mathcal{M}_1(\R)$. 
\end{cor}
%\begin{rem}
% Corollary \ref{c:fixed.time} was the original motivation for the current paper, but it turned out to be necessary to prove the weak weak quenched limits for the path processes first. 
%\end{rem}

Theorems \ref{WQLTn1} and \ref{WQLXn} imply the following corollaries
on the convergence of $\T_\e$ and $\chi_\e$ under the averaged measure
$\P$.  
%As a corollary of Theorem \ref{WQLTn1}, we obtain a limiting distribution for the paths of hitting times under the averaged measure.
\begin{cor}\label{AveragedTn}
For any $\k \in (0,2)$, the hitting time paths $\T_\e$, viewed as
random elements of $(D_\infty, d_\infty^{M_1})$, converge weakly under
the averaged measure $\P$. Furthermore, 
\begin{enumerate}
 \item if $\k \in (0,1)$, the limit is a $\k$-stable L\'evy
   subordinator; 
 \item If $\k \in [1,2)$, the limit is a $\k$-stable L\'evy process
   that is totally skewed to the right. Moreover, if $\k \in (1,2)$, 
   the limit is a strictly stable L\'evy process.  
\end{enumerate}
%For any $\k<2$, $\T_\e$ converges in distribution as $\e \ra 0$ under the averaged measure $\P$ to a $\k$-stable L\'evy process $Z_\k$. That is, for any $\mathcal{B}(D_\infty, d_\infty^{M_1})$-measurable set $A$ such that $\Pv(Z_\k \in \del A) = 0$, we have $\P(\T_\e \in A) \ra \Pv( Z_\k \in A)$ as $\e \ra 0$. 
%%%%That is, for any bounded, $d_\infty^{M_1}$-continuous function $f$ on $D_\infty$, we have that $\E[ f(\T_\e) ] \ra \Ev[ f(Z_\k) ]$ as $\e \ra 0$. 
\end{cor}
Since a stable subordinator is a strictly increasing process, its
inverse has continuous sample paths. Correspondingly, we can
strengthen the topology on the space $D_\infty$ when considering weak
convergence of the paths of the location of the RWRE 
under the average probability measure $\P$ in the case
$\k\in(0,1)$. To this end, let $(D_\infty, d_\infty^U)$ denote the
space $D_\infty$ equipped with the topology of uniform convergence on
compact sets. This space is not separable, but Theorem 6.6 in \cite
{bCOPM} allows us to conclude weak convergence on the
ball-$\sigma$-field in that space, the so-called weak$^\circ$
convergence. Moving from the $M_1$ topology to the $J_1$ topology, on
the other hand, does not cause any difficulties. 
\begin{cor}\label{AveragedXn} \ 

\begin{enumerate}
 \item If $\k \in (0,1)$ then the paths $\chi_\e$, viewed as
random elements of $(D_\infty, d_\infty^{J_1})$, converge weakly under
the averaged measure $\P$ to the inverse of a $\k$-stable
subordinator.  Furthermore, $\chi_\e$ as
random elements of $(D_\infty, d_\infty^{U})$ equipped with the
ball-$\sigma$-field, we have weak$^\circ$
convergence to the same limit.
 \item If $\k \in [1,2)$, then the paths $\chi_\e$, viewed as
random elements of $(D_\infty, d_\infty^{M_1})$, converge weakly to a
$\k$-stable L\'evy process that is totally skewed to the
left. Moreover, if $\k \in (1,2)$, then the limit is a strictly stable
L\'evy process.  
\end{enumerate}
\end{cor}
\begin{rem}
 The statement of Corollary \ref{AveragedXn} in the case $\k \in
 (0,1)$  appeared in Remark 2.5 in \cite{eszStable}. To the
 best of our knowledge the other statements in Corollaries
 \ref{AveragedTn} and \ref{AveragedXn} are new.  
\end{rem}

Part (2) of Corollary \ref{AveragedXn} is an immediate consequence of
the corresponding part of Corollary \ref{AveragedTn}, the coupling
results in parts (2) and (3) of Theorem \ref{WQLXn} and Theorem 3.1 in
\cite {bCOPM}. Further, \cite[Corollary
13.6.4]{wSPL} says that the operator $\mathfrak{I}$ from the subset
$D_{u,\uparrow\uparrow}^{+}\subset D_{u,\uparrow}^+$ of strictly increasing,
non-negative, unbounded paths endowed with the $d_\infty^{M_1}$ metric
to $D_{u,\uparrow}^+$ endowed with the $d_\infty^U$ metric, is
continuous. Since, in the case $0<\k<1$, a $\k$-stable subordinator is
in $D_{u,\uparrow\uparrow}^{+}$ with probability one, the continuous mapping
theorem shows that part (1) of Corollary \ref{AveragedXn} also follows
from the corresponding part of Corollary \ref{AveragedTn}. 

%In the case $\k \in (0,1)$ one must use the fact that the operator $\mathfrak{I}$ is continuous on the subset $D_\infty^{++}\subset D_\infty^+$ of strictly increasing, non-negative, unbounded paths. 
%, since the inverse of a $\k$-stable subordinator is continuous with probability one

The proof of Corollary \ref{AveragedTn} is also rather
straightforward, but, because it introduces certain key ideas and
notation used later in the paper, we present the proof here.  
%The proof of Corollary \ref{AveragedXn} is essentially the same and is therefore ommitted. 
\begin{proof}[Proof of Corollary \ref{AveragedTn}]
Let $N_{\l,\k}$ be the Poisson point process on $(0,\infty]\times [0,\infty)$ 
defined in Theorem \ref{WQLTn1}, and let $\vec\tau = \{\tau_i\}_{i\geq
  1}$ be an i.i.d. sequence of standard exponential random variables;
we assume that $N_{\l,\k}$ and $\vec\tau$ are defined on two different
probability spaces, with the corresponding probability measures $\Pv$ and $\Pv_\tau$.
On the product probability space we define 
\be\label{Zlkdef}
 Z_{\l,\k}(t) = 
\begin{cases}
 W(N_{\l,\k},\vec\tau)(t) & \k \in (0,1) \\
 \lim_{\d \ra 0} W_\d(N_{\l,\k}, \vec\tau)(t) - \l t \log(1/\d) & \k = 1 \\
 \lim_{\d \ra 0} W_\d(N_{\l,\k}, \vec\tau)(t) - \l t \d^{1-\k}/(\k-1)
 & \k \in (1,2),  
\end{cases}
\ee
$t\geq 0$. The definition is understood as a.s. convergence in
$(D_\infty, d_\infty^{M_1})$ on the product probability space. 
%In the case $\k\in(0,1)$, if $N_{\l,\k} = \sum_{i\geq 1}
%\d_{(x_i,t_i)}$ then, $\Pv$-a.s., $\sum_{i\geq 1} x_i \ind{t_i \leq
%t} < \infty$ for all $t<\infty$ so that the sum in the definition of
%$W(N_{\l,\k},\vec\tau)(t)$ is finite, $\Pv_\tau$-a.s.
This convergence takes place by the proposition in Section 2 of \cite{kRPWOD2},
and it is standard to see that $Z_{\l,\k}$ is a $\k$-stable L\'evy process with the required properties of Corollary \ref{AveragedTn}. 
%\begin{align*}
%& W_{\d_k}(N_{\l,\k}, \vec\tau)(t) - \frac{\l t}{\k-1} {\d_k}^{1-\k}\\
%&= W_{\d_1}(N_{\l,\k}, \vec\tau)(t) - \frac{\l t}{\k-1} {\d_1}^{1-\k}
% + \sum_{i=2}^k \left(  W_{\d_i}(N_{\l,\k}, \vec\tau)(t) -  W_{\d_{i-1}}(N_{\l,\k}, \vec\tau)(t) - \frac{\l t}{\k-1} \left( {\d_i}^{1-\k} - {\d_{i-1}}^{1-\k} \right) \right).
%\end{align*}
In order to show that the averaged distribution of $\T_\e$ converges
to the distribution of $Z_{\l,\k}$ under the product probability
measure $\Pv\times \Pv_\tau$, it is enough to 
show that $\P(\T_\e \in A) \ra \Pv \times \Pv_\tau( Z_{\l,\k} \in A)$
as $\e \ra 0$ for all cylindrical sets $A\subset
D_\infty$ such that $\Pv\times \Pv_\tau(Z_{\l,\k} \in \del A) =
0$. (Recall that the Borel $\sigma$-field under all the Skorohod
topologies coincides with the
cylindrical $\sigma$-field; see Theorem 11.5.2 in \cite {wSPL}.) Let
$A$ be such a set. 
Recall that $\Pv \times \Pv_{\tau}(Z_{\l,\k} \in A ) = \Ev[
\mu_{\l,\k}(A) ]$, where $\mu_{\l,\k}$ is defined in Theorem \ref{WQLTn1}. 
%Let $\mu_{\l,\k}$ be the limiting distribution of $m_{\e,\w}$ given in Theorem \ref{WQLTn1}. Note that $\mu_{\l,\k}$ is the distribution of $Z_{\l,\k}$ with respect to $\Pv_\tau$ for $N_{\l,\k}$ fixed. 
By Fubini's theorem, $\mu_{\l,\k}(\del A) = 0$ almost surely. 
Also, the evaluation mapping mapping $\mu \mapsto \mu(A)$ on
$\mathcal{M}_1(D_\infty)$ is continuous on the set of measures $\{\mu
\in \mathcal{M}_1(D_\infty) : \, \mu(\del A) = 0 \}$. Since the random
measure $\mu_{\l,\k}$ is in this set with probability one, and since
Theorem \ref{WQLTn1} implies that $m_{\e,\w} \Lra \mu_{\l,\k}$, then
the mapping theorem implies that $m_{\e,\w}(A)$ converges in
distribution to $m_{\l,\k}(A)$. Since these random variables are
between 0 and 1, this implies that 
\[
 \lim_{\e\ra 0} \P( \T_\e \in A ) = \lim_{\e\ra 0} E_P[ m_{\e,\w}(A) ]
 = \Ev[ \mu_{\l,\k}(A) ] = \Pv \times \Pv_\tau (Z_{\l,\k} \in A).  
\]
\end{proof}

%\begin{proof}[Proof of Corollary \ref{AveragedXn}]
% We will give the proof only in the case $\k \in (0,1)$. The proof in the case $\k \in[1,2)$ is similar. 
%
%Corollary 13.6.2 in \cite{wSPL} implies that the operator $\mathfrak{I}$ is continuous on the subset $D_\infty^{++} \subset D_\infty^+$ of strictly increasing, non-negative, unbounded paths. Since $\Pv \times \Pv_\tau( Z_{\l,\k} \in D_\infty^{++} ) = 1$, Corollary \ref{AveragedTn} and the mapping theorem imply that under the averaged measure $\P$, $\mathfrak{I} \T_\e$ converges in distribution on the space $(D_\infty, d^{M_1})$ to the inverse of a $\k$-stable subordinator.
%Next, the coupling of $\chi_\e$ and $\mathfrak{I} \T_{\e^\k}$ in Theorem \ref{WQLXn} implies that $\chi_\e$ also converges in distribution to the inverse of a $\k$-stable subordinator.  Finally, one notes that since the inverse of a $\kappa$-stable subordinator is continuous with probability 1, then the convergence in distribution of $\chi_\e$ can be improved from the Skorohod $M_1$-topology to the stronger uniform topology on $D_\infty$.  
%\end{proof}

The limiting random probability measure $\mu_{\l,\k}$ is a $\k$-stable
random element of $\mathcal{M}_1(D_\infty)$ under convolutions. That is, the
convolution of two independent copies of this random probability measure is 
(after re-scaling and shifting) a random probability measure with the same law.  
This can be seen in the same way as the stability of the limiting random probability
measures on $\R$ was checked in \cite{psWQLTn}. Stability of random
probability measures on $D_\infty$ does not seen to have been
investigated before, but a systematic description of infinitely
divisible (in particular, stable) random probability measures on $\R$
was given in \cite{stRPD}; we recall these notions in Section
\ref{StableRPD}. The latter paper introduced also a notion
of $\mathcal{M}_1(\R)$-valued L\'evy process. If we recall the
maps $\Phi_t$,  $0\leq t<\infty$, defined in \eqref{e:projection},
then we can define a (measurable) map $\Phi$ from $\mathcal{M}_1(D_\infty)$ to
$D_\infty(\mathcal{M}_1(\R))$ by setting $\Phi(\mu)$ to be the
measure-valued path $ \{ \Phi_t(\mu), \, t\geq 0 \}$. 

One would expect that a version of Theorem \ref{WQLTn1} would give us 
a convergence to a $\mathcal{M}_1(\R)$-valued L\'evy process as
well. The following corollary gives such convergence, but only in the
sense of convergence of finite dimensional distributions.

%\be\label{e:projection}
% \Phi(\mu) = \{ \Phi_t(\mu), \, t\geq 0 \}, \quad \Phi_t(\mu) = \mu \circ \Pi_t^{-1},
%\ee
%where $\Pi_t$ is the natural evaluation projection on paths that was defined above Corollary \ref{c:fixed.time}.

% $t\mapsto \Phi_t(\mu)$, where $\Phi_t(\mu) = \mu \circ \Pi_t^{-1}$ and $\Pi_t$ is the natural evaluation projection on paths that was defined above Corollary \ref{c:fixed.time}.
%$\Phi_t(\mu) \in \mathcal{M}_1(\R)$ is defined by 
%\[
% \Phi_t(\mu)(A) = \mu\left( \{ x \in D_\infty: \, x(t) \in A \} \right), \quad \text{for any Borel } A\subset \R. 
%\]
\begin{cor}\label{PathMeasureCor}
Let $\mu_{\l,\k}$ and $m_{\e,}$ be the random probability measures
on $D_\infty$ given in Theorem \ref{WQLTn1}, $0<\k<2$  (so that
$m_{\e} \Lra \mu_{\l,\k}$).  
Then $\Phi(m_{\e})$ converges weakly to $\Phi(\mu_{\l,\k})$ in the sense of finite dimensional distributions.  
Moreover, for any $\k \in (0,2)$, $\Phi(\mu_{\l,\k})$ is a stable L\'evy process on
$\mathcal{M}_1(\R)$. It is a strictly stable
L\'evy process if $\k\neq 1$.  
%%%%% MORE PROPERTIES OF $\Phi(\mu_{\l,\k})$. 
%Also, $\Phi(\mu_{\l,\k})$ is a stable random variable on the space $D_\infty(\mathcal{M}_1(\R))$ in the sense that 
%$\{ \Phi(\mu_{\l,\k})(t) :\, t\geq s \}$ has the same distribution as $\{ \Phi(\mu_{\l,\k})(s)*\Phi(\mu'_{\l,\k})(t) :\, t\geq 0 \}$ where $\mu'_{\l,\k}$ is an independent copy of $\mu_{\l,\k}$, 
%\begin{itemize}
% \item $\Phi(\mu_{\l,\k})$ has stationary increments: For any $s\geq 0$, 
%\[
% \{ \Phi(\mu_{\l,\k})(t) :\, t\geq s \} \overset{\text{Law}}{=} \{ \phi(\mu_{\l,\k})(s)*\phi(\mu'_{\l,\k})(t) :\, t\geq 0 \},
%\]
%where $\mu'_{\l,\k}$ is an independent copy of $\mu_{\l,\k}$.
%\end{itemize}
%if $\mu_{\l,\k}^{(1)},\mu_{\l,\k}^{(2)},\ldots$ are i.i.d.\ copies of $\mu_{\l,\k}$ then 
%\[
% \Phi(\mu_{\l,\k}^{(1)})(t) * \Phi(\mu_{\l,\k}^{(2)})(t) * \cdots * \Phi(\mu_{\l,\k}^{(k)})(t) \overset{\text{Law}}{=} \Phi(\mu_{\l,\k})(k t) \overset{\text{Law}}{=} \Phi(\mu_{\l,\k})(t)(\cdot /k^{1/\k}). 
%\]
\end{cor}
%\begin{rem}
% The notion of infinitely divisible (in particular, stable) random probability measures on $\R$
% was studied systematically in \cite{stRPD}, and we will recall these notions in Section \ref{StableRPD}. In particular, we will recall the definition of $\mathcal{M}_1(\R)$-valued L\'evy processes that was given in \cite{stRPD}.
%\end{rem} 
\begin{rem}
 One would like to improve the finite dimensional distribution convergence in Corollary \ref{PathMeasureCor} to a full convergence in distribution of $\mathcal{M}_1(\R)$-valued path processes. Such a statement seems would require setting a topology on the space $D_\infty(\mathcal{M}_1(\R))$ of measure-valued path processes. Choosing an appropriate topology seems to be a difficult task as neither the Skorohod $J_1$-topology nor a natural definition of the Skorohod $M_1$-topology appear to be sufficient. This is complicated by the fact that the mapping $\Phi: \mathcal{M_1}(D_\infty) \ra D_\infty(\mathcal{M}_1(\R))$ is not continuous in these topologies (even on the support of the limiting measure $\mu_{\l,\k}$).
 These issues are discussed further in Section \ref{StableRPD}. 
\end{rem}

\section{Random Environment}\label{Notation}
It will be important for us to identify sections of the environment
that contribute the most to the distribution of the hitting times. To
this end, we define the ladder locations $\nu_k = \nu_k(\w)$ of the environment by  
\be\label{nudef}
 \nu_0 = 0, \quad\text{and}\quad \nu_k = \inf \left\{ j > \nu_{k-1} :
   \, \prod_{i=\nu_{k-1}}^{j-1} \rho_i < 1 \right\} \text{ for } k\geq
 1.  
\ee
(The ladder locations are those locations where the \emph{potential}
of the environment introduced by Sinai \cite{sRRWRE} reaches a new
minimum to the right of the origin.) 
Occasionally we will denote $\nu_1$ by $\nu$ instead for compactness. 
Since the environment is i.i.d., the sections of the environment
$\{\w_x : \nu_k \leq x < \nu_{k+1} \}$ between ladder locations are
also i.i.d. However, the environment immediately to the left of
$\nu_0 = 0$ is different from the environment immediately to the left of
$\nu_k$ for any $k\geq 1$ since 
$\prod_{j=i}^{\nu_k-1} \rho_j < 1$ for any $k\geq 1$ and $0\leq i < \nu_k$ but it can happen that $\prod_{j=i}^{-1} \rho_j \geq 1$ for some $i < 0$.  
%$\rho_{\nu_k -1} < 1$ for any $k\geq 1$ but it can happen that $\rho_{-1} > 1$. 
Thus, the environment is not stationary
under shifts of the environment by the ladder locations. To resolve
this complication we define a new measure $Q$ on environments by 
\[
 Q(\cdot) = P(\cdot \, | \mathcal{R} ), \qquad\text{where } \mathcal{R} = \left\{ \w: \,  \prod_{j=i}^{-1} \rho_j < 1, \, \forall i \leq -1 \right\}.
\]
It is important to note that the definition of the measure $Q$ only
affects the environment to the left of the origin. Therefore, the
blocks of the environment $\{ \w_x : \, \nu_k \leq x < \nu_{k+1} \}$
between the ladder locations are i.i.d.\ and have the same
distribution under both $P$ and $Q$. For instance 
\begin{equation} \label{e:nubar}
\bar\nu := E_P \nu_1 = E_Q \nu_1 .
\end{equation}
The measure $Q$ is also stationary under shifts of the environment by the ladder locations in the sense that 
$\w$ has the same distribution as $\theta^{\nu_k(\w)}\w$ under $Q$ and
$\nu_1(\theta^{\nu_k(\w)}\w) = \nu_{k+1}(\w) - \nu_k(\w)$. 
%$\w$ and $\theta^{\nu_1(\w)} \w$ have the same distribution under $Q$. 
Therefore, if we let $\b_i = \b_i(\w) = E_\w[T_{\nu_i} - T_{\nu_{i-1}}] $ for any $i\geq 1$, it follows that $\{\b_i\}_{i\geq 1}$ is stationary under the measure $Q$. 
The following tail asymptotics of the $\b_i$ were derived in \cite{pzSL1} and will be crucial throughout this paper.  There exists a constant $C_0 > 0$ such that 
\be\label{btail}
 Q(\b_1 > x) = Q( E_\w T_\nu > x ) \sim C_0 x^{-\k}, \qquad \text{as } x \ra\infty. 
\ee
%The tail decay asymptotics of $\b_1$ also imply tail behavior of certain expectations which we state now for later use
%\begin{lem}
% \begin{enumerate}
%  \item If $\k \in (0,1)$ then $E_Q[ \b_1 \ind{\b_1 \leq x} ] \sim \frac{\tc \k}{1-\k} x^{1-\k} $ as $x \ra \infty$. 
% \item If $\k = 1$ then $E_Q[ \b_1 \ind{\b_1 \leq x} ] \sim \tc \log x$ as $x \ra \infty$. 
% \item If $\k \in (0,2)$ then $E_Q[ \b_1^2 \ind{\b_1 \leq x} ] \sim \frac{\tc \k}{2-\k} x^{2-\k} $ as $x \ra \infty$. 
% \end{enumerate}
%\end{lem}

We conclude this section with a simple lemma that will be of use later in the paper. 
\begin{lem}
 Let $\bar\b = E_Q[ \b_1 ]$. If $\k>1$, then  $\bar\b = \bar\nu/ \vp$. 
\end{lem}
\begin{proof}
 First, note that the sequence $\{ E_\w[ T_i - T_{i-1} ] \}_{i\geq 1}$
 is ergodic under the measure $P$ (since it represents the shifts of a
 fixed function of an i.i.d., hence ergodic, sequence). 
Therefore, Birkhoff's Ergodic Theorem implies that 
\be\label{EwTnn}
 \lim_{n\ra\infty} \frac{E_\w T_n}{n} = E_P[E_\w T_1] = \E T_1, \quad P\text{-a.s.}
\ee
Since the measure $Q$ is defined by conditioning $P$ on an event of
positive probability, we see that this holds $Q$-a.s.\ as well. 
Moreover, if $\k > 1$, then the limiting velocity $\vp = 1/\E T_1 > 0$
(see \cite{sRWRE} or \cite{zRWRE} for a reference).  
%where the last equality only makes sense when $\k>1$ and is used in the proof of the limiting velocity for RWRE (see \cite{sRWRE} or \cite{zRWRE} for a reference). 

Secondly, note that, since the $\{\nu_{i} - \nu_{i-1} \}_{i\geq 1}$
are i.i.d.\ under Q , it follows that $\lim_{n\ra\infty} \nu_n/n = \bar\nu$, $Q$-a.s.
This implies that
\[
 \lim_{n\ra\infty} \frac{1}{n} \sum_{i=1}^n \b_i = \lim_{n\ra\infty} \frac{E_\w T_{\nu_n}}{n} = \lim_{n\ra\infty} \frac{ E_\w T_{\nu_n}}{\nu_n} \frac{\nu_n}{n} = \frac{\bar\nu}{\vp}, \quad Q\text{-a.s.}
\]
Finally, since the $\b_i$ are stationary under $Q$, it is a
consequence of Birkhoff's Ergodic Theorem that this nonrandom limit of
$1/n \sum_{i=1}^n \b_i$ must coincide with $E_Q[ \b_1 ]$.  
\end{proof}

\section{Topological Generalities}\label{MTgeneral}

\subsection{Skorohod Topologies}
In this section we recall the definitions of the Skorohod $J_1$ and
$M_1$ metrics on the space $D_\infty$, and the corresponding
topologies. We also give certain technical results that will be needed
in the sequel. The details that we omit can be found in \cite{bCOPM}
and \cite{wSPL}. 

For $0<t<\infty$ the $J_1$ and $M_1$ Skorohod metrics on the space
$D_t$ of \cadlag\ functions on $[0,t]$ are defined as follows. 
Let $\L_t$ be the set of time-change functions on $[0,t]$ -- functions
that are strictly increasing and continuous bijections from $[0,t]$ to
itself.  The Skorohod $J_1$-metric (on $D_t$) is defined by 
\[
 d_t^{J_1}(x,y) = \inf_{\l \in \L_t} \max \left \{ \sup_{s\leq t} |
   \l(s) - s | , \, \sup_{s \leq t} | x(\l(s)) - y(s) | \right\}.  
\]
Next, recall that the completed graph of a \cadlag\  function $x \in
D_t$ is the subset $\Gamma_x \subset [0,t] \times \R$ defined by 
\[
 \Gamma_x = \{ (u,v): \, u \in [0,t], \, 
v = (1-\theta) x(u-) + \theta x(u), \text{ for some } \theta \in [0,1] \}. 
\]
The natural order $\preceq_{\Gamma_x}$  on the completed graph is
given by $(u_1,v_1) \preceq_{\Gamma_x} (u_2,v_2)$ if either $u_1 < u_2$ or $u_1 = u_2$ and $|v_1 - x(u_1-)| \leq |v_2 - x(u_1-)|$. 
A \emph{parametric representation} of the completed graph $\Gamma_x$
is a function from $[0,1]$ onto $\Gamma_x$ that is continuous with
respect to the subspace topology on $\Gamma_x$ and 
non-decreasing with respect to the order $\preceq_{\Gamma_x}$. Let
$\Pi(x)$ be the set of parametric representations of $\Gamma_x$,
with each parametric representation given by a pair of functions $u$ and $v$ on $[0,1]$ such that $\Gamma_x = \{(u(s),v(s)): \, s \in[0,1] \}$.
%each one of which viewed as a pair of functions $u$ and $v$ on
%$[0,1]$ such that $\Gamma_x = \{(u(s),v(s)): \, s \in[0,1] \}$. 
The Skorohod $M_1$-metric on $D_t$ is defined by 
\[
 d^{M_1}_t(x,y) = \inf_{(u,v) \in \Pi(x), \, (u',v') \in \Pi(y)} \max\left\{  \sup_{s \in [0,1]} |u(s) - u'(s)|  ,  \sup_{s \in [0,1]} |v(s) - v'(s)|  \right\}. 
\]
%Maybe an equivalent definition is 
%\[
% d^{M_1}_T(x,y) = \inf_{\vec{r} \in \Pi(x), \, \vec{p} \in \Pi(y)} \sup_{s \in[0,1]} \| \vec{r}(s) - \vec{p}(s) \|_\infty. 
%\]
The Skorohod $J_1$ and $M_1$-metrics on $D_t$ for all finite  $t$
produce corresponding metrics on the space $D_\infty$ by 
\[
 d^{J_1}_\infty(x,y) = \int_{0}^{\infty} e^{-t} \left( d^{J_1}_t(x^{(t)},y^{(t)}) \wedge 1 \right) \, dt, \quad \text{and}\quad d^{M_1}_\infty(x,y) = \int_{0}^{\infty} e^{-t} \left( d^{M_1}_t(x^{(t)},y^{(t)}) \wedge 1 \right) \, dt. 
\]
Here, for $x\in D_\infty$, the function $x^{(t)}\in D_t$ is the
restriction of $x$ to the finite time interval $[0,t]$. Using instead the uniform
metric on each $D_t$ produces the metric $d^{U}_\infty$ on the space
$D_\infty$. 

The following is a list of several useful properties of the Skorohod
metrics that we will use throughout the paper; see \cite{wSPL}. 
\begin{itemize}
 \item $d_\infty^{M_1}(x,y) \leq d_\infty^{J_1}(x,y)$.  
 \item $d_\infty^{J_1}(x,y) \leq e^{-s} + \sup_{t\leq s}|x(t)-y(t)|$
   for any $0<s<\infty$;  thus uniform convergence on compact subsets of
   $[0,\infty)$ implies convergence in the $J_1$-Skorohod metric.  
 \item $d_t^{M_1}(x,y) \geq |x(t) - y(t)|$ for each $0<t<\infty$. 
 \item $d_\infty^{J_1}(x_n,x) \ra 0$ if and only if $d_t^{J_1}(x_n,x)
   \ra 0$ for all continuity points $t$ of $x$. An analogous statement
   is true for the $M_1$-Skorohod topology. 
\end{itemize}
The $J_1$ and $M_1$-metrics generate topologies on the space of
\cadlag\ functions. Even though the two metrics are not complete, each
of them has an equivalent metric that is complete. Therefore, the $J_1$ and $M_1$
topologies are the topologies of complete separable metric spaces. 

The $J_1$ and $M_1$-metrics on $D_\infty$ induce in the standard way
the corresponding
Prohorov's metrics, $\rho^{J_1}$ and $\rho^{M_1}$ on the space of
Borel probability measures $\mathcal{M}_1(D_\infty)$. For example, for
any $\mu,\pi \in \mathcal{M}_1(D_\infty)$,  
\[
 \rho^{M_1}(\mu,\pi) = \inf \left\{ \d>0: \mu(A) \leq \pi(A^{\d,M_1})
   + \d, \ 
\text{for every Borel $A\subset D_\infty$}\right\}
\]
(recall that the $J_1$ and $M_1$-metrics generate the same Borel sets
on $D_\infty$; these are also the cylindrical sets). 
Further, 
\[
A^{\d,M_1} = \{ y:\, d_\infty^{M_1}(x,y) < \d, \text{ for some } x \in A \} .
\]
Since $(D_\infty, d_\infty^{M_1})$ is a separable metric space,
convergence in the
Prohorov metric $\rho^{M_1}$ is equivalent to 
convergence in distribution in $(D_\infty, d_\infty^{M_1})$; see
Theorem 3.2.1 in \cite{wSPL}. Moreover, the space $\bigl(
\mathcal{M}_1(D_\infty), \rho^{M_1}\bigr)$ is a complete separable
metric space (Theorem 6.8 in \cite{bCOPM}).

\subsection{Continuity of functionals}
We proceed with two results on the continuity of certain functionals that we will need later. 
We begin, by recalling the following result from \cite{wSPL} on the continuity of the composition map. 
\begin{lem}[Theorems 13.2.2, 13.2.3  in \cite{wSPL}]\label{compcont}
 The composition map $\psi: D_\infty \times D_\infty^+ \ra D_\infty$
 defined by $\psi(x,y)=x\circ y$ is continuous on the set $D_\infty
 \times C_{\uparrow\uparrow}^+$, where $D_\infty^+$ is the set of all
 nonnegative functions in $D_\infty$, and $C_{\uparrow\uparrow}^+$ is the set
 of continuous, non-negative, strictly increasing functions on
 $[0,\infty)$. The continuity holds whenever either the $J_1$-topology is
 used throughout, or the $M_1$-topology is
 used throughout. 
\end{lem}
The composition map $\psi$ induces a map $\Psi:
\mathcal{M}_1(D_\infty) \times D_\infty^+\ra \mathcal{M}_1(D_\infty)$
by 
\[
 \Psi(\mu,y)(\{x: \, x \in A \}) = \mu( \{x: \, x \circ y \in A \} ) .
\] 
Lemma \ref{compcont} leads to the following continuity result for $\Psi$. 
\begin{cor}\label{mcompcont}
 The map $\Psi$ is continuous on the set $\mathcal{M}_1(D_\infty)
 \times C_{\uparrow\uparrow}^+$, if the same topology (either $J_1$ or
 $M_1$) is used throughout. 
\end{cor}
\begin{proof}
 Suppose that $(\mu_n,y_n) \ra (\mu,y) \in \mathcal{M}_1(D_\infty) \times C_{\uparrow\uparrow}^+$. 
By the Skorohod representation theorem (e.g. Theorem 3.2.2 in
\cite{wSPL}), there are $D_\infty$-valued random elements $(X_n),X$
defined on a common probability space such that $X_n\sim \mu_n$ for
each $n$, $X\sim \mu$, and $X_n\to X$ a.s. in the corresponding
Skorohod metric. By Lemma \ref{compcont} we know that $X_n\circ y_n\to
X\circ y$ in the same metric. Since a.s. convergence implies
weak convergence, the claim follows. 
\end{proof}

%Two more functionals that we will use continuity properties of are the inversion $\mathfrak{I}$ and reflection $\mathfrak{R}$ operators defined in the introduction prior to Theorem \ref{WQLXn}. The reflection operator $\mathfrak{R}$ is obviously continuous, but the inversion operator $\mathfrak{I}$ is not continuous. The following Lemma from \cite{wSPL} will be useful
%\begin{lem}
% The inversion functional $\mathfrak{I}$ on $D_\infty$ defined
%\end{lem}

\subsection{Deducing weak convergence of  random probability measures}

In order to prove weak convergence of  a sequence of random
probability measures on $ D_\infty$  we will often use the coupling
technique which we now describe. Suppose that $\mu,\pi \in \mathcal{M}_1(D_\infty)$. Then a
\emph{coupling} of $\mu$ and $\pi$ is a probability measure $\theta$
on the product space $D_\infty \times D_\infty$ with marginals $\mu$ and
$\pi$, respectively.  A coupling of two random probability measures on
$D_\infty$ defined on a common probability space is a random element
of $\mathcal{M}_1(D_\infty\times D_\infty) $ defined on the same
probability space that couples the two measures for every $\omega$. 
The following simple lemma, which is a path
space extension of Lemma 3.1 in \cite{psWQLTn}, is the key ingredient
in our approach. 
\begin{lem}\label{WAEcor}
 Suppose that $(\mu_n), (\pi_n)$ are two sequences of random elements
 in $\mathcal{M}_1(D_\infty)$ defined on a common probability space
 with probability measure $\Pv$ and expectation $\Ev$. Suppose that
 one of the following conditions holds. 
\begin{enumerate}
 \item $\lim_{n \ra \infty} \Pv( \rho^{M_1}( \mu_n , \pi_n ) \geq \eta ) = 0$, for all $\eta>0$. \label{rhoM}
%\item For each $n$ there exists a coupling $\theta_n$ of the random probability measures $\mu_n$ and $\pi_n$ such that \label{weakcouple}
%\[
% \lim_{n\ra\infty} \Pv( \theta_n ( \{ (x,y): \, d_\infty^{M_1}(x,y) \geq \eta \} ) > \eta ) = 0, \quad \forall \eta>0.
%\]
 \item For each $n$ there exists a coupling $\theta_n$ of the random probability measures $\mu_n$ and $\pi_n$ such that \label{annealed}
\[
 \lim_{n\ra\infty} \Ev\left[ \theta_n ( \{ (x,y): \,
   d_\infty^{M_1}(x,y) \geq \eta \} ) \right] = 0, \quad \text{for
   all} \ \eta>0. 
\]
 \item For each $n$ there exists a coupling $\theta_n$ of the random probability measures $\mu_n$ and $\pi_n$ such that \label{Wass}
\[
 \lim_{n\ra\infty} \Pv( E_{\theta_n} [ d_\infty^{M_1}(x,y) ] \geq \eta
 ) = 0, \quad \text{for all} \ \eta>0,
\]
where $E_{\theta_n}$ denotes expectations under the measure $\theta_n$. 
% \item For each $n$ there exists a coupling $\theta_n$ of the random probability measures $\mu_n$ and $\pi_n$ such that \label{EE}
%\[
% \lim_{n\ra\infty} \Ev\left[ E_{\theta_n} [ d_\infty^{M_1}(x,y) ] \right] = 0.
%\]
%where $E_{\theta_n}$ denotes expectations under the measure $\theta_n$. 
\end{enumerate}
If $\mu_n \Lra \mu$ weakly in $(\mathcal{M}_1(D_\infty),\rho^{M_1})$, then $\pi_n
\Lra \mu$ weakly in $(\mathcal{M}_1(D_\infty) ,\rho^{M_1})$. 
\end{lem}
\begin{proof}
Under condition \eqref{rhoM} the statement follows from Theorem 3.1 in
\cite{bCOPM}. 
%Next, we claim that conditions \eqref{weakcouplec} and \eqref{weakcouple} imply conditions \eqref{rhoTc} and \eqref{rhoT}, respectively.
Next, note that the definition of the Prohorov metric implies that if $\theta_n( \{ (x,y) : d_\infty^{M_1} (x,y) \geq \eta \} ) \leq \eta$ then $\rho^{M_1}( \mu_n, \pi_n) \leq \eta$. Therefore, 
\begin{align*}
 \Pv( \rho^{M_1}( \mu_n , \pi_n ) > \eta ) 
&\leq \Pv( \theta_n ( \{(x,y):\,  d_\infty^{M_1}(x,y) \geq \eta \}) > \eta ) \\
&\leq \frac{1}{\eta} \Ev \left[ \theta_n ( \{(x,y):\,  d_\infty^{M_1}(x,y) \geq \eta \}) \right]. 
\end{align*}
Thus, condition \eqref{annealed} implies condition \eqref{rhoM}.
Furthermore, condition \eqref{Wass} implies  condition
\eqref{annealed} by Chebyshev's inequality. 
\end{proof}

The following lemma will allow us to reduce checking condition
\eqref{Wass} in Lemma \ref{WAEcor} to the finite time situation. We
note that a similar reduction holds under the metrics $d^{J_1}$ and
$d^U$ as well. 
\begin{lem}\label{Dt2Dinfty}
 Suppose that $(\mu_n), (\pi_n)$ are two sequences of random elements
 in $\mathcal{M}_1(D_\infty)$ defined on a common probability space
 with probability measure $\Pv$ and expectation $\Ev$. If for each $n$
 there exists a coupling $\theta_n$ of the random probability measures
 $\mu_n$ and $\pi_n$ such that for every $0<t<\infty$ 
\be \label{Dtconv}
  \lim_{n\ra\infty} \Pv( E_{\theta_n} [ d_t^{M_1}(x^{(t)},y^{(t)}) ] \geq \eta ) =
  0, \quad \forall \eta>0,  
\ee
then condition \eqref{Wass} in Lemma \ref{WAEcor} holds. 
\end{lem}
\begin{proof}
 By the bounded convergence theorem, \eqref{Dtconv} implies that 
$$
 \lim_{n\ra\infty} \Ev\left[ E_{\theta_n}[ d_t^{M_1}(x^{(t)},y^{(t)}) ] \wedge 1
 \right] = 0, \quad \forall t<\infty.  
$$
By the definition of $d_\infty^{M_1}$, Fubini's Theorem and dominated
convergence theorem we immediately see that
\begin{align*}
 \Ev\left[ E_{\theta_n}[ d_\infty^{M_1}(x,y) ] \right]
&= \Ev\left[ E_{\theta_n}\left[ \int_0^\infty e^{-t} \left(d_t^{M_1}(x^{(t)},y^{(t)}) \wedge 1 \right) \, dt \right] \right] \\
&= \int_0^\infty e^{-t} \Ev\left[ E_{\theta_n}\left[ \left(d_t^{M_1}(x^{(t)},y^{(t)}) \wedge 1 \right)  \right] \right] \, dt
\end{align*}
vanishes as $n\ra \infty$. This implies condition \eqref{Wass} in
Lemma \ref{WAEcor}.  
\end{proof}

\section{Comparison with sums of exponentials}\label{CompareWithExp}

The main goal of this section is to reduce the study of the hitting
time process $\T_\e$ to the study of a process $\S_\e$ that is defined
in terms of sums of exponential random variables. To this end, recall
the definition of the ladder locations of the environment in
\eqref{nudef} and the notation $\b_i = \b_i(\w) = E_\w[ T_{\nu_i} - T_{\nu_{i-1}} ]$ for the quenched expectation of the time to cross from $\nu_{i-1}$ to $\nu_i$. 
%as in \cite{psWQLTn}, let $\b_i$ be the quenched
%expectation of the amount of time it takes the random walk to cross
%from $\nu_{i-1}$ to $\nu_i$. 
%That is,   
%\[
% \b_i = \b_i(\w) = E_\w\left[ T_{\nu_i} - T_{\nu_{i-1}} \right], \
% i=1,2,\ldots. 
%\]
Also, we expand the measure $P_\w$ to include an i.i.d. sequence of
standard exponential random variables $(\tau_i)$; it will be used in
the coupling procedure below by comparing $T_{\nu_i} - T_{\nu_{i-1}}$ with $\b_i \tau_i$. 

For any realization of the environment we construct random paths
$\U_\e, \S_\e \in D_\infty$ as follows. For $t\geq 0$, 
\[
 \U_\e(t) =
\begin{cases}
 \e^{1/\k} T_{\nu_{\fl{t/\e}}} & \k \in(0,1) \\
 \e ( T_{\nu_{\fl{t/\e}}} - t/\e D'(1/\e) ) & \k = 1 \\
 \e^{1/\k} ( T_{\nu_{\fl{t/\e}}} - \bar\b t/\e ) & \k \in (1,2),
\end{cases}
\]
and 
\be\label{Sdef}
 \S_\e(t) = 
\begin{cases}
 \e^{1/\k} \sum_{i=1}^{\lfloor t/\e\rfloor} \b_i \tau_i & \k \in(0,1) \\
 \e ( \sum_{i=1}^{\lfloor t/\e\rfloor} \b_i \tau_i - t/\e D'(1/\e) ) & \k = 1 \\
 \e^{1/\k} ( \sum_{i=1}^{\lfloor t/\e\rfloor} \b_i \tau_i - \bar\b t/\e ) & \k \in (1,2),
\end{cases}
\ee
where  $D'(x) = E_Q[ \b_1 \ind{\b_1 \leq x}] \sim \tc \log x$ when $\k=1$ and $\bar\b = E_Q[\b_1] = E_Q[ E_\w T_\nu ] $ when $\k \in (1,2)$.
\begin{rem}
The proof below will show that in the case $\k=1$, the function $D$ in
definition of $\T_\e$ can be chosen to be $D(x)=D'(x)/\bar\nu$ (recall
that $\bar\nu = E_Q[\nu_1]$.) In particular, the constant $A$ in Theorem
\ref{averagedlimlaw} satisfies $A = \tc/\bar\nu$. 
\end{rem}
%\be
% \U_\e(t) = \e^{1/\k} T_{\nu_{\fl{t/\e}}}, \quad\text{and}\quad \S_\e(t) = \e^{1/\k} \sum_{i=1}^{t/\e} \b_i \tau_i. 
%\ee
%\[
% \U_\e(t) = \e^{1/\k} T_{\nu_{t/\e}} \text{ if } t \in \e \Z, \text{ and piecewise linear elsewhere.} 
%\]
%\[
% \S_\e(t) = \e^{1/\k} \sum_{i=1}^{t/\e} \b_i \tau_i \text{ if } t \in \e \Z, \text{ and piecewise linear elsewhere.}
%\]
Let $u_{\e}=u_{\e,\w}, s_{\e}=s_{\e,\w} \in \mathcal{M}_1(D_\infty)$ be the quenched
distributions of $\U_\e$ and $\S_\e$, respectively. That is,  
\[
 u_{\e,\w} = P_\w( \U_\e \in \cdot ), \quad\text{and}\quad s_{\e,\w} = P_\w(\S_\e \in \cdot). 
\]
We view $u_{\e}$ and $s_{\e} $ as random elements in $\mathcal{M}_1(D_\infty)$. The proof of Theorem \ref{WQLTn1} is
accomplished via the following two propositions.  The first
proposition establishes  weak convergence of $s_\e$ in
$\mathcal{M}_1(D_\infty)$. The notation and the terminology are the
same as in Theorem \ref{WQLTn1}. 
\begin{prop}\label{WQLSn} Let $\l = \tc \k$, where  $\tc$ is the tail
  constant in \eqref{btail}.  The following statements hold under the
  probability measure $Q$ on the environments. 
%If $\k \in (0,1)$, then there exists a $\l>0$ such that $s_{\e,\w} \overset{Q}{\Lra}_\e \H(N_{\l,\k})$. 
\begin{enumerate}
\item If $\k\in(0,1)$, then $s_{\e}
  \Lra \H(N_{\l,\k})$.  
\item If $\k=1$, then 
\[
 s_{\e} \Lra \lim_{\d\ra 0} \H_\d(N_{\l,1}) * \lm(-\l \log(1/\d)). 
\]
%\[
% s_{\e,\w} \overset{Q}{\Lra} \lim_{\d\ra 0} \H_\d(N_{\l,1}) * \d_{-\tilde{x}_{\l,1,\d}}, 
%\]
%where $\tilde{x}_{\l,1,\d} \in D_\infty$ is the path defined by $\tilde{x}_{\l,1,\d} (t) = \l \log(\bar\nu/\e) t$. 
\item If $\k\in(1,2)$, then 
\[
 s_{\e} \Lra \lim_{\d\ra 0} \H_\d(N_{\l,\k}) * \lm(-\l
 \d^{-\k+1}/(\k-1)).  
\]
%\[
% s_{\e,\w} \overset{Q}{\Lra} \lim_{\d\ra 0} \H_\d(N_{\l,\k}) * \d_{-x_{\l,\k,\d}}, 
%\]
%where $x_{\l,\k,\d} \in D_\infty$ is the path defined by $x_{\l,\k,\d} (t) = \frac{\l}{\k-1} \d^{-(\k-1)} t$. 
\end{enumerate}
\end{prop}
%To compare $\T_\e$ with $\S_\e$ we need to do a re-scaling of
%time. To this end, for any $c>0$ define the time re-scaling operator
%$\a_c:D_\infty \ra D_\infty$ by $(\a_c x)(t) = x(c t)$. 
The second proposition relates a weak limit of $s_\e$ in
$\mathcal{M}_1(D_\infty)$ to the corresponding weak limit of $m_\e$.
\begin{prop}\label{mnsncompare}
Define $\l_0\in C_{\uparrow\uparrow}^+$ by 
$\l_0(t) = t/\bar\nu$. If $s_{\e}\Lra\mu$ weakly in
$\mathcal{M}_1(D_\infty)$ under $Q$, then $m_{\e} \Lra \Psi(\mu,\l_0)$
under $P$. 
\end{prop}

Before giving the proofs of Propositions \ref{WQLSn} and
\ref{mnsncompare}, we show how they imply Theorem \ref{WQLTn1}.  
\begin{proof}[Proof of Theorem \ref{WQLTn1}]
%\textbf{Case I: $\k \in (0,1)$.}
%By Corollary \ref{mcompcont} the mapping $\mu \mapsto \Psi(\mu,\l_0)$ is a continuous mapping of $D_\infty$ onto itself. Thus, combining Propositions \ref{WQLSn} and \ref{mnsncompare} we obtain that $m_{\e,\w} \overset{P}{\Lra} \Psi(\H(N_{\l,\k}), \l_0).$ It remains only to show that $ \Psi(\H(N_{\l,\k}), \l_0)$ has the same distribution as $\H(N_{\l',\k})$ where $\l' = \l/\bar\nu$. To this end, note that
%\begin{align*}
% \Psi\left( \H\left( \sum_{i\geq 1} \d_{(x_i,t_i)} \right), \l_0 \right) = \Pv_\tau \left( \left\{ t \mapsto \sum_{i\geq 1} x_i \tau_i \ind{t_i \leq t/\bar\nu} \right \} \in \cdot \right) = \H\left( \sum_{i\geq 1} \d_{(x_i,t_i \bar \nu)} \right).
%\end{align*}
%However, if $\sum_{i\geq 1} \d_{(x_i,t_i)}$ is a Poisson point process with intensity measure $\l x^{-\k-1} \, dx \, dt$ then $\sum_{i\geq 1} \d_{(x_i,t_i \bar\nu)}$ 
%is a Poisson point process with intensity measure $ (\l/\bar\nu) x^{-\k-1} \, dx \, dt$. 
%
%\textbf{Case II: $\kappa \in [1,2)$.}
The proof is essentially the same, whether $\k \in (0,1)$, $\k=1$, or
$\k\in(1,2)$, therefore we only spell out the details in the case $\k
\in (1,2)$.   

First, note that by Propositions \ref{WQLSn} and \ref{mnsncompare} and
Corollary \ref{mcompcont}, under the measure $P$, 
\[
 m_{\e} \Lra \lim_{\d \ra 0} \Psi\left( \H_\d(N_{\l,\k}) * \lm\left(-\frac{\l \d^{-\k+1}}{\k-1}\right) , \l_0 \right).
\]
Therefore, it is enough to show that, for any $\l>0$, with $\l' =
\l/\bar\nu$, 
\be\label{rescale}
 \Psi\left( \H_\d(N_{\l,\k}) * \lm\left(-\frac{\l
       \d^{-\k+1}}{\k-1}\right) , \l_0 \right) \overset{\text{Law}}{=}
 \H_\d(N_{\l',\k}) * \lm\left(-\frac{\l' \d^{-\k+1}}{\k-1}\right).  
\ee
To see this, note that for any $m>0$, 
\[
 \Psi\left( \H_\d\left( \sum_{i\geq 1} \d_{(x_i,t_i)} \right)* \lm(-m),
  \l_0 \right) = \H_\d\left( \sum_{i\geq 1} \d_{(x_i,t_i \bar \nu)}
\right) * \lm(-m/\bar\nu). 
\]
However, if $\sum_{i\geq 1} \d_{(x_i,t_i)}$ is a Poisson point process
with intensity measure $\l x^{-\k-1} \, dx \, dt$ then $\sum_{i\geq 1}
\d_{(x_i,t_i \bar\nu)}$  is a Poisson point process with intensity
measure $ (\l/\bar\nu) x^{-\k-1} \, dx \, dt$. This implies 
\eqref{rescale}. 
\end{proof}

It remains to prove Propositions \ref{WQLSn} and \ref{mnsncompare}. 
Proposition \ref{WQLSn} will be proved in Section \ref{ExpWQL}, and in
the remainder of this section will focus on the proof of Proposition
\ref{mnsncompare} which follows immediately from Lemma \ref{WAEcor}
and the following lemmas.  
\begin{lem}\label{Expcouple}
There exists a coupling of $\U_\e$ and $\S_\e$ such that, for any $\eta>0$,
$$\lim_{\e\ra 0} Q\left( E_\w[ d_\infty^{J_1}( \U_\e, \S_\e) ] \geq \eta \right) = 0.$$
%For any $\d>0$, $\lim_{\e\ra 0} Q\left( W^1( u_{\e,\w}, s_{\e,\w} ) \geq \d \right) = 0$.
\end{lem}
Under the assumption of Proposition \ref{mnsncompare}, this lemma and
part (3) of Lemma \ref{WAEcor} will imply that $u_{\e} \Lra \mu$ in
$\mathcal{M}_1(D_\infty)$ under $Q$. 
\begin{lem}\label{TnTnuncouple}
If $u_{\e} \Lra \mu$ in $\mathcal{M}_1(D_\infty)$ under $Q$, then
$m_{\e} \Lra \Psi( \mu, \l_0 )$ in $\mathcal{M}_1(D_\infty)$ under
$Q$. 
%Let $c= (E_P \nu_1)^{-1}$. Then, for any $T<\infty$, 
%\[
% \lim_{\e\ra 0} Q\left( E_\w\left[ d^{M_1}_T( \T_\e, \a_c \U_\e) \right] \geq \d \right) = 0, \quad \forall \d>0. 
%\]
\end{lem}
Under the assumption of Proposition \ref{mnsncompare}, this lemma will
imply that $m_{\e} \Lra \Psi(\mu,\l_0)$ in $\mathcal{M}_1(D_\infty)$ under $Q$. 
\begin{lem}\label{PQcouple}
% There exists a measure $\mathfrak{P}$ on pairs of environments $(\w,\w')$ such that the marginal distributions of $\w$ and $\w'$ are $P$ and $Q$, respectively, and such that  
%\[
% \lim_{\e\ra 0} \mathfrak{P}\left( W^1( m_{\e,\w}, m_{\e,\w'} ) \geq \d \right) = 0, \quad \forall \d>0. 
%\]
There exists a measure $\mathfrak{P}$ on pairs of environments
$(\w,\w')$ such that the marginal distributions of $\w$ and $\w'$ are
$P$ and $Q$, respectively, and, for each $\e>0$,  there exists a
coupling $P_\e = P_{\e;\w,\w'}$ of the random 
measures $m_{\e,\w}$ and $m_{\e,\w'}$ such that  
%\[
% \lim_{\e\ra 0} \mathfrak{P}\left( E_{\w,\w'}\left[ \sup_{t<\infty} |\T_\e(t) - \T_\e'(t) | \right] \geq \eta \right) = 0, \quad \forall \eta > 0. 
%\] 
\[
 \lim_{\e\ra 0} \mathfrak{P}\left( E_{\e}[ d_\infty^{U}(x,y) ] \geq
   \eta \right) = 0, \quad \text{for all} \ \eta > 0.  
\]
\end{lem}
Under the assumption of Proposition \ref{mnsncompare}, this lemma and
another appeal to part (3) of Lemma \ref{WAEcor} will imply the claim
of the proposition. We proceed now to prove the three lemmas. 
\begin{proof}[Proof of Lemma \ref{PQcouple}]
We use the same construction as in the proof of Lemma 4.2 in \cite{psWQLTn}. 
%We will describe this coupling briefly and refer to \cite{psWQLTn} for the details - which are rather straightforward. 
First let $\w$ and $\tilde{\w}$ be independent with distributions $P$
and $Q$ respectively. Then, construct $\w'$ by letting  
\[
 \w'_x = \begin{cases} \tilde{\w}_x & x \leq -1 \\ \w_x & x \geq 0. \end{cases}
\]
Then $\w'$ has distribution $Q$ and is identical to $\w$ on the
non-negative integers. Let $\mathfrak{P}$ be the joint law of 
$(\w,\w')$. 

Given a pair of environments $(\w,\w')$, we construct coupled 
random walks $\{X_n\}$ and $\{X_n'\}$ 
%with hitting times $\{T_n\}$ and $\{T_n'\}$, respectively, 
so that the marginal laws of $\{X_n\}$ and $\{X_n'\}$ 
%(or $\{T_n\}$ and $\{T_n'\}$) 
are $P_\w$ and $P_{\w'}$ respectively. 
We do that by coordinating all steps of the random walks to the right
of 0. That is, since $\w_x = \w_x'$ for any $x\geq 0$, we require that
on the respective $i^{th}$ visits of the walks $X_n$ and $X_n'$ to
site $x$ they both either move to the right or both move to the
left. The details of this coupling can be found in \cite{psWQLTn}. 
Let $P_{\e}=P_{\e;\w,\w'}$ denote the joint quenched law of the two
random walks coupled in this manner; the  corresponding expectation is 
denoted by $E_{\e}=E_{\e;\w,\w'}$. Let $T_k$, $T_k'$ and $\T_\e$, $\T_\e'$ be
the hitting times and the path processes of hitting times
corresponding to the random walks $\{X_n\}$ and $\{X'_n\}$,
respectively.  Note that
\[
 d_\infty^{U}( \T_\e, \T_\e' ) \leq \sup_{t<\infty} |\T_\e(t) -
 \T_\e'(t)| = \e^{1/\k} \sup_{n\geq 1} |T_n - T_n'|.  
\]
However, it is easy to see that the coupling of $X_n$ and $X_n'$ is
such that $\sup_{n\geq 1} |T_n - T_n'| = |L-L'|$, where $L$ and $L'$
are the number of steps that the walks $\{X_n\}$ and $\{X_n'\}$,
respectively, spend to the left of 0. It is easy to see (and was shown
in the proof of Lemma 4.2 in \cite{psWQLTn}) that $E_{\e;\w,\w'}|L-L'|
\leq E_\w L + E_{\w'} L' < \infty$, $\mathfrak{P}$-a.s. 
Therefore, for any $\eta>0$
\[
 \lim_{\e\ra 0} \mathfrak{P}\left( E_{\e}\left[ d_\infty^{U}( \T_\e, \T_\e' ) \right] \geq \eta \right) 
\leq \lim_{\e\ra 0} \mathfrak{P}\left( \e^{1/\k} E_{\w,\w'}|L-L'| \geq \eta \right) = 0. 
\]
\end{proof}

\begin{proof}[Proof of Lemma \ref{TnTnuncouple}]
We start with a time change in the process $\U_\e$ to align its jumps
with the hitting times of corresponding ladder locations in the process
$\T_\e$.  
To this end, define $\l_\e \in D_\uparrow^+$, the space of nonnegative
non-decreasing functions in $D_\infty$, by 
\[
 \l_\e(t) = \e \max\{ k: \, \e\nu_k \leq t \}, \ t\geq 0.
\]
Then, the renewal theorem implies that $\lim_{\e\ra 0} \l_\e(t) =
\l_0(t)$, $Q$-a.s, for any fixed $t\geq 0$. Since $\l_\e$ is
non-decreasing and $\l_0$ is continuous and non-decreasing, the
convergence is uniform on compact subsets of $[0,\infty)$.   
Therefore, $\lim_{\e\ra 0} d_\infty^{J_1}(\l_\e, \l_0) = 0$,
$Q$-a.s. Furthermore, it follows from the functional central limit
theorem for renewal sequences with a finite variance that 
$\e^{-1/2}( \l_\e-\l_0 )$ converges weakly in $(D_\infty,J_1)$ to a
Brownian motion, as $\e \ra 0$. See Theorem 7.4.1 in \cite{wSPL}. 

The assumption $u_{\e} \Lra \mu $ and Corollary \ref{mcompcont} show
that $\Psi(u_{\e}, \l_\e) \Lra \Psi(\mu, \l_0)$ under $Q$, so 
%Since $\l_0$ is continuous and strictly increasing, Corollary \ref{mcompcont} implies that
%\[
% u_{\e,\w} \overset{Q}{\Lra} \mu \, \text{ implies that } \Psi(u_{\e,\w}, \l_\e) \overset{Q}{\Lra} \Psi(\mu, \l_0). 
%\]
by Lemmas \ref{WAEcor} and \ref{Dt2Dinfty}, the claim of the present
lemma will follow once we show that for every $0<t<\infty$ and
$\eta>0$, 
\be\label{UlTcompare}
 \lim_{\e\ra 0} Q\left( E_\w\left[ d_t^{M_1}( \U_\e \circ \l_\e, \,
     \T_\e ) \right] \geq \eta \right) = 0. 
\ee
%It will be enough to show that 
%\be\label{UlTcompare}
% \lim_{\e\ra 0} Q\left( E_\w\left[ d_\infty^{M_1}( \U_\e \circ \l_\e, \, \T_\e ) \right] \geq \eta \right) = 0, \qquad \forall \eta > 0. 
%\ee
To simplify the notation, we omit the superscripts in functions of the
type $\T_\e^{(t)}$. 
Because of the centering present when $\k \in[1,2)$ but not when $\k
\in (0,1)$, we treat the two cases separately. 

\textbf{Case I: $\k \in (0,1)$.}
Note that the definition of $\l_\e$ implies that $\U_\e(\l_\e(t)) =
\e^{1/\k} T_{\nu_j} = \T_\e(t)$ when $t= \e \nu_j$.  We arrange
the respective parametric representations of the completed graphs of
the two random functions, $\U_\e\circ \l_\e$ and $\T_\e$, so that at each $s_j=j/(k+1)\in [0,1)$ both 
parametric representations are equal, to $\bigl( \e \nu_j, \e^{1/\k}
T_{\nu_j}\bigr)$. Here $k$ is the largest $j$ so that $\nu_j\leq
t/\e$. For $s_j<s<s_{j+1}$ with $j=0,1,\ldots, k-1$ we arrange the two
parametric representations so that the vertical ($v$) coordinates
always stay the same (see Figure \ref{MatchingFigure}). Then the
distance between the corresponding points on the completed graphs on the
interval in that range of $s$ is taken horizontally, and it is, at most,
$\e (\nu_{j+1} - \nu_j)$. This horizontal matching cannot, generally, be performed on
the interval $(s_k,1]$ since the two functions may not be equal at
time $t$. On this interval we keep horizontal ($u$) coordinates the
same. The distance between the corresponding points is now taken vertically, and
it is, at most, 
$\e^{1/\k} (T_{\nu_{k+1}} - T_{\nu_k})$. Therefore, 
\[
 d_t^{M_1}(\U_\e\circ \l_\e, \, \T_\e) \leq \max \left\{ \max_{j<k}
   \e(\nu_{j+1} - \nu_j) , \, \e^{1/\k} (T_{\nu_{k+1}} - T_{\nu_k})
 \right\}. 
\]
Since $k\leq t/\e$, we conclude, using stationarity of the sequence
$(\nu_{j+1}-\nu_j)$ under $Q$  that for $0 <\e<1$ so small that $\e
(\log 1/\e)^2\leq \eta$, 
\begin{align*}
Q\left( E_\w [ d_t^{M_1}(\U_\e\circ \l_\e, \, \T_\e) ] \geq \eta \right) 
%&\leq Q\left( \nu_{j+1} - \nu_j > \log^2(1/\e), \, \text{ for some } j < t/\e \right) \\
%&\qquad + Q\left( \e^{1/\k} \Ev[ (T_{\nu_{k+1}} - T_{\nu_k}) \ind{t \in [\e\nu_k, \e\nu_{k+1})} ] > \eta \right) \\
 &\leq \frac{t}{\e}Q\left( \nu_{1} > \log^2(1/\e) \right) \\
& + Q\left( \e^{1/\k} \b_{k+1} > \eta \ \text{for} \ t \in [\e\nu_k, \e \nu_{k+1})  \right). 
\end{align*}
Since $\nu_1$ has some finite exponential moments (see
\cite{pzSL1}), the first term on the right above vanishes as $\e\ra
0$. For the second term note that $t \in [\e \nu_k,\, \e \nu_{k+1})$
is equivalent to $\l_\e(t) = \e k$, hence   
\begin{align*}
& Q\left( \e^{1/\k} \b_{k+1} > \eta \ \text{for} \ t \in [\e\nu_k, \e
  \nu_{k+1})  \right) \\
&\quad \leq Q\left( |\l_\e(t) - t/\bar\nu| > \e^{1/4} \right) + Q\left(\exists k: | k - t/(\bar\nu \e)| \leq \e^{-3/4} , \,  \b_{k+1} > \eta \e^{-1/\k} \right) \\
&\quad \leq Q\left( |\l_\e(t) - t/\bar\nu| > \e^{1/4} \right) + 2 \e^{-3/4} Q(\b_1 > \eta \e^{-1/\k} ),
\end{align*}
using the stationarity of the $(\b_k)$ under $Q$ in the last inequality. 
The functional central limit theorem for renewal sequences implies
that the first probability on the right vanishes as $\e\ra 0$. The
second term also vanishes as $\e\ra 0$ by the tail decay \eqref{btail} of $\b_1$.  
This finishes the proof of \eqref{UlTcompare} in the case $\k \in (0,1)$. 
%Thus, we obtain that
%\be\label{EdUlT}
% \lim_{\e \ra 0} Q\left( E_\w[ d_t^{M_1}(\U_\e\circ \l_\e, \, \T_\e) ] \geq \eta \right) = 0, \qquad \forall \eta>0, \, \forall t< \infty. 
%\ee
%We claim that this is enough to imply \eqref{UlTcompare}. Indeed, first note that \eqref{EdUlT} implies that 
%\[
% \lim_{\e \ra 0} E_Q\left[ E_\w\left[ d_t^{M_1}(\U_\e\circ \l_\e, \, \T_\e) \wedge 1 \right] \right] = 0, \qquad \forall t<\infty. 
%\]
%Then, the definition of $d_\infty^{M_1}$ and the Fubini's Theorem imply that
%\begin{align*}
%E_Q\left[ E_\w\left[ d_\infty^{M_1}(\U_\e\circ \l_\e, \, \T_\e) \right] \right]
%&= E_Q\left[ E_\w\left[ \int_0^\infty e^{-t} \left( d_t^{M_1}(\U_\e\circ \l_\e, \, \T_\e) \wedge 1 \right) \, dt \right] \right] \\
%&=  \int_0^\infty e^{-t} E_Q\left[ E_\w\left[  d_t^{M_1}(\U_\e\circ \l_\e, \, \T_\e) \wedge 1  \right] \right] dt,
%\end{align*}
%and this last integral vanishes as $\e\ra 0$ by the dominated convergence theorem. This finishes the proof of \eqref{UlTcompare} and thus also the proof of the lemma when $\k \in (0,1)$.  

\begin{figure}
\includegraphics{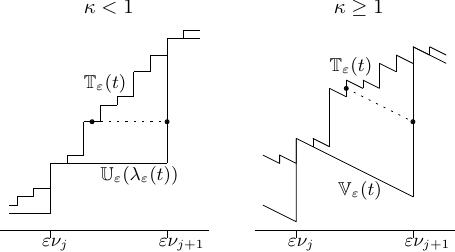}
\caption{A demonstration of the matching of the parameterizations of the completed graphs of $\T_\e$ with the completed graphs of $\U_\e \circ \l_\e$ and  $\V_\e$ when $\k\in(0,1)$ and $\k \in[1,2)$, respectively.\label{MatchingFigure}}
\end{figure}

\textbf{Case II: $\k \in [1,2)$.}
To overcome the difficulty of matching the centering terms of $\U_\e
\circ \l_\e$  and $\T_\e$ we define $\V_\e \in D_\infty$ by 
\[
 \V_\e(t) = 
\e^{1/\k} T_{\nu_{\fl{\l_\e(t)/\e}}} - 
\begin{cases}
t D(1/\e)  & \k = 1 \\
(t/\vp) \e^{-1+1/\k} & \k \in (1,2).  
\end{cases}
\]
$\V_\e$ is defined so that the hitting times portion is the same as in
$\U_\e \circ \l_\e$ while the linear centering is the same as in
$\T_\e$.  

Since the only difference between $\U_\e \circ \l_\e$ and $\V_\e$ is
in the centering term, we have for any $t<\infty$ 
\be\label{UlVcompare}
 d_t^{M_1}( \U_\e \circ \l_\e, \, \V_\e) 
\leq \sup_{t'\leq t} | \l_\e(t') - t'/\bar\nu | 
\begin{cases}
 D'(1/\e) & \k = 1 \\
 \e^{-1+1/\k} \bar\b & \k \in(1,2), 
\end{cases}
\ee
using $D'(1/\e) = \bar\nu D(1/\e)$ when $\k = 1$ and $\bar\b = \bar\nu/\vp$ when $\k\in(1,2)$. 
Recall that the random element of $D_t$, $t'\mapsto \e^{-1/2}(
\l_\e(t') - t'/\bar\nu )$,  converges weakly in
$(D_t,J_1)$ to Brownian motion, which is a continuous process. Every
continuous function in $D_t$ is a continuity point of the mapping
$x\mapsto \sup_{t'\leq t}|x(t')|$ from $D_t$ to $\R$. Therefore, we
can use the continuous mapping theorem to show that the term in the
right hand side of \eqref{UlVcompare} converges to 0 in
$Q$-probability as $\e\ra 0$, by noticing that both $D'(1/\e)$ (when
$\k=1$) and $\e^{-1+1/\k}$ (when $\k\in(1,2)$) are $o(\e^{-1/2})$. 
Therefore, in order to prove \eqref{UlTcompare} it is enough to show
that for every $0<t<\infty$ and $\eta>0$, 
\be\label{dVT}
 \lim_{\e\ra 0} Q\left( E_\w \left[ d_t^{M_1}( \V_\e, \T_\e ) \right]
   > \eta \right) = 0. 
\ee
 
The proof of \eqref{dVT} is very similar to the proof of
\eqref{UlTcompare} when $\k \in(0,1)$. 
Indeed, note that $\V_\e(t) = \T_\e(t)$ whenever $t = \e\nu_j$ for
some $j$. Again, we  arrange
the respective parametric representations of the completed graphs of
the two random functions so that, for $k$ being the largest $j$ so that $\nu_j\leq
t/\e$, both parametric representations are equal, to $\bigl( \e \nu_j,
\e^{1/\k} T_{\nu_j}\bigr)$ at $s_j=j/(k+1), \, j=0,1,\ldots, k$. For
$s_j<s<s_{j+1}$ with $j=0,1,\ldots, k-1$ the two parametric representation can be
chosen in such a way that the line connecting the two corresponding
points is always parallel to the segment, connecting the points 
$(\e \nu_j, \V_\e(\e \nu_j))$ and $(\e \nu_{j+1}, \V_\e(\e \nu_{j+1} - ))$. 
%%%% Incorrect description in prior version
% $\bigl( \e \nu_j, \e^{1/\k} T_{\nu_j}\bigr)$ and $\bigl( \e \nu_{j+1}, \e^{1/\k} T_{\nu_{j+1}}\bigr)$.  
See Figure \ref{MatchingFigure} for a visual
representation of this matching. In this case the distance between the
two corresponding points does not exceed the length of the above
segment, which is shorter than $\e^{1/2}(\nu_{j+1} - \nu_j)$ for $\e$
small enough. As in the case $\k\in(0,1)$, on the interval $(s_k,1]$
we keep horizontal ($u$) coordinates of the two parametric
representations the same. Overall, we obtain the bound 
\[
 d_t^{M_1}( \V_\e,\, \T_\e) \leq \max \left\{ \e^{1/2} \max_{j\leq k}
   (\nu_{j+1} - \nu_j), \, \e^{1/\k} ( T_{\nu_{k+1}} - T_{\nu_k} )
 \right\}.  
\]
From here we proceed as in the case $\k\in(0,1)$ above. 
\end{proof}

\begin{proof}[Proof of Lemma \ref{Expcouple}]
By Lemma \ref{Dt2Dinfty}  it is enough to show that for each
$0<s<\infty$ and $\eta>0$, 
\[
 \lim_{\e\ra\infty} Q\left( E_\w\left[ \sup_{t \leq s} |\U_\e(t) -
 \S_\e(t)| \right] \geq \eta \right) = 0.  
\]
Since both $\U_\e$ and $\S_\e$ are piecewise linear with the same
slope between times $t\in \e\Z$, 
%if $k\e < t \leq (k+1)\e$ then $|\U_\e(t) - \S_\e(t)| \leq |\U_\e(k\e) - \S_\e(k\e)| \vee |\U_\e((k+1)\e) - \S_\e((k+1)\e)|$. Thus, 
\[
 \sup_{t \leq s} |\U_\e(t) - \S_\e(t)| 
= \e^{1/\k} \max_{k\leq \lceil s/\e \rceil } \left|T_{\nu_{k}} -
\sum_{i=1}^{k} \b_i \tau_i \right|. 
\]
Now, it is easy to see that $M_k = T_{\nu_{k}} - \sum_{i=1}^{k} \b_i
\tau_i$ is a martingale under $P_\w$. Therefore, by the
Cauchy-Schwartz and  $L^p$-maximum inequalities for martingales,
\begin{align*}
 E_\w\left[ \sup_{t \leq s} |\U_\e(t) - \S_\e(t)| \right] 
&\leq \left( E_\w\left[ \sup_{t \leq s} |\U_\e(t) - \S_\e(t)|^2 \right] \right)^{1/2} \\
&\leq \e^{1/\k} \left( 4 E_\w \left[ M_{\cl{s/\e}}^2 \right] \right)^{1/2} \\
&= 2 \e^{1/\k} \left( \Var_\w\left( T_{\nu_{\cl{s/\e}}} - \sum_{i=1}^{\cl{s/\e}} \b_i \tau_i \right) \right)^{1/2}.
\end{align*}
Therefore, 
\be
 Q\left( E_\w\left[ \sup_{t \leq s} |\U_\e(t) - \S_\e(t)| \right] \geq \eta \right)
\leq Q\left( 4 \e^{2/\k} \Var_\w\left( T_{\nu_{\cl{s/\e}}} - \sum_{i=1}^{\cl{s/\e}} \b_i \tau_i \right)  \geq \eta^2 \right). \label{UeSesupub}
\ee
In the proof of Lemma 4.4 in \cite{psWQLTn}, a natural coupling of $\tau_i$ and $T_{\nu_i} - T_{\nu_{i-1}}$ was constructed so that for any $\eta>0$,
%However, in the proof of Lemma 4.4 in \cite{psWQLTn} it was shown that
%for any $\eta>0$, 
\[
 \lim_{n\ra\infty} Q\left( n^{-2/\k} \Var_\w\left( T_{\nu_{n}} - \sum_{i=1}^{n} \b_i \tau_i \right)  \geq \eta \right) = 0.
\]
Applying this to \eqref{UeSesupub} completes the proof of the lemma. 
\end{proof}

\section{Weak weak quenched limits for $\S_\e$}\label{ExpWQL}

In this section we prove Proposition \ref{WQLSn}. 
For any environment $\w$ and $\e>0$,  define a point process by 
\[
 N_{\e} = \sum_{i\geq 1} \d_{(\e^{1/\k} \b_i,\, \e i)}. 
\]
We view $N_{\e}$ as a random element of $\mathcal{M}_p((0,\infty]
\times [0,\infty))$. 
%Recalling the definition of $W(\zeta,\vec\tau)$ in \eqref{Wdef}, we see that $\e^{1/\k} \sum_{i = 1}^{t/\e} \b_i \tau_i  = W(N_{\e,\w},\vec\tau)(t)$ for all $t\geq 0$. 
%Thus, recalling the definition of $\S_\e$ in \eqref{Sdef} we obtain that 
%\[
% \S_\e(t) = 
%\begin{cases}
% W(N_{\e,\w},\vec\tau)(t) & \k \in (0,1) \\
% W(N_{\e,\w},\vec\tau)(t) - t D'(1/\e)  &  \k = 1 \\
% W(N_{\e,\w},\vec\tau)(t) - \bar\b \e^{-1+1/\k} t & \k \in (1,2)
%\end{cases}
%\]
%Therefore, letting $y_{\k,\e} \in D_\infty$ for $\k\in[1,2)$ be defined by $y_{1,\e}(t) = D'(1/\e)t$ and $y_{\k,\e}(t) = \bar\b \e^{-1+1/\k} t$ for $\k \in (1,2)$, we have that 
Recalling the definitions of $\H$ in \eqref{Hdef} and $\S_\e$ in
\eqref{Sdef}, we see that the quenched law of $\S_\e$ satisfies 
\be\label{sewHNew}
 s_{\e} = 
\begin{cases}
 \H(N_{\e}) & \k \in (0,1)\\
 \H(N_{\e}) * \lm(-D'(1/\e)) & \k =1\\
 \H(N_{\e}) * \lm(-\bar\b \e^{-1+1/\k}) & \k \in (1,2). 
\end{cases}
\ee
The key to the proof of Proposition \ref{WQLSn} is the following lemma
which shows weak convergence of the point process $N_{\e}$.  
\begin{lem}\label{Newlim}
Under the measure $Q$, as $\e\to 0$, the point process $N_{\e}$
converges  weakly in the space $\mathcal{M}_p((0,\infty]
\times [0,\infty))$ to a non-homogeneous Poisson point process $N_{\l,\k}$
with intensity measure $\l x^{-\k-1} dx \, dt$. Moreover, 
$\l = \tc \k$, where  $\tc$ is the tail constant in \eqref{btail}. 
\end{lem}
\begin{proof}
The idea of the proof is similar to that of the proof of Proposition
5.1 in \cite{psWQLTn}. It was shown in the above proof that for
$0<\e<1$ there is a stationary under $Q$ sequence of random variables
$\bigl( \b_i^{(\e)}, \, i=1,2,\ldots\bigr)$ on $\Omega$ such that
$\b_i^{(\e)}$ and $\b_j^{(\e)}$ are independent if $|i-j| >
\e^{-1/2}$, and such that, for some $C,C'>0$, 
\begin{equation} \label{e:bcutoff}
Q\left( \left| \b_1 - \b_1^{(\e)}\right| > e^{-\e^{-1/4}} \right) \leq C e^{-C'
   \e^{-1/2}}, \ 0<\e<1\,.
\end{equation}
We define an approximating point process by
\[
 N^{(1)}_{\e} = \sum_{i\geq 1} \d_{(\e^{1/\k} \b^{(\e)}_i,\, \e i)}, \ 0<\e<1\,,
\]
and proceed by proving the convergence
\begin{equation} \label{e:approx.conv}
 N^{(1)}_{\e} \Lra N_{\l,\k} \ \text{weakly in $\mathcal{M}_p((0,\infty]
\times [0,\infty))$}
\end{equation}
as $\e\to 0$, under the measure $Q$. 

We start by considering measurable functions $f:\, (0,\infty] \times
[0,\infty) \ra \R_+$ of the form 
\begin{equation} \label{e:simple.f}
f(x,t) = \sum_{i=1}^k f_i(x)\one_{[a_{i-1},a_i)}(t), \ (x,t)\in (0,\infty] \times
[0,\infty)\,,
\end{equation} 
where $k=1,2,\ldots$, $f_i:\, (0,\infty] \ra \R_+, \, i=1,\ldots, k$
are continuous functions that vanish for 
all $0<x<\d$ for some $\d>0$, and are Lipschitz on the interval
$(\d,\infty)$, and $0=a_0<a_1<\ldots <a_k<\infty$. We will prove that
for such a function, 
\be \label{e:Laplace.conv.1}
 \lim_{\e\ra 0} E_Q\left[ e^{-N^{(1)}_{\e}(f) } \right] = \exp
 \left\{ - \sum_{i=1}^k (a_i-a_{i-1})\int_0^\infty  (1-e^{-f_i(x)}) \l
   x^{-\k-1} \, dx \right\}.    
\ee

To this end we define, as in \cite{psWQLTn}, for $0<\tau<1$, 
$$
K_\e(\tau) = \text{card}\bigl\{ i=1,\ldots, \lfloor a_k/\e\rfloor:\
\text{both} \ \b_i^{(\e)}>\d \e^{-1/\k} \ \text{and} \ \b_j^{(\e)}>\d \e^{-1/\k} 
$$
$$
\text{for some} \
i+1\leq j\leq i+\tau/\e, \, j\leq a_k/\e.  \bigr\};
$$
as in the above reference we have
\begin{equation} \label{e:small.K}
\lim_{\tau\to 0} \limsup_{\e\to 0} Q(K_\e(\tau)>0)=0.
\end{equation}

Define random sets 
$$
D_\e^{(j)} =  \{ a_{j-1}/\e<i<a_j/\e:\, \b_i^{(\e)}>\d \e^{-1/\k}\}, \
j=1,\ldots, k,
$$
so that 
\begin{align*}
E_Q\left[ e^{-N^{(1)}_{\e}(f) } \right] & =
E_Q\exp\left\{ - \sum_{j=1}^k \sum_{i\in D_\e^{(j)}}
f_j\bigl(\e^{1/\k}\b_i^{(\e)}\bigr)\right\}  \\
&= E_Q\left[ \exp\left\{ - \sum_{j=1}^k \sum_{i\in D_\e^{(j)}}
f_j\bigl(\e^{1/\k}\b_i^{(\e)}\bigr)\right\} \one\bigl(
K_\e(\tau)=0\bigr)\right] \\
&+ E_Q\left[ \exp\left\{ - \sum_{j=1}^k \sum_{i\in D_\e^{(j)}}
f_j\bigl(\e^{1/\k}\b_i^{(\e)}\bigr)\right\} \one\bigl(
K_\e(\tau)>0\bigr)\right] \\
&:= H_\e^{(1)} + H_\e^{(2)}. 
\end{align*}
It follows from \eqref{e:small.K} that the term $H_\e^{(2)}$ is
negligible as $\e\to 0$ and then $\tau\to 0$. Furthermore, given the
event $\{  K_\e(\tau)=0\}$,  for a fixed
$0<\tau<1$ and $\e$ small enough, the points in the random set 
$D_\e:=\cup_j D_\e^{(j)}$ are separated by more than $\e^{-1/2}$, so
that, given also the set $D_\e$, the random variables $\b_i^{(\e)}, \,
i\in D_\e$ are independent, each one with the conditional distribution
of $\b_1^{(\e)}$ given $\b_1^{(\e)}>\d\e^{-1/\k}$. Since for every
$j=1,\ldots, k$, 
$$
 E_Q \Bigl( \exp\bigl\{
  -f_j\bigl(\e^{1/\k}\b_1^{(\e)}\bigr)\bigr\} \big|\b_1^{(\e)}>\d
  \e^{-1/\k}\Bigr) \to \int_1^\infty e^{-f_j(\d t)}\k t^{-(\k+1)}\,
  dt, 
$$
the claim \eqref{e:Laplace.conv.1} will follow once we check that 
\begin{align*} 
& \exp\bigl\{ -C_0 \delta^{-\k}
\sum_{j=1}^k(a_j-a_{j-1})(1-\alpha_j)\bigr\} \leq \lim_{\tau\to 
    0}\liminf_{\e\to 0} E_Q \left( \prod_{j=1}^k \alpha_j^{{\rm
        card} \, D_\e^{(j)}}\Big| 
K_\e(\tau)=0\right) \label{e:mgf.card} \\
 = & \lim_{\tau\to 0}\limsup_{\e\to 0} E_Q \left(
\prod_{j=1}^k \alpha_j^{{\rm  card} \, D_\e^{(j)}}
\Big| K_\e(\tau)=0\right)
\leq  \exp\bigl\{ -C_0 \delta^{-\k}
\sum_{j=1}^k(a_j-a_{j-1})(1-\alpha_j)\bigr\}  
\end{align*}
for any $0<\alpha_j<1, \, j=1,\ldots, k$. This, however, can be proved
in the same way as (48) was proved in \cite{psWQLTn}. 

In order to prove  weak convergence in \eqref{e:approx.conv}, it is
enough to prove that for any Lipschitz continuous function $f:\, (0,\infty] \times
[0,\infty) \ra \R_+$ with support in $[\d,\infty]\times [0,a]$ for
some $0<\d,a<\infty$, 
\be \label{e:Laplace.conv.2}
 \lim_{\e\ra 0} E_Q\left[ e^{-N^{(1)}_{\e}(f) } \right] = \exp
 \left\{ - \int_0^\infty \int_0^\infty (1-e^{-f(x,t)}) \l
   x^{-\k-1} \, dx\, dt \right\};    
\ee
see \cite{rEVRVPP} and Remark 5.2 in \cite{psWQLTn}. To this end, for
$m=1,2,\ldots$ we define
$$
f_j(x) = f(x,ja/m), \, x\in (0,\infty], \, j=1,\ldots, m,
$$
and 
$$
\tilde f(x,t) = \sum_{j=1}^k f_j(x)\one_{[(j-1)a/m,ja/m)}(t), \ (x,t)\in (0,\infty] \times
[0,\infty)\,.
$$
Note that $|f(x,t)-\tilde f(x,t)|\leq La/m$ for all finite $(x,t)$,
where $L$ is the Lipschitz constant of $f$. Therefore, 
$$
\left| E_Q\left[ e^{-N^{(1)}_{\e}(f) } \right] - E_Q\left[
  e^{-N^{(1)}_{\e}(\tilde f) } \right]\right|
\leq \frac{La}{m} E_Q \left[  N^{(1)}_{\e}\bigl( [\delta,\infty]\times
[0,a]\bigr) \right].
$$
Notice that, by stationarity,
$$
E_Q \left[  N^{(1)}_{\e}\bigl( [\delta,\infty]\times
[0,a]\bigr)\right] \leq a\e^{-1} Q\bigl( \b_1^{(\e)}>\d\e^{-1/\k}\bigr),
$$
which, by \eqref{btail} and \eqref{e:bcutoff}, remains bounded as $\e\to 0$. Since the
function $\tilde f$ is of the type \eqref{e:simple.f}, it follows from
\eqref{e:Laplace.conv.1} that 
$$
\lim_{\e\ra 0} E_Q\left[ e^{-N^{(1)}_{\e}(\tilde f) } \right] = \exp
 \left\{ - \int_0^\infty \int_0^\infty (1-e^{-\tilde f(x,t)}) \l
   x^{-\k-1} \, dx\, dt \right\}. 
$$
This proves \eqref{e:Laplace.conv.2} (and, hence, also
\eqref{e:approx.conv}). It follows by \eqref{e:bcutoff} and the
Lipschitz property that for any function $f$ as in
\eqref{e:Laplace.conv.2} we also have 
$$
\lim_{\e\ra 0} E_Q\left[ e^{-N_{\e}(f) } \right] = \exp
 \left\{ - \int_0^\infty \int_0^\infty (1-e^{-f(x,t)}) \l
   x^{-\k-1} \, dx\, dt \right\}.
$$
As before, this establishes the weak convergence stated in the lemma. 
\end{proof}

We would like to use the representation \eqref{sewHNew} of $s_{\e,}$
and the fact that $N_{\e} \Lra N_{\l,\k}$ to obtain a weak limit for
$s_{\e}$ as a random element of $\mathcal{M}_1(D_\infty)$.  
%the fact that $N_{\e,\w} \Lra N_{\l,\k}$ for some $\l>0$ to conclude that $s_{\e,\w}$ converges in distribution on the space $\mathcal{M}_1(D_\infty)$. 
Unfortunately, the function $\H$ is not continuous and so we need the
following lemma which shows that the truncated function $\H_\d$ is
``almost continuous''.  
\begin{lem}\label{Hdcont}
Define subsets $C_{\d}, E \subset \mathcal{M}_p( (0,\infty]\times [0,\infty) )$ by
\[
 C_\d = \{ \zeta: \, \zeta(\{\d,\infty\} \times [0,\infty)) = 0 \}
\]
and 
\[
 E = \{ \zeta: \, \zeta((0,\infty] \times \{t\}) \leq 1, \forall t\in (0,\infty) \} \cap \{ \zeta: \zeta( (0,\infty] \times \{0\}) =  0 \}.  
\]
Then $\H_\d$ is continuous on $C_{\d} \cap E$. 
\end{lem}
\begin{proof}
Suppose that $\zeta_n \ra \zeta \in C_{\d}\cap E$. 
We will couple the paths $W_\d(\zeta_n,\vec\tau)$ and
$W_\d(\zeta,\vec\tau)$ by using the same sequence $\vec\tau$ of i.i.d.
standard exponential random variables. Using this coupling we will show that 
$ \lim_{n\ra\infty} W_\d(\zeta_n,\vec\tau) = W_\d(\zeta,\vec\tau)$, $\Pv_\tau$-a.s.
Since almost sure convergence implies weak convergence, $H_\d(\zeta_n) \ra \H_\d(\zeta)$. 

To prove that $W_\d(\zeta_n,\vec\tau)$ converges a.s. to
$W_\d(\zeta,\vec\tau)$ it will be enough to show that  for every
$0<s<\infty$ such that $W_\d(\zeta,\vec\tau)$  is continuous at $s$,
and for every realization $\vec\tau$ with finite values, 
\be\label{Wdlim}
 \lim_{n\ra\infty} d_s^{J_1}( W_\d(\zeta_n,\vec\tau),
 W_\d(\zeta,\vec\tau)) = 0. 
\ee
To this end, take $s$ as above. Then $\zeta([\d,\infty] \times \{s\}) = 0$. 
The assumption that $\zeta \in E$ implies that we may order the atoms of $\zeta$ in $[\d,\infty]\times [0,s]$ so that for $M = \zeta([\d,\infty]\times [0,s])$ we have 
\[
 \zeta\bigl(\cdot \cap ([\d,\infty] \times [0,s] )\bigr) = \sum_{i =
   1}^M \d_{(x_i,t_i)}(\cdot), \quad \text{with $0<t_1 < t_2 < \ldots < t_M < s$.}
\]
Similarly, we can order the atoms of $\zeta_n$ in $[\d,\infty]\times
[0,s]$ so that for $M_n = \zeta_n([\d,\infty]\times[0,s])$ we have 
\[
 \zeta_n\bigl((\cdot \cap ([\d,\infty] \times [0,s] )\bigr) = \sum_{i
   = 1}^{M_n} \d_{(x_i^{(n)},t_i^{(n)})}, \quad \text{with $0 \leq t_1^{(n)} \leq
 t_2^{(n)} \leq \ldots \leq t_{M_n}^{(n)} \leq s$.} 
\]
The vague convergence of $\zeta_n$ to $\zeta$ and the fact that
$\zeta$ has no atoms on the boundary of $[\d,\infty] \times [0,s]$,
imply that for $n$ large enough $M_n = M$ and  
\be\label{vcpoints}
 \lim_{n\ra\infty} \max_{i\leq M} \left( |x_i^{(n)}-x_i| \vee |t_i^{(n)}-t_i|\right) = 0. 
\ee
Therefore, for $n$ sufficiently large, $0<t_1^{(n)} < t_2^{(n)} < \ldots < t_M^{(n)} < s$. 
For such $n$ we define a time-change function $\l_n^s$ of the interval
$[0,s]$ by $\l_n^s(0) = 0$, $\l_n^s(s) = s$, $\l_n^s(t_i) = t_i^{(n)}$
for all $i\leq M$, and extend it everywhere else by linear
interpolation.  Then,
\[
 \sup_{t\leq s} \left| \l_n^s(t) - t \right| = \max_{i\leq M} |t_i^{(n)} - t_i |,
\]
and, since $W_\d(\zeta_n,\vec\tau)$ and $W_\d(\zeta,\vec\tau)$ are
constant between jumps, 
\[
 \sup_{t\leq s} \left| W_\d(\zeta_n,\vec\tau)(\l_n^s(t)) - W_\d(\zeta,\vec\tau)(t) \right| = \max_{j\leq M} \left| \sum_{i=1}^j \left(x_i^{(n)} - x_i \right) \tau_i \right| \leq \sum_{i=1}^M \left|x_i^{(n)} - x_i \right| \tau_i. 
\]
Therefore, for $n$ sufficiently large, 
\[
 d_s^{J_1}\left( W_\d(\zeta_n,\vec\tau), W_\d(\zeta,\vec\tau) \right)
 \leq  \max \left\{  \max_{i\leq M} | t_i^{(n)} - t_i | , \sum_{i=1}^M
   \left|x_i^{(n)} - x_i \right| \tau_i  \right\}, 
\]
which vanishes as $n\ra\infty$ by \eqref{vcpoints}. This completes the
proof of \eqref{Wdlim} and thus of the lemma.  
\end{proof}

The relationship between $s_{\e}$ and $N_{\e}$ in \eqref{sewHNew} and
Lemma \ref{Hdcont} will allow us now to complete the proof of Proposition
\ref{WQLSn}.  
\begin{proof}[Proof of Proposition \ref{WQLSn}]
For  $\d>0$ we define a truncated version of $\S_\e$ by 
\[
 \S_{\e,\d}(t) = \e^{1/\k} \sum_{i=1}^{\lfloor t/\e\rfloor} \b_i
 \tau_i \ind{\e^{1/\k} \b_i > \d} - t \gamma_{\k,\e,\d}, \ t\geq 0\,,
\]
%\[
% \S_{\e,\d}(t) = 
%\begin{cases}
% \e^{1/\k} \sum_{i=1}^{t/\e} \b_i \tau_i \ind{\e^{1/\k} \b_i > \d} & \k \in (0,1) \\
% \e \sum_{i=1}^{t/\e} \b_i \tau_i \ind{\e^{1/\k} \b_i > \d} - t \gamma_{1,\e,\d} & \k =1 \\ 
% \e^{1/\k} \sum_{i=1}^{t/\e} \b_i \tau_i \ind{\e^{1/\k} \b_i > \d} - t \gamma_{\k,\e,\d} & \k \in (1,2), 
%\end{cases}
%\]
with 
\be\label{gdef}
 \gamma_{\k,\e,\d} = 
\begin{cases}
 0 & \k \in (0,1) \\
 E_Q\left[\b_1 \ind{\e \b_1 \in (\d,1]}\right] & \k =1 \\
 \e^{1/\k-1} E_Q\left[ \b_1 \ind{\e^{1/\k} \b_1 > \d} \right] & \k \in (1,2). 
\end{cases}
\ee
%\[
% \S_{\e,\d}(t) = 
%\begin{cases}
% \e^{1/\k} \sum_{i=1}^{t/\e} \b_i \tau_i \ind{\e^{1/\k} \b_i > \d} & \k \in (0,1) \\
% \e \sum_{i=1}^{t/\e} \b_i \tau_i \ind{\e^{1/\k} \b_i > \d} - t E_Q\left[\b_1 \ind{\e \b_1 \in (\d,1]}\right] & \k =1 \\ 
% \e^{1/\k} \sum_{i=1}^{t/\e} \b_i \tau_i \ind{\e^{1/\k} \b_i > \d} - t \e^{1/\k-1} E_Q\left[ \b_1 \ind{\e^{1/\k} \b_1 > \d} \right] & \k \in (1,2), 
%\end{cases}
%\]
Then the quenched law of $\S_{\e,\d}$ is 
$s_{\e,\d} = \H_\d(N_{\e}) * \lm(-\gamma_{\k,\e,\d})$. 

If $N_{\l,\k}$ is the Poisson point process as in the statement of
Lemma \ref{Newlim}, then $\Pv(N_{\l,\k} \in C_\d \cap E) = 1$ for any
$\d>0$. Thus, Lemma \ref{Newlim}, Lemma \ref{Hdcont}, and the
continuous mapping
theorem imply that, under the measure $Q$, $ \H_\d(N_{\e}) \Lra 
\H_\d(N_{\l,\k})$, where $\l = C_0 \k$.  
Also, by \eqref{btail} and Karamata's theorem, 
\[
 \lim_{\e \ra 0} \gamma_{\k,\e,\d} = 
\begin{cases}
 C_0 \ln(1/\d) & \k = 1\\
 \frac{C_0 \k}{\k-1} \d^{-\k+1} & \k \in (1,2).
\end{cases}
\]
Since the mapping from $\mathcal{M}_1(D_\infty) \times \R$ to
$\mathcal{M}_1(D_\infty)$ defined by $(\mu,\gamma)\mapsto \mu *
\lm(\gamma)$ is continuous, we conclude that, under the measure $Q$, 
%Therefore, we have for $\l= C_0 \k$ that 
\be\label{SedWQL}
 s_{\e,\d} 
%= \H_\d(N_{\e,\w}) * \lm(-\gamma_{\k,\e,\d}) %%%%(when $\k\in[1,2)$)
\Lra 
\begin{cases}
 \H_\d(N_{\l,\k}) & \k\in(0,1) \\
 \H_\d(N_{\l,\k}) * \lm(-\l \ln(1/\d) ) & \k = 1 \\
 \H_\d(N_{\l,\k}) * \lm(-\l \d^{-\k+1}/(\k-1)) & \k \in (1,2).  
\end{cases}
\ee

To relate \eqref{SedWQL} to a limit statement about $s_{\e,}$ we use
\cite[Theorem 3.2]{bCOPM}. To this end, it is enough to show that the
limit in  $\mathcal{M}_1\bigl( (D_\infty,d_\infty^{M_1})\bigr)$
\be\label{aslim}
 \lim_{\d\ra\infty} 
\begin{cases}
 \H_\d(N_{\l,\k}) & \k\in(0,1) \\
 \H_\d(N_{\l,\k}) * \lm(-\l \ln(1/\d) ) & \k = 1 \\
 \H_\d(N_{\l,\k}) * \lm(-\l \d^{-\k+1}/(\k-1)) & \k \in (1,2),  
\end{cases}
\qquad \text{exists } \Pv_\tau \text{-a.s.}
\ee
and 
\be\label{iplim}
 \lim_{\d\ra 0} \limsup_{\e\ra 0} Q\left( \rho^{M_1}(s_{\e,\d}, \,
   s_{\e}) \geq \eta \right) = 0, \quad \forall \eta>0.  
\ee
%Since $\rho^{M_1}(s_{\e,\d,\w}, \, s_{\e,\w})^2 \leq E_\w[ d_\infty^{M_1}( \S_{\e,\d}, \, \S_\e ) ] \leq E_\w[ \sup_{t\leq s} \left| \S_{\e,\d}(t) - \S_\e(t) \right| ] + e^{-s}$ for any $s<\infty$, \eqref{iplim} will follow from the stronger statement
As in the case of Lemma \ref{WAEcor}, \eqref{iplim} will follow from
following, stronger, statement: for every $0<s<\infty$, 
\be\label{iplim2}
 \lim_{\d\ra 0} \limsup_{\e\ra 0} Q\left( E_\w \left[ \sup_{t\leq s} \left| \S_{\e,\d}(t) - \S_\e(t) \right| \right] \geq \eta \right) = 0, \quad \forall \eta>0.
\ee

Therefore, to complete the proof of Proposition \ref{WQLSn}, it
remains only to prove \eqref{aslim} and \eqref{iplim2}.  
We divide the proof of these statements into two cases: $\k\in(0,1)$ and $\k \in [1,2)$. 

\subsection{Case I: $\k \in (0,1)$}
To prove \eqref{aslim} we let $F_1 \subset \mathcal{M}_p( (0,\infty]
\times [0,\infty))$ be defined by 
%\[
% F_1 = \left\{ \zeta : \, \iint x \ind{t\leq s} \, \zeta(dx, \, dt) < \infty, \, \forall s<\infty \right\}.
%\]
\[
 F_1 = \left\{ \zeta = \sum_{i\geq 1} \d_{(x_i,t_i)} \, : \, \sum_{i\geq 1} x_i \ind{t_i \leq t} < \infty, \, \forall t<\infty \right\}.
\]
(Note that on the set $F_1$, the sum in the definition of
$W(\zeta,\vec\tau)$ is $\Pv_\tau$-a.s.\ finite.) Since $\Pv(N_{\l,\k}
\in F_1) = 1$ when $\k \in (0,1)$, it will be enough to show that
$H_\d(\zeta) \ra \H(\zeta)$ as $\d\ra 0$ for any $\zeta \in F_1$.  
%We claim that for any $\zeta \in F_1$, $\lim_{\d\ra 0} \H_\d(\zeta) = \H(\d)$ in the space $\mathcal{M}_1(D_\infty)$.  
Fix $\zeta = \sum_{i\geq 1} \d_{(x_i,t_i)} \in F_1$. For $0<s<\infty$
the obvious coupling of $W(\zeta,\vec\tau)$ and $W_\d(\zeta,\vec\tau)$
gives that  
\be\label{supWWd}
 \sup_{t\leq s} | W(\zeta,\vec\tau)(t) - W_\d(\zeta,\vec\tau)(t) |  
= \sup_{t\leq s} \left| \sum_{i\geq 1} x_i \tau_i \ind{x_i\leq \d, \, t_i \leq t} \right| 
= \sum_{i\geq 1} x_i \tau_i \ind{x_i\leq \d, \, t_i \leq s}. %= \iint x \ind{x\leq \d, \, t\leq s} \, \zeta(dx,\, dt).  
\ee
%Therefore, 
%\[
% \rho( \H_\d(\zeta), \, \H(\zeta) )^2 \leq e^{-s} + \Ev\left[ \sup_{t\leq s} | W(\zeta,\vec\tau)(t) - W_\d(\zeta,\vec\tau)(t) | \right] = e^{-s} + \sum_{i\geq 1} x_i \ind{x_i\leq \d, \, t_i \leq s}. 
%\]
Since $\zeta \in F_1$,  finiteness of the mean of an exponential
random variable  shows that the sum on the right is finite with
probability one for any $\d>0$. Letting  $\d\ra 0$ the dominated
convergence theorem shows that $W_\d(\zeta,\vec\tau)$ converges almost
surely to $W(\zeta,\vec\tau)$ in the space $D_s$ in the uniform
metric, hence also in the $M_1$-metric, for any
$0<s<\infty$. Therefore, $W_\d(\zeta,\vec\tau)$ converges almost
surely to $W(\zeta,\vec\tau)$ in $D_\infty$ as $\d\ra 0$ and, since
a.s. convergence implies convergence in distribution, 
$\H_\d(\zeta)$ converges to $\H(\zeta)$ in the space
$\mathcal{M}_1\bigl( (D_\infty,d_\infty^{M_1})\bigr)$ as $\d\ra
0$. This proves \eqref{aslim}. Further, since  $W_\d(N_{\e},\vec\tau)
= \S_{\e,\d}$ and $W(N_{\e},\vec\tau) = \S_{\e}$, we have by
  \eqref{supWWd} with $\zeta = N_{\e}$, 
\[
 E_\w \left[ \sup_{t\leq s} \left| \S_{\e,\d}(t) -  \S_\e(t) \right| \right] 
= 
E_\w \left[ \sum_{i=1}^{\lfloor s/\e\rfloor} \e^{1/\k} \b_i \tau_i
  \ind{\e^{1/\k} \b_i \leq \d} \right] =  
\e^{1/\k} \sum_{i=1}^{\lfloor s/\e\rfloor} \b_i \ind{\e^{1/\k} \b_i \leq \d}.
\]
By Chebyshev's inequality and  stationarity of $\b_i$ under $Q$, 
\begin{align*}
 Q\left( E_\w \left[ \sup_{t\leq s} \left| \S_{\e,\d}(t) -  \S_\e(t) \right| \right] \geq \eta \right) 
&= Q\left( \e^{1/\k} \sum_{i=1}^{\lfloor s/\e\rfloor} \b_i \ind{\e^{1/\k} \b_i \leq \d} \geq \eta \right)\\
%&\leq \frac{\e^{1/\k}}{\eta} E\left[ \sum_{i=1}^{s/\e} \b_i \ind{\e^{1/\k} \b_i \leq \d} \right] \\
&\leq \frac{s \e^{1/\k - 1}}{\eta} E_Q\left[ \b_1 \ind{\e^{1/\k} \b_1 \leq \d} \right].
\end{align*}
Karamata's theorem and \eqref{btail} imply that $E_Q\left[\b_i \ind{\e^{1/\k} \b_i \leq \d}\right] \sim \tc \k/(1-\k) \d^{1-\k}\e^{1-1/\k}$ as $\e\ra 0$. Therefore, 
\[
 \lim_{\d\ra 0} \limsup_{\e\ra 0} Q\left( E_\w \left[ \sup_{t\leq s} \left| \S_{\e,\d}(t) -  \S_\e(t) \right| \right] \geq \eta \right) \leq \lim_{\d\ra 0} \frac{s \tc \k}{\eta(1-\k)} \d^{1-\k} = 0. 
\]
This proves \eqref{iplim2}.

\subsection{Case II: $\k \in [1,2)$}
To prove \eqref{aslim}, note that the right hand side of \eqref{aslim}
is the law (with respect to $\Pv_\tau$) of the random element of
$D_\infty$ 
\[
t \mapsto  W_\d(N_{\l,\k}, \vec\tau)(t) - 
\begin{cases}
 \l t \log(1/\d) & \k = 1\\
 \frac{\l \d^{1-\k} t}{\k-1} & \k \in (1,2). 
\end{cases}
\]
It was shown in the proof of Corollary \ref{AveragedTn} that this
random element converges almost surely, in the uniform metric, 
under the joint law $\Pv \times \Pv_\tau$ of $(N_{\l,\k},\vec\tau)$.
Therefore, convergence takes place in the $M_1$-metric as well. 
Fubini's theorem implies that the convergence also holds
$\Pv_\tau$-a.s.\ for almost every realization of the point process
$N_{\l,\k}$. Once again, a.s. convergence implies convergence in
distribution, so \eqref{aslim} holds.

To prove \eqref{iplim2}, note that the definition of
$\gamma_{\k,\e,\d}$ in \eqref{gdef} implies that
\begin{align}
\sup_{t\leq s} & | \S_{\e,\d}(t) - \S_\e(t) | \nonumber \\
%&= \sup_{t\leq s} \left| \left( \e^{1/\k} \sum_{i=1}^{t/\e} \b_i \tau_i \ind{\e^{1/\k} \b_i > \d } - \gamma_{\k,\e, \d} t \right) - 
%\left( \e^{1/\k} \sum_{i=1}^{t/\e} \b_i \tau_i - \bar\b \e^{-1+1/\k} t \right) \right| \nonumber \\
&=  \sup_{ t\leq s} \left| \e^{1/\k} \sum_{i=1}^{\lfloor t/\e\rfloor} \b_i \tau_i \ind{\e^{1/\k} \b_i \leq \d }  - E_Q[\b_1 \ind{\e^{1/\k} \b_1 \leq \d}] \e^{-1+1/\k} t \right| \nonumber\\
&\leq \sup_{ t\leq s} \left| \e^{1/\k} \sum_{i=1}^{\lfloor t/\e\rfloor} \b_i (\tau_i-1) \ind{\e^{1/\k} \b_i \leq \d} \right| \label{taucentered} \\
&\quad + \sup_{ t\leq s} \left| \e^{1/\k} \sum_{i=1}^{\lfloor
    t/\e\rfloor} \left\{ \b_i \ind{\e^{1/\k} \b_i \leq \d} - E_Q[\b_1
    \ind{\e^{1/\k} \b_1 \leq \d}] \right\} \right| + \e^{1/\k}
E_Q[\b_1 \ind{\e^{1/\k} \b_1 \leq \d}],  \nonumber
%\label{betasum}
\end{align}
where the last term comes from rounding in the number of terms in the
sum. This terms is, clearly, bounded by $\d$. 

For $\b_i$ fixed, the sum inside the supremum in the first term in
\eqref{taucentered} 
is a sum of independent, zero-mean random variables.  
Thus, by the Cauchy-Schwartz and  $L^p$-maximum inequalities for
martingales, 
\begin{align*}
E_\w \left[ \sup_{t\leq s} \left| \e^{1/\k} \sum_{i=1}^{\lfloor t/\e\rfloor} \b_i (\tau_i-1) \ind{\e^{1/\k} \b_i \leq \d} \right| \right] 
%&\leq  \left( \Ev\left[ \sup_{t\leq s} \left| \e^{1/\k} \sum_{i=1}^{t/\e} \b_i (\tau_i-1) \ind{\e^{1/\k} \b_i \leq \d} \right|^2 \right] \right)^{1/2} \\
&\leq \left( 4 E_\w \left| \e^{1/\k} \sum_{i=1}^{\lfloor s/\e\rfloor} \b_i (\tau_i-1) \ind{\e^{1/\k} \b_i \leq \d} \right|^2  \right)^{1/2} \\
&=  2 \e^{1/\k} \left(\sum_{i=1}^{\lfloor s/\e\rfloor} \b_i^2 \ind{\e^{1/\k} \b_i \leq \d} \right)^{1/2} 
\end{align*}
Therefore, for $\eta>0$ fixed and $\d$ sufficiently small we have
\begin{align}
& Q\left( E_\w \left[ \sup_{t\leq s} | \S_{\e,\d}(t) - \S_\e(t) | \right] \geq \eta \right) \nonumber \\
& \leq Q\left(  \e^{2/\k} \sum_{i=1}^{\lfloor s/\e\rfloor} \b_i^2
  \ind{\e^{1/\k} \b_i \leq \d}  \geq \eta^2/36 \right) \nonumber \\
%\label{b2bound} \\
& \quad + Q\left( \sup_{t\leq s} \left| \e^{1/\k} \sum_{i=1}^{\lfloor t/\e\rfloor} \left\{ \b_i \ind{\e^{1/\k} \b_i \leq \d} - E_Q[\b_1 \ind{\e^{1/\k} \b_1 \leq \d}] \right\} \right| \geq \eta/3 \right). \label{bprocess} 
% \\
%&\quad \leq \frac{36 s \e^{2/\k - 1}}{\eta^2}  E_Q\left[ \b_i^2 \ind{\e^{1/\k} \b_i \leq \d}  \right] \label{b2bound} \\
%& \qquad + Q\left( \sup_{t\leq s} \left| \e^{1/\k} \sum_{i=1}^{t/\e} \left\{ \b_i \ind{\e^{1/\k} \b_i \leq \d} - E_Q[\b_1 \ind{\e^{1/\k} \b_1 \leq \d}] \right\} \right| \geq \eta/3 \right). \label{bprocess}
\end{align}
Notice that 
\begin{align*}
 \limsup_{\e\ra 0} Q\left(  \e^{2/\k} \sum_{i=1}^{\lfloor s/\e\rfloor} \b_i^2 \ind{\e^{1/\k} \b_i \leq \d}  \geq \eta^2/36 \right) 
&\leq \limsup_{\e\ra 0} \frac{36 s \e^{2/\k - 1}}{\eta^2}  E_Q\left[ \b_i^2 \ind{\e^{1/\k} \b_i \leq \d}  \right]  \\
&= \frac{36 s }{\eta^2} \frac{\tc \k}{2-\k} \d^{2-\k},
\end{align*}
where the last equality follows from \eqref{btail} and Karamata's
Theorem. This vanishes as $\d \ra 0$ since $\k<2$. It remains only to
show that the term in \eqref{bprocess} vanishes as first $\e\ra 0$ and then $\d\ra
0$. A similar statement (without the supremum inside the probability)
was shown in \cite[Lemma 5.5]{psWQLTn}. One can modify the techniques
of \cite{psWQLTn} to give a bound on \eqref{bprocess} that vanishes as
first $\e\ra 0$ and then $\d\ra 0$. Since the argument is somewhat 
technical, we postpone it until Appendix \ref{ProcessApp}.  
\end{proof}

\section{Weak weak quenched limits for the position of the random
  walk}
In this section we prove Theorem
\ref{WQLXn}. We start by defining the running maximum version of the
scaled path process of the random walk $\chi_\e$ in \eqref{chidef}. 
For $t\geq 0$, let $X_t^* = \max \{X_k: k\leq t\}$ denote the running
maximum of the RWRE. The corresponding random element in $ D_\infty$
is 
\[
 \chi_\e^*(t) 
= \begin{cases}
   \e^\k X_{t/\e}^* & \k \in (0,1) \\
   \frac{1}{\e \d(1/\e)^2} \left( X_{t/\e}^* - t \d(1/\e) \right) & \k = 1 \\
   \vp^{-1-1/\k} \e^{1/\k} \left(X_{t/\e}^* - t \vp / \e \right) & \k
   \in (1,2), 
  \end{cases}
\]
with the same function $\d$ in the case $\k=1$ as in \eqref{chidef}.
%Recall that in the case $\k=1$, $\d(x)$ is a function that satisfies $\d(x) D(\d(x)) = x + o(1)$ as $x\ra\infty$ and that $\d(x) \sim x/(A \log x)$.
%Also, recall that $p_{\e,\w}$ is the quenched distribution of $\chi_\e$ and let $p_{\e,\w}^* = P_\w(\chi_\e^* \in \cdot) $ be the quenched distribution of $\chi_\e^*$. 
The path $\chi_\e^*$  is easier to compare to transforms of the
hitting times path $\T_\e$ than the path $\chi_\e$  is. The following
lemma shows that the quenched distributions of $\chi_\e$ and
$\chi_\e^*$ are asymptotically equivalent, since the distance between
$\chi_\e$ and $\chi_\e^*$ is typically very small. 
\begin{lem}\label{chichistar}
For any $s<\infty$ and $\eta>0$, 
\[
 \lim_{\e\ra 0} \P\left( \sup_{t \leq s} |\chi_\e(t) - \chi_\e^*(t)| \geq \eta \right) = 0. 
\]
\end{lem}
\begin{proof}
The definitions of $\chi_\e$ and $\chi_\e^*$ imply that for all $\e>0$
small enough, 
\[
 \sup_{t \leq s} |\chi_\e(t) - \chi_\e^*(t)| = \max_{k \leq s/\e} (X_k^* - X_k) 
\begin{cases}
 \e^\k & \k \in (0,1) \\
 \frac{1}{\e \d(1/\e)^2} & \k = 1 \\
 \e^{1/\k} & \k \in (1,2) 
\end{cases} 
\leq \e^{\k/4}\, \max_{k \leq s/\e} (X_k^* - X_k).
\]
If, for some $0\leq k\leq s/\e$, $X_k^* - X_k \geq \eta \e^{-\k/4}$,
then, for some location $0\leq j\leq s/\e$ the random walk returns to
$X_j-\cl{\eta \e^{-\k/4}}$ after visiting location $j$. Thus, by the
stationarity of the environment under the measure $\P$, 
\[
 \P\left( \sup_{t \leq s} |\chi_\e(t) - \chi_\e^*(t)| \geq \eta \right) 
\leq \P\left(\max_{k\leq s/\e} (X^*_k - X_k) \geq \eta \e^{-\k/4} \right) 
\leq (1+s/\e) \P( T_{- \cl{\eta \e^{-\k/4}}} < \infty). 
\]
Since $\P(T_{-x} < \infty)$ decays exponentially fast as $x\ra\infty$
(see \cite[Lemma 3.3]{gsMVSS}), the term on the right vanishes as $\e\ra 0$. 
\end{proof}

We now prove Theorem \ref{WQLXn}. According to 
Lemma \ref{chichistar}, we may and will replace  $\chi_\e$ by
$\chi_\e^*$ when proving the coupling part. We consider the cases $\k \in (0,1)$,
$\k = 1$, and $\k \in (1,2)$ separately.

\subsection{Case I: $\k \in (0,1)$}
%By lemma \ref{chichistar}, it is enough to show that $p_{\e,\w}^* \Lra_\e H_{0,\infty}(N_{\l,\k}) \circ \mathfrak{I}^{-1}$. 
%To this end, we want to compare $\chi^*_\e$ with $\mathfrak{I} \T_{\e^\k}$. 
We wish to compare $\chi_\e^*$ with $\mathfrak{I}\T_{\e^\k}$ where  $\mathfrak{I}$ is the inversion operator defined in \eqref{Idef}. 
To this end, note that
\[
 X_{t/\e}^* = \max\{ k \in \Z: T_k \leq t/\e \} = \sup \{ x \geq 0: T_{\fl{x/\e^\k}} \leq t/\e \} \e^{-\k} - 1. 
\]
Therefore, for every $t \geq 0$, 
\[
 \chi_\e^*(t) 
%= \sup \{ x \geq 0: T_{\fl{x/\e^\k}} \leq t/\e \}  - \e^\k
%= \sup \{ x \geq 0: \e T_{\fl{x/\e^\k}} \leq t \}  - \e^\k \\
= \sup \{ x \geq 0: \T_{\e^\k}(x) \leq t \}  - \e^\k
= \mathfrak{I} \T_{\e^\k} (t) - \e^{\k}, 
\]
which implies that  $\sup_{t < \infty} | \chi_\e^*(t) - \mathfrak{I} \T_{\e^\k}(t) | =\e^\k$. 
%A simple consequence of this is that $d_\infty^{J_1}( \chi_\e^*, \, \mathfrak{I} \T_{\e^\k} ) \leq \e^\k$. 
%\[
% d_\infty^\circ( \chi_\e^*, \, \mathfrak{I} \T_{\e^\k} ) \leq \e^\k. 
%\]
Hence, we obtain the stated coupling in Theorem \ref{WQLXn}. 
%Therefore, by Lemma \ref{chichistar} we have that $\lim_{\e\ra 0} \P( d_\infty^{J_1}(\chi_\e, \mathfrak{I}\T_{\e^\k} ) \geq \eta ) = 0$ for any $\eta>0$.
%This is enough to show that $p_{\e,\w}$ and $m_{\e^{\k},\w} \circ
%\mathfrak{I}^{-1}$ have the same limiting distributions. 
By Lemma \ref{WAEcor}, it remains only to show that, under the measure
$P$, 
\be\label{mcI}
m_{\e^{\k}} \circ \mathfrak{I}^{-1} \Lra \H(N_{\l,\k}) \circ \mathfrak{I}^{-1}. 
\ee
%$m_{\e^{\k},\w} \circ \mathfrak{I}^{-1}$ converges in distribution to $\H(N_{\l,\k}) \circ \mathfrak{I}^{-1}$ for some $\l>0$. 

To this end, first note that $\mathfrak{I}$ is continuous on the subset $D_{\uparrow\uparrow}^+
\subset D_\infty^+$ of \emph{strictly} increasing functions, when the
$M_1$ topology is used both on the domain and the range  (see
\cite[Corollary 13.6.4]{wSPL} for even topologically stronger statement). 
Therefore, the mapping theorem implies that the function $\mu \mapsto
\mu \circ \mathfrak{I}^{-1}$ on $\mathcal{M}_1(D_\infty)$ is
continuous on the subset of measures $\{ \mu \in
\mathcal{M}_1(D_\infty): \, \mu( D_{\uparrow\uparrow}^+ ) = 1 \}$. 
In the notation introduced in \eqref{Zlkdef}, $\H(N_{\l,\k}) =
\Pv_\tau(Z_{\l,\k} \in \cdot)$. Since $Z_{\l,\k}$ is a $\k$-stable
subordinator under $\Pv \times \Pv_\tau$, then
$\H(N_{\l,\k})( D_{\uparrow\uparrow}^+ ) = \Pv_\tau (Z_{\l,\k} \in
D_{\uparrow\uparrow}^+ )=1$ for almost every realization of
$N_{\l,\k}$, and so \eqref{mcI} follows from Theorem
\ref{WQLTn1} and the continuous mapping theorem.   
%
%Then, since $m_{\e^\k,\w}$ converges in distribution to $\H(N_{\l,\k})$ for some $\l>0$, \eqref{mcI} will follow from the mapping theorem if we can show that $\Pv( \H(N_{\l,\k})(D_\infty^{++}) = 1 ) = 1$. However, $\H(N_{\l,\k})$ is the distribution of the path $W(N_{\l,\k},\vec\tau)$ which is constant on an interval $[t,t']$ only if $\tau_i=0$ for some $i$ or if the 
%point process $N_{\l,\k}$ doesn't have any atoms in the set $(0,\infty]\times [t,t']$. Thus, with probability one, $\H(N_{\l,\k})$ is concentrated on paths in $D_\infty^{++}$. 

\subsection{Case II: $\k \in (1,2)$}
We start by replacing the piecewise constant path of the hitting times
in \eqref{Tedef} by a piecewise linear and continuous version via 
linear interpolation. Specifically, for $x \in \Z$ and $\theta \in [0,1)$ we let
\[
 \tilde{T}_{x+\theta} = (1-\theta)T_x + \theta T_{x+1}. 
\]
Correspondingly, we will define $\tilde{\T}_\e(t) = \e^{1/\k}
(\tilde{T}_{t/\e} - t/(\e \vp))$, $t\geq 0$. The following lemma shows
that the $M_1$-distance between $T_\e$ and $\tT_\e$ is 
typically small. 
\begin{lem}\label{TtTcouple}
 For any $\eta>0$, 
$\lim_{\e \ra 0} \P( d_\infty^{M_1}(\T_\e, \tT_\e) \geq \eta ) = 0$. 
\end{lem}
\begin{proof} 
As in  Lemma \ref{Dt2Dinfty}, it is enough to prove that for every
$0<t<\infty$, $\P( d_t^{M_1}(\T_\e, \tT_\e) \geq \eta )\to 0$ for
every $\eta>0$. We use a matching of the kind similar to that
constructed in the proof of Lemma \ref{TnTnuncouple}. 
We will describe this matching in the case $\k \in (1,2)$, but a similar argument works when $\k = 1$ or $\k \in (0,1)$. 
For every
$k=0,1,2,\ldots$ such that $\e k \leq t$ we arrange both parametric
representations to contain the point $\bigl(\e  k, 
\e^{1/\k}(T_k-k/\vp)\bigr)$. If $\e (k+1) \leq t$, then between the points $\bigl(\e  k, 
\e^{1/\k}(T_k-k/\vp)\bigr)$ and $\bigl(\e  (k+1), 
\e^{1/\k}(T_{k+1}-(k+1)/\vp)\bigr)$ we keep the parametrization of $\tT_\e$
at the former point  until the parametrization of $\T_\e$ reaches the
point $\bigl(\e  (k+1), \e^{1/\k}(T_k-(k+1)/\vp)\bigr)$, at which time we
complete the two parametrizations in the interval by keeping the
slope between the matched points equal to $-\e^{1/\k-1}/\vp$. 
Clearly, within this interval the
horizontal distance between the two parametrizations is at most $\e$
and the vertical distance is at most $\e^{1/\k}/\vp$. If $\e k<t<\e
(k+1)$, then we use the obvious  vertical matching of the
parameterizations, with equal horizontal components, and vertical
components at most $\e^{1/\k} (T_{\fl{t/\e}+1} - T_{\fl{t/\e}})$ apart. Therefore, for
$\e$ small enough 
\[
 d_t^{M_1}( \T_\e, \tT_\e ) \leq \max\left\{ \e^{1/\k}/\vp, \,
   \e^{1/\k} ( T_{\fl{t/\e}+1} - T_{\fl{t/\e}}) \right\}.  
\]
Since $T_{\fl{t/\e}+1} - T_{\fl{t/\e}}$ has,  under the measure $\P$,
the same distribution as $T_1$ we conclude that 
\be\label{TtTcouplet}
\limsup_{\e\ra 0} \P\left( d_t^{M_1}( \T_\e, \tT_\e ) \geq \eta \right) \leq \lim_{\e\ra 0} \P\left(  T_{\fl{t/\e}+1} - T_{\fl{t/\e}} \geq \e^{-1/\k} \eta \right) = 0,
\ee
as required. 
%The statement of the Lemma then follows from this since 
%\begin{align*}
% \lim_{\e \ra 0} \P( d_\infty^{M_1}(\T_\e, \tT_\e) \geq \eta ) 
%&= \lim_{\e \ra 0} \P\left( \int_0^\infty e^{-t} ( d_t^{M_1}(\T_\e, \tT_\e) \wedge 1 ) \, dt \geq \eta \right) \\
%&\leq \frac{1}{\eta} \lim_{\e \ra 0} \E \left[ \int_0^\infty e^{-t} ( d_t^{M_1}(\T_\e, \tT_\e) \wedge 1 ) \, dt \right] \\
%&= \frac{1}{\eta} \lim_{\e \ra 0} \int_0^\infty e^{-t} \E \left[  ( d_t^{M_1}(\T_\e, \tT_\e) \wedge 1 ) \right] \, dt
%\end{align*}
\end{proof}

Note that $x\mapsto \tilde{T}_x$ is a strictly increasing and
continuous function on $[0,\infty)$. Let $\phi(t)$ be its
inverse. Then $\tilde{T}_{\phi(t)} = t$ for all $t\geq 0$.  
If $T_n \leq t < T_{n+1}$ then $X_t^* = n \leq \phi(t) < n+1$, so that
$\sup_{t\geq 0} |X_t^* - \phi(t)| \leq 1$. 
One consequence of this comparison is that
\be\label{pwlim}
 \lim_{t\ra\infty} \frac{\phi(t)}{t} 
%= \lim_{n\ra\infty} \frac{n}{T_n} 
= \lim_{n\ra\infty} \frac{X_n^*}{n} = \vp, \qquad \P\text{-a.s.}
\ee
Next define $\phi_\e(t) = \e \phi(t/\e)$ for $\e>0$ and $\phi_0(t) =
\vp t$. Then, \eqref{pwlim} implies that $\phi_\e$ converges pointwise
to $\phi_0$ as $\e \ra 0$. Moreover, since $\phi_\e$ and $\phi_0$ are
monotone increasing and $\phi_0$ is continuous, we conclude that
$\phi_\e$ converges uniformly on compact subsets to $\phi_0$. In
particular, $\lim_{\e\ra 0} d_\infty^{U}(\phi_\e,
\phi_0) = 0$, $\P$-a.s. 

Now, recalling the definition of $\tT_\e$ we obtain that 
\begin{align*}
 \tT_\e(\phi_\e(t)) &= \e^{1/\k}( \tT_{\phi_\e(t)/\e} - \phi_\e(t)/(\e \vp)) \\
&= \e^{1/\k}( \tT_{\phi(t/\e)} - \phi(t/\e)/\vp) \\
%&= \e^{1/\k}( t/\e - \phi(t/\e)/ \vp) \\
&= -\vp^{-1} \e^{1/\k} ( \phi(t/\e) - t\vp/\e ) \\
&= -\vp^{-1} \e^{1/\k} ( X_{t/\e}^* - t\vp/\e ) + \vp^{-1} \e^{1/\k}( X_{t/\e}^* - \phi(t/\e) ) \\
&= -\vp^{1/\k} \chi_\e^*(t) + \vp^{-1} \e^{1/\k}( X_{t/\e}^* - \phi(t/\e) ).
\end{align*}
Since $| X_{t/\e}^* - \phi(t/\e)| \leq 1$ for all $t$, this implies that 
\be\label{chitTpe}
 d_\infty^{U}( \chi_\e^*, \, -\vp^{-1/\k} \tT_\e\circ \phi_\e) \leq \vp^{-1-1/\k} \e^{1/\k}. 
\ee
%$d_\infty^{J_1}( \chi_\e^*, \, \vp^{-1/\k} \tT_\e\circ \phi_\e) \leq \vp^{-1-1/\k} \e^{1/\k}$. 
Next, we compare $\tT_\e\circ \phi_\e$ with $\T_\e\circ \phi_0$. 
To this end, let $\eta, \eta'>0$ be fixed. By Corollary
\ref{AveragedTn}, the laws of $(\T_\e)$ under $\P$ are tight in
$(D_\infty, d_\infty^{M_1})$. Therefore, we can choose a compact
subset $K \subset D_\infty$ so that $\P( T_\e \in K) \geq 1-\eta'$
for all $\e$ small enough. Further, the composition function
$\psi(x,y)=x\circ y$ is continuous at any point $(x,\phi_0) \in D_\infty \times
C_{\uparrow\uparrow}^+\subset D_\infty \times D_\infty$.  Therefore,
it is uniformly continuous (with the $d_\infty^{M_1}$ metric on each
coordinate) at the points of the compact set $K \times
\{\phi_0\}$. Choose now $\d>0$ such that $d_\infty^{M_1}(x'
\circ \phi', x \circ \phi_0) < \eta$ whenever $x \in K$,
$d_\infty^{M_1}(x,x')<\d$, and $d_\infty^{M_1}(\phi_0, \phi') < \d$.  
Then 
\begin{align*}
&\limsup_{\e\ra 0} \P( d_\infty^{M_1}( \tT_\e \circ \phi_\e, \, \T_\e \circ \phi_0 ) \geq \eta ) \\
&\quad \leq \limsup_{\e\ra 0} \P( \T_\e \notin K ) + \P(
d_\infty^{M_1}( \T_\e, \tT_\e ) \geq \d ) + \P( d_\infty^{M_1}(
\phi_\e, \phi_0 ) \geq \d ) \\ 
&\quad \leq \eta',
\end{align*}
where the last inequality follows from our choice of the compact set $K$, Lemma \ref{TtTcouple}, and the almost sure convergence of $\phi_\e$ to $\phi_0$. 
Since $\eta'>0$ was arbitrary, we see that $\P(  d_\infty^{M_1}(
\tT_\e \circ \phi_\e, \, \T_\e \circ \phi_0 ) \geq \eta ) \ra 0$ for
any $\eta>0$. Combining this with \eqref{chitTpe} we conclude that  
\be\label{chiTp0}
 \lim_{\e\ra 0} \P( d_\infty^{M_1}( \chi_\e^*, -\vp^{-1/\k} \T_\e
 \circ \phi_0 ) \geq \eta ) = 0, \quad \forall \eta > 0.  
\ee
Finally, note that the definition of $\T_\e$ and $\phi_0$ imply that 
\[
 \vp^{-1/\k} \T_\e (\phi_0(t)) = \vp^{-1/\k} \e^{1/\k}( T_{\vp t/\e} - t/\e ) = \T_{\e/\vp}(t),
\]
so that \eqref{chiTp0} proves the coupling part of Theorem \ref{WQLXn}
in the case $\k \in (1,2)$. Since $m_{\e/\vp} \circ
\mathfrak{R}^{-1}$ is the quenched distribution of $-\T_{\e/\vp}$, and 
$\mathfrak{R}$ is a continuous operator, the continuous mapping
theorem implies that $m_{\e} \circ
\mathfrak{R}^{-1}\Lra \mu_{\l,\k} \circ \mathfrak{R}^{-1}$. The
coupling now implies that we also have $p_{\e,\w} \Lra
\mu_{\l,\k} \circ \mathfrak{R}^{-1}$.  

\subsection{Case III: $\k = 1$}
The proof here is similar to the proof in the case $\k \in (1,2)$, so
we will omit some of the details. 
As above, let $\tilde{T}_x$ and $\phi(t)$ be as above, so that
$\tilde{T}_{\phi(t)} = t$.  
We claim that 
\be\label{phiasymp}
\lim_{t\ra\infty} \frac{\phi(t)}{t/\log t} = \frac{1}{A} \qquad \text{ in } \P \text{-probability,}
\ee
where $A$ is the positive constant from Theorem
\ref{averagedlimlaw}. To see this, first note that  by Theorem
\ref{averagedlimlaw}, 
$\lim_{n\ra\infty} T_n/(n \log n) = A$ in $\P$-probability, hence also
$\lim_{x\ra\infty} \tilde{T}_x/(x \log x) = A$ in $\P$-probability.  
Using $x=\phi(t)$ gives us 
\[
 \lim_{t\ra\infty} \frac{t}{\phi(t)\log(\phi(t))} = A \qquad \text{ in } \P \text{-probability,}  
\]
which proves \eqref{phiasymp}. 

The function $\d(x) = \sup\{u>0:\, uD(u)\leq x\}$, $x>0$, satisfies 
$\d(x) \sim x/(A \log x)$ as $x\ra\infty$ and 
$\d(x)D(\d(x)) = x+o(\d(x))$ as $x \ra \infty$.  
We define $\phi_\e \in D_\infty$ by $\phi_\e(t) = \phi(t/\e)/\d(1/\e)$. 
Then the asymptotics of $\phi$ from \eqref{phiasymp} and the
asymptotics of $\d$ imply that for any $t\geq 0$, 
\[
 \lim_{\e\ra 0} \phi_\e(t) = \lim_{\e\ra 0} \frac{\phi(t/\e)}{\d(1/\e)} 
%= \lim_{\e\ra 0} \frac{t/\e}{A \log(t/\e)} \frac{A\log(1/\e)}{1/\e} = \lim_{\e\ra 0} \frac{t\log(1/\e)}{\log(t/\e)} 
= t \qquad \text{in } \P\text{-probability}. 
\]
Once again, since $\phi_\e$ is non-decreasing, and the identity
function is continuous, $\phi_\e$ converges uniformly on compact
subsets (and thus also in the $d_\infty^{J_1}$ metric), in
$\P$-probability, to the identity function. 
%Then, since $\phi_\e$ is non-decreasing with a pointwise limit that is continuous and non-decreasing the convergence is actually uniform on compact subsets of $[0,\infty)$. 
%This implies that $\phi_\e$ converges in $Q\times P_\w$-probability to the identity function $I \in D_\infty$. 

Let $\tT_\e(t) = \e ( \tilde{T}_{t/\e} - t/\e D(1/\e) )$, $t\geq
0$. Then 
\begin{align*}
 \tT_{\d(1/\e)^{-1}}(\phi_\e(t)) &= \d(1/\e)^{-1} \left( \tilde{T}_{\phi_\e(t) \d(1/\e) } - \phi_\e(t) \d(1/\e) D( \d(1/\e) )  \right) \\
&= \d(1/\e)^{-1} \left( \tilde{T}_{\phi(t/\e)} - \frac{ \phi(t /\e) }{
    \d(1/\e) } \left( \frac{1}{\e} + o(\d(1/\e)) \right) \right) \\ 
&= \d(1/\e)^{-1} \left( t/\e - \frac{ X_{t/\e}^* + \bigo(1) }{ \d(1/\e) } \left( \frac{1}{\e} + o(\d(1/\e)) \right) \right) \\
%&= \frac{1}{\e \d(1/\e)^2} \left( t \d(1/\e) - (X_{t/\e}^* + \bigo(1))(1+o(\e)) \right) \\
%&= \frac{1}{\e \d(1/\e)^2} \left( t \d(1/\e) - X_{t/\e}^* + X_{t/\e}^*o(\e) + \bigo(1) + o(\e) \right) \\
%&= \frac{1}{\e \d(1/\e)^2} \left( t \d(1/\e) - X_{t/\e}^* \right) + o\left( \frac{t}{\e \d(1/\e)^2} \right) + \bigo\left(  \frac{1}{\e \d(1/\e)^2}  \right) + o \left( \frac{1}{ \d(1/\e)^2}  \right) \\
&= \frac{1}{\e \d(1/\e)^2} \left( t \d(1/\e) - X_{t/\e}^* \right) + 
o(1) \frac{X_{t/\e}^*}{\d(1/\e)}  + \bigo\left(  \frac{1}{\e \d(1/\e)^2}  \right)\\
&= -\chi_\e^*(t) + o(1) \frac{X_{t/\e}^*}{\d(1/\e)} +
\bigo\left(  \frac{1}{\e \d(1/\e)^2}  \right), 
\end{align*}
where in the third equality we used that $|\phi(t) - X_{t}^*| \leq 1$ for all $t$. 
Since $1/(\e \d(1/\e)^2) \sim A^2 \e \log^2(1/\e) \ra 0$ as $\e\ra 0$,
while $X_{t/\e}^*/\d(1/\e)$ converges in probability by Theorem
\ref{averagedlimlaw}, this implies that 
%$d_\infty^{J_1}(\chi_\e^*,  %-\tT_{\d(1/\e)^{-1}} \circ \phi_\e )   %\ra 0$ as $\e\ra 0$. 
\be\label{chitTpe1}
 \lim_{\e\ra 0} d_\infty^{U}(\chi_\e^*, -\tT_{\d(1/\e)^{-1}} \circ
 \phi_\e ) = 0 \qquad \text{ in } \P \text{-probability.}
\ee
As in case $\k \in (1,2)$ we can use the fact that $\phi_\e$ converges
to the identity function to show that for any $\eta>0$, 
\be\label{tTpeT}
 \lim_{\e\ra 0} \P\left( d_\infty^{M_1}( \tT_{\d(1/\e)^{-1}} \circ
   \phi_\e, \tT_{\d(1/\e)^{-1}} ) \geq \eta \right) = 0. 
\ee
Combining \eqref{chitTpe1}, \eqref{tTpeT} and Lemma \ref{chichistar}
establishes the coupling part of Theorem \ref{WQLXn}, and the rest is
the same as in the case $\k \in (1,2)$.

\section{$\mathcal{M}_1(\R)$-valued Stable L\'evy process limits}\label{StableRPD}
In this section we discuss  Corollary \ref{PathMeasureCor}.
We begin with a short proof of the convergence of the finite dimensional distributions of $\Phi(m_\e)$. 
Let $m\geq 1$ and $0 \leq t_1 < t_2 < \ldots < t_m$ be given, and define $\Phi_{t_1,\ldots,t_m}: \mathcal{M}_1(D_\infty) \ra \mathcal{M}_1(\R)^m$ by 
\[
 \Phi_{t_1,t_2,\ldots,t_m}(\mu) = ( \Phi_{t_1}(\mu),  \Phi_{t_2}(\mu), \ldots, \Phi_{t_m}(\mu)). 
\]
It is easy to see that $\Phi_{t_1,t_2,\ldots,t_m}$ is continuous at
every $\mu \in \mathcal{M}_1(D_\infty)$ concentrated on paths that are
continuous at $t_i$, $i=1,2,\ldots, m$; see p. 383 in \cite{wSPL}.  Since $m_\e \Lra \mu_{\l,\k}$
and $\mu_{\l,\k}$ is, with probability, one concentrated on paths that
are continuous at $t_i$, $i=1,2,\ldots,m$, then the continuous mapping
theorem implies that $\Phi_{t_1,t_2,\ldots,t_m}(m_\e) \Lra
\Phi_{t_1,t_2,\ldots,t_m}(\mu_{\l,\k})$. This proves the convergence
of finite dimensional distributions claimed in Corollary
\ref{PathMeasureCor}.  

We now turn to the stated properties of the random measure-valued path $\Phi(\mu_{\l,\k})$, namely that $\Phi(\mu_{\l,\k}) $ is a stable L\'evy process on $\mathcal{M}_1(\R)$.
We start by recalling 
the notions of stable random variables and L\'evy
processes on $\mathcal{M}_1(\R)$; the reader is referred to
\cite{stRPD} for more details. 

\begin{defn}
 A $\mathcal{M}_1(\R)$-valued random variable $\mu$ is a \textbf{stable random variable on $\mathcal{M}_1(\R)$} if for any $n\geq 2$ there exist constants $b_n \in \R$ and $c_n >0$ such that 
%for $n$ independent copies $\mu_1, \mu_2,\ldots \mu_n$ of $\mu$,
%\[
% \left( \mu_1 * \mu_2 * \cdots * \mu_n * \delta_{-b_n} \right) (\cdot / c_n ) \overset{\text{Law}}{=} \mu(\cdot). 
%\]
\[
 \left( \mu_1 * \mu_2 * \cdots * \mu_n \right)(b_n + c_n^{-1} \cdot ) \overset{\text{Law}}{=} \mu(\cdot). 
\]
Here $\mu_1, \mu_2,\ldots \mu_n$ are independent copies of $\mu$. 
Moreover, if $b_n = 0$ for every $n\geq 2$, then $\mu$ is a \textbf{strictly stable random variable on $\mathcal{M}_1(\R)$}. 
\end{defn}
\begin{defn}\label{LevyDef}
 A $\mathcal{M}_1(\R)$-valued stochastic process $\{ \Xi(t) \}_{t\geq 0}$ is a \textbf{L\'evy process on $\mathcal{M}_1(\R)$} if 
there exists a two parameter family of $\mathcal{M}_1(\R)$-valued random variables $\{\Xi_{s,t}\}_{0\leq s<t}$ such that $\Xi(t) = \Xi_{0,t}$ and 
\begin{enumerate}
 \item $\Xi(0) = \d_0$ with probability one. 
 \item\label{IndepInc} For any $n\geq 2$ and $0=t_0<t_1<t_2<\cdots <
   t_n=t$, $\{\Xi_{t_{i-1},t_i}\}_{i=1}^n$ are independent and  
\[
\Xi(t) =  \Xi_{t_0,t_1} * \Xi_{t_1,t_2} * \cdots * \Xi_{t_{n-1}, t_n}, \quad \text{almost surely}. 
\]
 \item\label{Stationary} For any $0<s<t$, $\Xi(t-s) \overset{\text{Law}}{=} \Xi_{s,t}$.  
 \item\label{StochCont} For any fixed $t_0\geq 0$, the process $\{
   \Xi(t) \}_{t\geq 0}$  is continuous at $t_0$ in probability. 
 \item\label{Cadlag} There is an event of probability 1 on which every 
path $\{t \mapsto \Xi(t) \}$ is in $ D_\infty(\mathcal{M}_1(\R))$. 
\end{enumerate}
\end{defn}
\begin{rem}
 Part \eqref{IndepInc} of Definition \ref{LevyDef} is a version of the
 independent increments property for stochastic processes with values in $\mathcal{M}_1(\R)$ with convolution of measures playing the role of addition. 
% In that space convolution of probability measures plays the role of addition. 
The version of the  independent increments property used in
 the above definition is necessary due to absence of 
an inverse operation to convolution. 
\end{rem}
\begin{defn}
A L\'evy process $\{ \Xi(t) \}_{t\geq 0}$ on $\mathcal{M}_1(\R)$ is a
\textbf{(strictly) stable L\'evy process on $\mathcal{M}_1(\R)$} if
for every fixed $t\geq 0$, $\Xi(t)$ is a (strictly) stable random
variable on $\mathcal{M}_1(\R)$.  
\end{defn}

We are now ready to show that $\Phi(\mu_{\l,\k})$ is a stable L\'evy
process on $\mathcal{M}_1(\R)$ and is strictly stable when $\k\neq 1$.  
It is easy to check that for any $t\geq 0$, $\Phi_t(\mu_{\l,\k})$ is a
stable random variable on  $\mathcal{M}_1(\R)$ and is strictly stable if $\k\neq 1$ (see the Remark 1.5 and the paragraph following Remark 1.6 in \cite{psWQLTn}),  so we will concentrate on showing that $\Phi(\mu_{\l,\k})$ is a L\'evy process on $\mathcal{M}_1(\R)$.  

We already know that the paths of $\Phi(\mu_{\l,\k})$ are in $
D_\infty(\mathcal{M}_1(\R))$; see the discussion before Theorem
\ref{PathMeasureCor}.  
It is also obvious that $\Phi_0(\mu_{\l,\k}) = \d_0$ with probability one
since $N_{\l,\k}( (0,\infty] \times \{0\} )  = 0$ with probability
one. Next, recall the stochastic process $Z_{\l,\k}$  defined in \eqref{Zlkdef}.  
Then $\Phi_t(\mu_{\l,\k})$ is the distribution of $Z_{\l,\k}(t)$ under
the measure $\Pv_\tau$, and we define 
%As was shown in the proof of Corollary \ref{AveragedTn}, when $\k \in [1,2)$ the limits in the definition of $Z_{\l,\k}(t)$ exist and are finite with probability one. 
for any $0\leq s<t$ 
\[
 \Phi(\mu_{\l,\k})_{s,t} = \Pv_\tau \left(Z_{\l,\k}(t) - Z_{\l,\k}(s) \in \cdot \right). 
\] 
Then, the independent increments condition \eqref{IndepInc} in
Definition \ref{LevyDef} follows from the fact that $\{ N_{\l,\k}(
\cdot \cap \bigl( (0,\infty] \times (t_{i-1},t_i] \bigr)\}_{i=1}^n$ are independent
for any $0=t_0 <  t_1 < \cdots < t_n$, and the stationarity condition
\eqref{Stationary} in Definition \ref{LevyDef} follows from the shift
invariance of the Lebesgue measure governing the time component of the
Poisson random measure $N_{\l,\k}$. 
 %from the fact that $$N_{\l,\k}( \cdot \cap (0,\infty] \times (s,t]) \overset{\text{Law}}{=} N_{\l,\k}( \cdot  \cap (0,\infty] \times (0,t-s]).$$ 
Finally, stochastic continuity of $\Phi(\mu_{\l,\k})$ at fixed points
follows from the fact that for each fixed $t_0$, $N_{\l,\k}((0,\infty]
\times \{t_0\}) =0$ with probability 1.

\subsection{Topologies on $D_\infty(\mathcal{M}_1(\R))$}

%We close this section with a discussion of the 
We now give a brief discussion of the
difficulties of  extending Corollary \ref{PathMeasureCor} to a full
weak convergence $\Phi(m_\e) \Lra \Phi(\mu_{\l,\k})$ of
$\mathcal{M}_1(\R)$-valued path processes.  
It is first necessary to decide on a topology for
$D_\infty(\mathcal{M}_1(\R))$, the space of \cadlag\ paths taking
values in the space of probability measures on $\R$. 
Recall that the Prohorov metric $\rho$ on $\mathcal{M}_1(\R)$ induces
the topology of convergence in distribution and that
$(\mathcal{M}_1(\R),\rho)$ is a Polish space.  
Then, both the uniform and the $J_1$-topologies have natural extensions to $D_\infty(\mathcal{M}_1(\R))$. 
%For instance, if $D_t(\mathcal{M}_1(\R))$ is the space of $\mathcal{M}_1(\R)$-valued \cadlag\ functions on $[0,t]$, then  
%\[
% \hat{d}_t^{J_1}(\mathbf{x}, \mathbf{y} ) = \inf_{\l \in \L_t} \max \left\{ \sup_{s \leq t} |\l(s) - s|  , \, \sup_{s \leq t} \rho\left( \mathbf{x}(\l(s)), \, \mathbf{y}(s) \right) \right\}, \quad \mathbf{x},\, \mathbf{y} \in D_t(\mathcal{M}_1(\R)),
%\]
%is the Skorohod $J_1$-metric on $D_t(\mathcal{M}_1(\R))$, and the $J_1$-metric $\hat{d}_\infty^{J_1}$ on $D_\infty(\mathcal{M}_1(\R))$ is defined in terms of $\hat{d}_t^{J_1}$ in the same way that $d_\infty^{J_1}$ was defined in terms of $d_t^{J_1}$ in section \ref{MTgeneral}. 

In the proof of Theorem \ref{WQLTn1}, it was necessary to equip
$D_\infty$ with the $M_1$-topology to accommodate the fact that the
macroscopic jumps of the process of ladder location hitting times were
an accumulation of smaller jumps $T_i-T_{i-1}$ for $i$ between
consecutive ladder locations. 
The $M_1$-topology naturally accomodates such accumulations of jumps
while the $J_1$-topology does not.  
A similar phenomenon occurs when trying to establish weak convergence
$\Phi(m_\e) \Lra \Phi(\mu_{\l,\k})$ in the space of
probability measure-valued functions, and thus it is natural to try to 
equip $D_\infty(\mathcal{M}_1(\R))$ with a Skorohod
$M_1$-topology. This is less standard than defining the Skorohod
$J_1$-topology\footnote{For instance, Whitt defines the Skorohod
  $M_1$-topology on $D_t(\Psi)$ if $\Psi$ is a separable Banach space
  \cite[p. 382]{wSPL}, but $\mathcal{M}_1(\R)$ is not a Banach
  space.}, 
but, since convex combinations $(1-\theta)\mu + \theta \pi$ of two
probability measures form a ``line segment'' between $\mu$ and $\pi$
in $\mathcal{M}_1(\R)$, one can define the $M_1$-topology and metric on
$D_t(\mathcal{M}_1(\R))$ and $D_\infty(\mathcal{M}_1(\R))$ in the
natural way. 
%In defining the $M_1$ topology, 
%one first needs to define the \emph{completed graph} of of a $\mathcal{M}_1(\R)$-valued path. 
%This can be done by connecting a jump of the path from $\mu$ to $\pi$ in $\mathcal{M}_1(\R)$ by the ``line segment'' of convex combinations $(1-\theta)\mu + \theta \pi$ of the respective measures. 
%From this, on can define the $M_1$-topology and metric on $D_t(\mathcal{M}_1(\R))$ and $D_\infty(\mathcal{M}_1(\R))$ in the natural way. 
Moreover, the resulting $M_1$-topology on
$D_\infty(\mathcal{M}_1(\R))$ defined in this way is the topology of a complete
separable metric space. 

Unfortunately, to this point we have been unable to prove weak
convergence  $\Phi(m_\e) \Lra \Phi(\mu_{\l,\k})$ in the $M_1$-topology
(as defined above) on $D_\infty(\mathcal{M}_1(\R))$. In fact, some
preliminary computations suggest that $\{\Phi(m_\e) \}_{\e>0}$ is not
a tight family of $D_\infty(\mathcal{M}_1(\R))$-valued random
variables in this topology, and thus a weaker topology on
$D_\infty(\mathcal{M}_1(\R))$ may be needed. We hope to address this in a future paper.

We close this section with an example that demonstrates some of the
difficulties establishing weak convergence  $\Phi(m_\e) \Lra
\Phi(\mu_{\l,\k})$. A natural approach to proving $\Phi(m_\e) \Lra
\Phi(\mu_{\l,\k})$ would be to apply Theorem \ref{WQLTn1} and the
continuous mapping theorem. Unfortunately, the mapping $\Phi:
\mathcal{M}_1(D_\infty) \ra D_\infty(\mathcal{M}_1(\R))$ is not
continuous.  The following example demonstrates
this lack of continuity even when in  $\mathcal{M}_1(D_\infty)$
we endow the space $D_\infty$ 
with the strongest of the Skorohod topologies, the $J_1$-topology, and
endow $D_\infty(\mathcal{M}_1(\R))$ with the weakest of the Skorohod
topologies, the $M_2$-topology. 

We restrict everything to the interval $[0,1]$ and consider
real-valued stochastic processes $X=(X(t), \, 0\leq t\leq 1)$ and 
$X_n=(X_n(t), \, 0\leq t\leq 1)$, $n=1,2,\ldots$, on the probability
space $\bigl( [0,1], {\mathcal B}, \text{Leb}\bigr)$, defined by 
$$
X(t; \omega) = \left\{ \begin{array}{ll}
0, & 0\leq t<\frac12, \\
1, & \frac12\leq t\leq 1, \, 0\leq \omega\leq\frac12,\\
2, & \frac12\leq t\leq 1, \, \frac12< \omega\leq 1,
\end{array}
\right.
$$
and
$$
X_n(t; \omega) = \left\{ \begin{array}{ll}
0, & 0\leq t<\frac12 - \frac{1}{2^{n+1}}, \\
1, & \frac12 - \frac{1}{2^{n+1}}\leq t\leq 1, \, 0\leq \omega\leq
\frac12,\\
0, & \frac12 - \frac{1}{2^{n+1}}\leq t< \frac12 +\frac{1}{2^{n+1}}, \,
\frac12< \omega\leq 1, \\
2, & \frac12 +\frac{1}{2^{n+1}}\leq t\leq 1, \, \frac12< \omega\leq 1,
\end{array}
\right.
$$
$n=1,2,\ldots$. Clearly, each process $X$ and $X_n$ has its sample
paths in $D_1=D[0,1]$. We denote by $\mu$ (correspondingly, $\mu_n$)
the probability measures these processes generate on the cylindrical
sets in $D[0,1]$.

Obviously, for any $\omega\in [0,1]$, $d^{J_1}(X,X_n)\leq 2^{-(n+1)}$,
so, with probability 1, $X_n\to X$ in $D[0,1]$ equipped with the
$J_1$-topology, so that $\rho^{J_1}(\mu_n, \mu) \ra 0$. 

Next, in the notation of \eqref{e:projection}, we have
$$
\Phi_t(\mu) = \left\{ \begin{array}{ll} 
\delta_0, & 0\leq t<\frac12, \\
\frac12 \delta_1 + \frac12 \delta_2, & \frac12\leq t\leq 1
\end{array} \right.
$$
and
$$
\Phi_t(\mu_n) = \left\{ \begin{array}{ll} 
\delta_0, & 0\leq t<\frac12 - \frac{1}{2^{n+1}}, \\
\frac12 \delta_0 + \frac12 \delta_1, & \frac12 - \frac{1}{2^{n+1}}\leq
t<\frac12 +\frac{1}{2^{n+1}}, \\
\frac12 \delta_1 + \frac12 \delta_2, & \frac12 +\frac{1}{2^{n+1}}\leq t\leq 1,
\end{array} \right.
$$
$n=1,2,\ldots$, where for $x\in\R$, $\delta_x$ is the point mass at
$x$. Note that for any $n$, there is a point on the completed graph of
the element $\Phi(\mu_n)$ of $D_1(\mathcal{M}_1(\R))$ with the second
component equal to $(1/2)\delta_0 + (1/2)\delta_1$, and the 
distance from that point to the completed graph of the $\Phi(\mu)$ has
a positive lower bound that does not depend on $n$. Therefore,
$\Phi(\mu_n)$ does not converge to $\Phi(\mu)$ in
$D_1(\mathcal{M}_1(\R))$ even if the latter space is endowed with the
$M_2$-topology (see Section 11.5 in \cite {wSPL} for the definition of the $M_2$-topology).

\appendix
\section{Estimation of the term in \eqref{bprocess}}\label{ProcessApp}

In order to finish the proof of Proposition \ref{WQLSn} we need to
estimate the term in \eqref{bprocess}. In this appendix we achieve
that by proving the following lemma. 

\begin{lem}\label{smallblem}
If $\k\in[1,2)$, then  for all $0< s<\infty$ and $ \eta > 0$, 
\[
\lim_{\d\ra 0} \limsup_{n \ra \infty}  Q\left( \sup_{t\leq s} \left|
    n^{-1/\k} \sum_{i=1}^{\lfloor tn\rfloor} \left\{ \b_i \ind{\b_i \leq \d n^{1/\k}}
      - E_Q[\b_1 \ind{\b_1 \leq \d n^{1/\k} }] \right\} \right| \geq
  \eta \right) = 0.
 \]
\end{lem}
\begin{rem}
 Lemma \ref{smallblem} is an improvement of \cite[Lemma 5.5]{psWQLTn}, 
 which stated that  
\[
 \lim_{\d \ra 0} \limsup_{n\ra \infty}  Q\left( \left| n^{-1/\k} \sum_{i=1}^{n} \left\{ \b_i \ind{\b_i \leq \d n^{1/\k}} - E_Q[\b_1 \ind{\b_1 \leq \d n^{1/\k} }] \right\} \right| \geq \eta \right), \quad \forall \eta > 0. 
\]
\end{rem}

Before giving the proof of Lemma \ref{smallblem}, we introduce new 
notation. Recall that $\rho_x =(1-\w_x)/\w_x$, and for $i\leq j$ let
\begin{equation}\label{PiRWdef}
 \Pi_{i,j} = \prod_{k=i}^j \rho_k, \quad R_{i,j} = \sum_{k=i}^j \Pi_{i,k}, \quad W_{i,j} = \sum_{k=i}^j \Pi_{k,j}, \quad W_j = \sum_{k=-\infty}^j \Pi_{k,j}. 
\end{equation}
This notation is often useful for writing certain quenched
expectations or probabilities in compact form. For instance, it is
easy to show that $E_\w^i [T_{i+1}] = 1 + 2 W_i$ (see \cite{zRWRE} for
a reference). In particular, 
\be\label{bform}
 \b_i = E_\w[ T_{\nu_i} - T_{\nu_{i-1}}] = \sum_{j=\nu_{i-1}}^{\nu_i-1}
 E_\w^j[T_{j+1}] = \nu_i-\nu_{i-1} + 2 \sum_{j=\nu_{i-1}}^{\nu_i-1}
 W_j. 
\ee

%Because of the tail decay of the $\b_i$ under $Q$, Lemma \ref{smallblem} would follow from standard arguments if the $\b_i$ were i.i.d. However, it turns out that the $\b_i$ do have good mixing properties under $Q$. In fact, $\b_i$ is mostly dependent on the environment in the interval $[\nu_{i-1},\nu_i)$ and (as mentioned in Section \ref{Notation}) these intervals of the environment are i.i.d.\ under $Q$. One way to see this is to note first of all that 
%\begin{align*}
% \b_i &= \nu_i-\nu_{i-1} + 2 \sum_{j=\nu_{i-1}}^{\nu_i-1} \left( W_{\nu_{i-1},j} + \Pi_{\nu_{i-1},j}W_{\nu_{i-1}-1} \right)\\
%&= \nu_i-\nu_{i-1} + 2 \sum_{j=\nu_{i-1}}^{\nu_i-1} W_{\nu_{i-1},j} + R_{\nu_{i-1},\nu_i-1}W_{\nu_{i-1}-1} 
%\end{align*}

It will be important for us to be able to control the tails, under the
measure $Q$, of the random variables of the type $W_{\nu_m -1}$. Since
under $Q$ the 
environment is stationary under shifts by the ladder locations, these
random variables 
all have the same distribution under $Q$ as $W_{-1}$. Further, it was shown
in \cite[Lemma 2.2]{pzSL1} that $W_{-1}$ has,  under
$Q$, exponential tails. That is, there exist constants $C_1,C_2>0$
such that  for any $x>0$, 
\be\label{Wtail}
 Q( W_{-1} > x ) \leq C_1 e^{-C_2 x}. 
\ee

We now proceed to prove Lemma \ref{smallblem}. 

\begin{proof}[Proof of Lemma \ref{smallblem}]
First, note that $\b_i \ind{\b_i \leq \d n^{1/\k}} = \b_i \wedge \d
n^{1/\k} - \d n^{1/\k} \ind{ \b_i > \d n^{1/\k}}$.   Thus, 
\begin{align}
 & Q\left( \sup_{t\leq s} \left| n^{-1/\k} \sum_{i=1}^{\lfloor tn\rfloor} \left\{ \b_i \ind{ \b_i \leq \d n^{1/\k}} - E_Q[\b_1 \ind{\b_1 \leq \d n^{1/\k}}] \right\} \right| \geq \eta \right) \nonumber \\
&\leq Q\left( \sup_{t\leq s} \left| n^{-1/\k} \sum_{i=1}^{\lfloor tn\rfloor} \left\{ \b_i \wedge \d n^{1/\k} - E_Q[\b_i \wedge \d n^{1/\k}] \right\} \right| \geq \eta/2 \right) \label{cutoff} \\
&\quad + Q\left( \d \sup_{t\leq s} \left| \sum_{i=1}^{\lfloor tn\rfloor} \ind{ \b_i > \d n^{1/\k}} - \fl{tn} Q(\b_i > \d n^{1/\k})  \right| \geq \eta/2 \right). \label{count}
\end{align}
Note that \eqref{btail} implies that $\fl{tn} Q(\b_i > \d n^{1/\k})
\ra t C_0 \d^{-\k}$ and, moreover,  that the convergence is uniform in $t \in [0,s]$. 
Therefore, to bound the term in \eqref{count} it is enough to show
that for all $0< s<\infty$ and $ \eta > 0$, 
\be\label{count2}
 \lim_{\d\ra 0} \limsup_{n\ra\infty} Q\left( \d \sup_{t\leq s} \left|
     \sum_{i=1}^{\lfloor tn\rfloor} \ind{ \b_i > \d n^{1/\k}} - t C_0
     \d^{-\k}  \right| \geq \eta \right) = 0.  
\ee
Now, for any $\d>0$ let $G_\d: \mathcal{M}_p((0,\infty]\times
[0,\infty)) \ra D_\infty^+$ (we equip the latter space with the $J_1$
topology) be defined by  
\[
 G_\d(\zeta)(t) = \zeta((\d,\infty] \times [0,t]), \ t\geq 0.
\]
Then $\sum_{i=1}^{\lfloor tn\rfloor} \ind{ \b_i > \d n^{1/\k}} =
G_\d(N_{1/n})(t)$. It is easy to see that $G_\d$ is continuous
on the set of point processes with no atoms on the line
$\{\d\} \times [0,\infty)$. Since $N_{\l,\k}$ belongs to this set
with probability 1, and $N_{1/n} \overset{Q}{\Lra} N_{\l,\k}$, the
continuous mapping theorem implies that $G_\d(N_{1/n})
\overset{Q}{\Lra} G_\d(N_{\l,\k})$. Furthermore, the supremum over a
compact interval is a continuous mapping from $D_\infty^+$ equipped
with the $J_1$ topology to the real line.  Therefore,
\begin{align*}
 \limsup_{n\ra\infty} Q\left( \d \sup_{t\leq s} \left| \sum_{i=1}^{tn} \ind{ \b_i > \d n^{1/\k}} - t C_0 \d^{-\k}  \right| \geq \eta \right) 
\leq Q\left( \d \sup_{t\leq s} \left| G_\d(N_{\l,\k})(t) - t C_0 \d^{-\k} \right| \geq \eta \right).
\end{align*}
Note that $G_\d(N_{\l,\k})$ is a homogeneous one-dimensional Poisson
process with rate $\l/\k \d^{-\k} = C_0 \d^{-\k}$. Therefore, using
once again the $L^p$-maximum inequality for
martingales, we have 
\[
 Q\left( \d \sup_{t\leq s} \left| G_\d(N_{\l,\k})(t) - t C_0 \d^{-\k} \right| \geq \eta \right) 
\leq \frac{\d^2}{\eta^2} \Var_Q\left( G_\d(N_{\l,\k})(s) \right) = \frac{ \l s \d^{2-\k}}{\eta^2 \k}. 
\]
Since $\k<2$ this last term vanishes as $\d\ra 0$ for any $\eta>0$ and
$s<\infty$. This completes the proof of \eqref{count2} and, therefore,
it only remains to estimate the term in  \eqref{cutoff}. 

To this end, we assume for (notational) simplicity that $s=1$, in
which case our task reduces to showing that for any $\eta>0$, 
\be \label{e:max.k}
 \lim_{\d\ra 0} \limsup_{n\ra\infty} Q\left( \max_{k\leq n} \left|
     n^{-1/\k} \sum_{i=1}^k \left\{ \b_i \wedge \d n^{1/\k} - E_Q[
       \b_1 \wedge \d n^{1/\k} ] \right\} \right| >  \eta \right) = 0.  
\ee
%We will use an adaptation of the proof of Etemadi's Lemma. 
%To this end, let 
%\[
% S_j^{(n,\d)} = n^{-1/\k} \sum_{i=1}^j \left\{ \b_i \wedge \d n^{1/\k} - E_Q[ \b_1 \wedge \d n^{1/\k} ] \right\}. 
%\]
For a fixed $n$ and $\d \in (0,1]$ denote 
\be\label{Smdef}
 S_j = n^{-1/\k} \sum_{i=1}^j \left\{ \b_i \wedge \d n^{1/\k} - E_Q[
   \b_1 \wedge \d n^{1/\k} ] \right\}, \quad  j=1,\ldots, n. 
\ee
For $\eta>0$ let $A_m = \{ \max_{j<m} | S_j | \leq  \eta  < |S_m| \}$. 
%\[
% A_m = \left\{ \max_{j<m} | S_j | \leq 4 \eta  < |S_m| \right\}.
%\]
Then, 
\begin{align}
%Q\left( \max_{k\leq n} |S_k| \geq 4 \eta \right) =
& Q\left( \max_{k\leq n} \left| n^{-1/\k} \sum_{i=1}^k \left\{ \b_i \wedge \d n^{1/\k} - E_Q[ \b_1 \wedge \d n^{1/\k} ] \right\} \right| > \eta \right)
 = \sum_{m=1}^n Q( A_m )  \nonumber \\ 
&\qquad \leq Q\left( |S_n| \geq  \eta/2 \right) + \sum_{m=1}^{n-1} Q\left( A_m \cap \{ | S_n | \leq \eta/2 \} \right) \nonumber \\
&\qquad \leq Q\left( |S_n| \geq \eta/2 \right) + \sum_{m=1}^{n-1}
Q\left( A_m \cap \{ | S_n - S_m | > \eta/2 \} \right).  \label{Etemadi1}
\end{align}
It was shown in the proof of Lemma 5.5 in \cite{psWQLTn} that for some
constant $C$, 
\be\label{VarSumbound}
 \Var_Q\left( \sum_{i=1}^n \b_i \wedge \d n^{1/\k} \right) \leq C \d^{2-\k} n^{2/\k}. 
\ee
By Markov's inequality this shows that the term $Q\left( |S_n| \geq
  \eta/2 \right)$ does not contribute to the limit in
\eqref{e:max.k}. Therefore, it remains only to bound the sum on the
right in \eqref{Etemadi1}.  
%\[
% Q\left( |S_n| \geq 2 \eta \right) \leq 
%%%\frac{1}{4 \eta^2} E_Q|S_n|^2 = 
%\frac{n^{-2/\k}}{4 \eta^2} \Var_Q\left(\sum_{i=1}^n \b_i \wedge \d n^{1/\k} \right) \leq \frac{C}{4 \eta^2} \d^{2-\k}. 
%\]
If the $\b_i$ were independent, then the general term in this sum
would be equal to $Q(A_m)Q(|S_n-S_m|> \eta/2)$ and the sum could be
handled in the same way as the term $Q\left( |S_n| \geq
  \eta/2 \right)$ above. While the $\b_i$ are not independent under $Q$, they have good
mixing properties and the following lemma gives an upper bound on the
general term in the sum, not far off from what it would be if  the
$\b_i$ were independent.   
%The main technical estimate is provided by the following lemma, whose proof we defer for now. 
\begin{lem}\label{QamSnSm}
There are constants $C,C'>0$ such that for any 
$n=1,2,\ldots$, $\d \in (0,1]$, $m=1,\ldots, n$ and $\eta>0$, 
\[
 Q( A_m \cap \{ |S_n - S_m| > \eta \} )
\leq C e^{-C' \eta n^{1/\k}/\log n} + \frac{1}{\eta^2} C \d^{2-\k}
\left( Q(A_m) + 1/n \right). 
\]
\end{lem}
Assuming the statement of Lemma \ref{QamSnSm}, the proof of Lemma
\ref{smallblem} can be completed by writing (changing the constants as
necessary) 
 \begin{align*}
 \sum_{m=1}^{n-1} Q( A_m \cap \{ |S_n - S_m| > \eta/2 \} ) 
& \leq \sum_{m=1}^{n-1} \left\{ C e^{-C' \eta n^{1/\k}/\log n} + \frac{1}{\eta^2} C \d^{2-\k} \left( Q(A_m) + 1/n \right) \right\} \\
& \leq C n e^{-C' \eta n^{1/\k}/\log n} + \frac{C \d^{\k-2} }{\eta^2}.
\end{align*}
Both terms vanish under the limits in \eqref{e:max.k}, so we only need
to prove Lemma \ref{QamSnSm}. 
\end{proof}

\begin{proof}[Proof of Lemma \ref{QamSnSm}]
Define a (discrete time) filtration on $\Omega = [0,1]^\Z$ by
$\mathcal{G}_n= \s( \w_i: \, i \leq n)$, $n=0,1,2,\ldots$. Then for
each $m=0,1,2,\ldots$, $\nu_m-1$ is a stopping time with respect to that
filtration, and we denote $\mathcal{F}_m =\mathcal{G}_{ \nu_m-1}$,
$m=1,2,\ldots$. Since each $\b_j$ with $j\leq m$ is
$\mathcal{F}_m$-measurable, so is each $S_j$ with $j\leq
m$. Therefore, 
\be\label{CondFm}
 Q\left( A_m \cap \{ | S_n - S_m | > \eta \} \right) = E_Q\left[
   \ind{A_m} Q\left( |S_n - S_m| > \eta \, \bigl| \, \mathcal{F}_m
   \right) \right] . 
\ee
Conditioned on $\mathcal{F}_m$, the difference $S_n-S_m$  no longer
has zero mean, but we will show that the conditional mean is typically
small. We begin by comparing the conditional and unconditional means
of $\b_j \wedge \d n^{1/\k}$. To this end we make explicit the
dependence of $\b_j$ on $\mathcal{F}_m$.  
Recall the definitions of $\Pi_{i,j}$, $W_{i,j}$ and $R_{i,j}$ in
\eqref{PiRWdef} and note that $W_i = W_{k,i} + \Pi_{k,i} W_{k-1}$ for
any $k\leq i$.  Therefore, for any $1\leq m<j$ we can rewrite\eqref{bform} as
\begin{align*}
 \b_j &= \nu_{j} - \nu_{j-1} + 2 \sum_{i=\nu_{j-1}}^{\nu_{j}-1} \left( W_{\nu_m,i} +  W_{\nu_m - 1} \Pi_{\nu_m,i} \right) \\
&= \nu_{j} - \nu_{j-1} + 2 \sum_{i=\nu_{j-1}}^{\nu_{j}-1} W_{\nu_m,i} + 2 W_{\nu_m - 1} \Pi_{\nu_m, \nu_{j-1}-1} R_{\nu_{j-1}, \nu_j -1} \\
&=: \b_{m,j} + 2 W_{\nu_m - 1} \Pi_{\nu_m, \nu_{j-1}-1} R_{\nu_{j-1}, \nu_j -1}.
\end{align*}
Note that $\b_{m,j}$ is independent of $\mathcal{F}_m$. We enlarge, if
necessary, the probability space to define a random variable 
$\tilde{W}$ with the same distribution as $W_{\nu_m-1}$ and
independent of all $(\omega_x)$; in particular, $\tilde{W}$ is
independent of $\mathcal{F}_m$. Denote 
$\tilde\b_j = \b_{m,j} + 2 \tilde{W}\Pi_{\nu_m, \nu_{j-1}-1}
R_{\nu_{j-1}, \nu_j -1}$, so that   
\[
 E_Q\left[ \b_j \wedge \d n^{1/\k}  \, \bigl| \, \mathcal{F}_m \right] - E_Q[\b_j \wedge \d n^{1/\k} ]
=   E_Q\left[  \b_j \wedge \d n^{1/\k} - \tilde\b_j \wedge \d n^{1/\k}
  \, \bigl| \, \mathcal{F}_m \right]. 
\]
Observe that $R_{\nu_{j-1},\nu_j -1} \leq \b_{m,j} \leq
\min(\b_j,\tilde{\b}_j)$. Thus, if $R_{\nu_{j-1},\nu_j - 1} \geq \d
n^{1/\k}$, then both $\b_j$ and $\tilde\b_j$ are larger than $\d
n^{1/\k}$ as well. This implies that 
\begin{align}
&\left| E_Q\left[ \b_j \wedge \d n^{1/\k}  \, \bigl| \, \mathcal{F}_m \right] - E_Q[\b_j \wedge \d n^{1/\k} ] \right| \\
&\quad \leq E_Q\left[ | \b_j - \tilde\b_j | \ind{R_{\nu_{j-1}, \nu_j -1} \leq \d n^{1/\k}} \, \bigl| \, \mathcal{F}_m \right] \nonumber \\
&\quad = E_Q\left[ 2 \Pi_{\nu_m, \nu_{j-1}-1} R_{\nu_{j-1}, \nu_j -1} |W_{\nu_m-1} - \tilde{W}| \ind{R_{\nu_{j-1}, \nu_j -1} \leq \d n^{1/\k}} \, \bigl|  \, \mathcal{F}_m \right] \nonumber \\
&\quad \leq  2 \bigl( E_Q[\Pi_{0,\nu-1}]\bigr)^{j-m-1}
E_Q[R_{0,\nu-1}\ind{R_{0,\nu -1} \leq \d n^{1/\k}}] \left(
  E_Q[\tilde{W}] + W_{\nu_m-1} \right), \nonumber 
%&\quad \leq C r^{j-m-1} \left( 1 + W_{\nu_m-1} \right), \nonumber
\end{align}
where in the last inequality we used the fact that the blocks of
environment between ladder locations are i.i.d.\ under the measure
$Q$. 
Since $R_{0,\nu-1} \leq \b_1$, there exists a constant $C$ so that
$Q(R_{0,\nu-1} > x) \leq C x^{-\k}$. This implies that $E_Q[
R_{0,\nu-1}  \ind{R_{0,\nu-1} \leq \d n^{1/\k}} ] \leq E_Q[
R_{0,\nu-1} ] < \infty$ when $\k > 1$ and $E_Q[R_{0,\nu-1}
\ind{R_{0,\nu-1} \leq \d n^{1/\k}} ] \leq C \log n$ when
$\k=1$ for some other $C$. Thus, we can always bound this expectation
by $C \log n$ for 
some $C>0$. Also, the definition of $\nu$ implies that $r:= E_Q[\Pi_{0,\nu-1}]
< 1$ and \eqref{Wtail} implies that $E_Q[\tilde{W}] = E_Q[W_{-1}] <
\infty$. Thus, there exists a constant $C$ so that 
\[
 \left| E_Q\left[ \b_j \wedge \d n^{1/\k}  \, \bigl| \, \mathcal{F}_m \right] - E_Q[\b_j \wedge \d n^{1/\k} ] \right|
\leq C  \log n \, r^{j-m-1} \left( 1 + W_{\nu_m-1} \right), 
\]
implying that
\begin{align*}
 \left| E_Q[S_n - S_m \, | \, \mathcal{F}_m ] \right| 
&\leq n^{-1/\k} \sum_{j=m+1}^n \left| E_Q\left[ \b_j \wedge \d n^{1/\k}  \, \bigl| \, \mathcal{F}_m \right] - E_Q[\b_j \wedge \d n^{1/\k} ] \right| \\
%&\leq n^{-1/\k} \sum_{j=m+1}^n C r^{j-m-1} \left( 1 + W_{\nu_m-1} \right) \\
&\leq C n^{-1/\k} \log n \left( 1 + W_{\nu_m-1} \right) .
\end{align*}

Applying Chebyshev's inequality conditionally, we obtain 
\begin{align}
& Q\left(| S_n - S_m| > \eta \, | \, \mathcal{F}_m \right) \nonumber \\
&\qquad \leq \indd{ \left| E_Q[S_n - S_m \, | \, \mathcal{F}_m ] \right| > \eta/2 } + Q\left(| S_n - S_m - E_Q[S_n - S_m \, | \, \mathcal{F}_m ] | > \eta/2 \right) \nonumber \\
&\qquad \leq \indd{ 1+W_{\nu_m-1}  > n^{1/\k} \eta/( 2C \log n) } +
\frac{4}{\eta^2}\Var_Q\left( S_n - S_m \, | \, \mathcal{F}_m \right).  \label{conCheby}
\end{align}
To handle the conditional variance in \eqref{conCheby} we write 
\begin{align}
 \Var_Q\left( S_n - S_m \, | \, \mathcal{F}_m \right) 
%= n^{-2/\k} \Var_Q\left( \sum_{j=m+1}^n \b_j \wedge \d n^{1/\k} \, | \, \mathcal{F}_m \right) \nonumber \\
& = n^{-2/\k} \sum_{j=m+1}^n \Var_Q\left( \b_j \wedge \d n^{1/\k} \, | \, \mathcal{F}_m \right) \nonumber \\
&\qquad + 2n^{-2/\k} \sum_{m<j<k\leq n} \Cov_Q\left( \b_j \wedge \d n^{1/\k} , \, \b_k \wedge \d n^{1/\k} \, | \, \mathcal{F}_m \right). \label{condVarExpand}
\end{align}
Upper bounds on the conditional variance and conditional covariance
terms above can be obtained in a similar manner to the proof of (49)
in \cite{psWQLTn}. One adapts this approach to take into account the
conditioning on $\mathcal{F}_m$, by replacing $\b_j$ by $\b_{m,j}$,  and
then controlling the difference between the two similarly to what was
done above when bounding  $E[S_n-S_m \,|\, \mathcal{F}_m]$. Doing
this we obtain that 
there exist constants $C>0$ and $r\in(0,1)$ such that 
\[
 \Var_Q\left( \b_j \wedge \d n^{1/\k} \, | \, \mathcal{F}_m \right) \leq C \d^{2-\k} n^{2/\k-1}\left( 1 + r^{j-m-1} W_{\nu_m-1}^2 \right)
\]
and 
\[
 \Cov_Q\left( \b_j \wedge \d n^{1/\k} , \, \b_k \wedge \d n^{1/\k} \, | \, \mathcal{F}_m \right) \leq C \d^{2-\k} n^{2/\k-1} \left( 1 + r^{j-m-1} W_{\nu_m-1}^2 \right) \sqrt{ r^{k-j-1} }. 
\]
Using these bounds in  \eqref{condVarExpand},  we see that for some $C>0$, 
\be\label{condVarBound}
 \Var_Q( S_n - S_m \, | \, \mathcal{F}_m ) \leq C \d^{2-\k}\left( 1 + \frac{W_{\nu_m-1}^2}{n} \right).
\ee

Combining \eqref{CondFm}, \eqref{conCheby}, and \eqref{condVarBound} we obtain
\begin{align*}
& Q( A_m \cap \{ |S_n - S_m| > \eta \} )\\
&\quad \leq Q( C'(1 + W_{\nu_m-1})  > \eta n^{1/\k}/\log n  ) + \frac{1}{\eta^2} C \d^{2-\k} E_Q\left[ \ind{A_m} \left( 1 + W_{\nu_m-1}^2/n \right) \right] \\
&\quad \leq C e^{-C' \eta n^{1/\k}/\log n} + \frac{1}{\eta^2} C \d^{2-\k} \left( Q(A_m) + 1/n \right)
\end{align*}
where the constants $C,C'$ may change from line to line (in the last
inequality we used \eqref{Wtail} and the fact that $W_{\nu_m-1}$ has
the same distribution as $W_{-1}$ under $Q$). This gives us the
statement of the lemma.   
\end{proof}

\textbf{Acknowledgment:} We would like to thank Olivier Zindy for pointing out that the convergence in distribution in Corollary \ref{AveragedXn} could be improved to the uniform topology. 

\bibliographystyle{alpha}
\bibliography{RWRE}

\def\cprime{$'$}
\begin{thebibliography}{ESTZ10}

\bibitem[Bil99]{bCOPM}
Patrick Billingsley.
\newblock {\em {Convergence of probability measures}}.
\newblock {Wiley Series in Probability and Statistics: Probability and
  Statistics}. John Wiley \& Sons Inc., New York, second edition, 1999.
\newblock A Wiley-Interscience Publication.

\bibitem[DG12]{dgWQLTn}
D.~Dolgopyat and I.~Goldsheid.
\newblock {Quenched Limit Theorems for Nearest Neighbour Random Walks in 1D
  Random Environment}.
\newblock {\em Communications in Mathematical Physics}, 315:241--277, 2012.

\bibitem[ESTZ10]{estzWQLTn}
Nathana{\"e}l Enriquez, Christophe Sabot, Laurent Tournier, and Olivier Zindy.
\newblock {Annealed and quenched fluctuations for ballistic random walks in
  random environment on $\mathbb{Z}$}.
\newblock Preprint, available at arXiv:1012.1959, 2010.

\bibitem[ESZ09]{eszStable}
Nathana{\"e}l Enriquez, Christophe Sabot, and Olivier Zindy.
\newblock {Limit laws for transient random walks in random environment on
  {$\mathbb{Z}$}}.
\newblock {\em Ann. Inst. Fourier (Grenoble)}, 59(6):2469--2508, 2009.

\bibitem[GS02]{gsMVSS}
Nina Gantert and Zhan Shi.
\newblock {Many visits to a single site by a transient random walk in random
  environment}.
\newblock {\em Stochastic Process. Appl.}, 99(2):159--176, 2002.

\bibitem[Kal74]{kRPWOD2}
Olav Kallenberg.
\newblock {Series of random processes without discontinuities of the second
  kind}.
\newblock {\em Ann. Probability}, 2:729--737, 1974.

\bibitem[KKS75]{kksStable}
H.~Kesten, M.~V. Kozlov, and F.~Spitzer.
\newblock {A limit law for random walk in a random environment}.
\newblock {\em Compositio Math.}, 30:145--168, 1975.

\bibitem[Pet09]{p1LSL2}
Jonathon Peterson.
\newblock {Quenched limits for transient, ballistic, sub-{G}aussian
  one-dimensional random walk in random environment}.
\newblock {\em Ann. Inst. Henri Poincar{\'e} Probab. Stat.}, 45(3):685--709,
  2009.

\bibitem[PS10]{psWQLTn}
Jonathon Peterson and Gennady Samorodnitsky.
\newblock {Weak quenched limiting distributions for transient one-dimensional
  random walk in a random environment}.
\newblock Preprint, available at arXiv:1011.6366v3, 2010.

\bibitem[PZ09]{pzSL1}
Jonathon Peterson and Ofer Zeitouni.
\newblock {Quenched limits for transient, zero speed one-dimensional random
  walk in random environment}.
\newblock {\em Ann. Probab.}, 37(1):143--188, 2009.

\bibitem[Res08]{rEVRVPP}
Sidney~I. Resnick.
\newblock {\em {Extreme values, regular variation and point processes}}.
\newblock {Springer Series in Operations Research and Financial Engineering}.
  Springer, New York, 2008.
\newblock Reprint of the 1987 original.

\bibitem[Sin83]{sRRWRE}
Ya.~G. Sinai.
\newblock {The Limiting Behavior of a One-Dimensional Random Walk in a Random
  Medium}.
\newblock {\em Theory Probab. Appl.}, 27(2):256--268, 1983.

\bibitem[Sol75]{sRWRE}
Fred Solomon.
\newblock {Random walks in a random environment}.
\newblock {\em Ann. Probability}, 3:1--31, 1975.

\bibitem[ST94]{stStable}
G.~Samorodnitsky and M.S. Taqqu.
\newblock {\em {{S}table {N}on-{G}aussian {R}andom {P}rocesses}}.
\newblock Chapman and Hall, New York, 1994.

\bibitem[ST06]{stRPD}
Tokuzo Shiga and Hiroshi Tanaka.
\newblock {Infinitely divisible random probability distributions with an
  application to a random motion in a random environment}.
\newblock {\em Electron. J. Probab.}, 11:no. 44, 1144--1183 (electronic), 2006.

\bibitem[Whi02]{wSPL}
Ward Whitt.
\newblock {\em {Stochastic-process limits}}.
\newblock {Springer Series in Operations Research}. Springer-Verlag, New York,
  2002.
\newblock An introduction to stochastic-process limits and their application to
  queues.

\bibitem[Zei04]{zRWRE}
Ofer Zeitouni.
\newblock {Random walks in random environment}.
\newblock In {\em {Lectures on probability theory and statistics}}, volume 1837
  of {\em {Lecture Notes in Math.}}, pages 189--312. Springer, Berlin, 2004.

\end{thebibliography}

\end{document}